\documentclass[12pt,a4paper]{amsart}
\usepackage{amssymb}
\usepackage[T1]{fontenc}
\usepackage[a4paper,left=2.25cm,right=2.25cm,top=3cm,bottom=3cm,headsep=1cm]{geometry}
\usepackage{hyperref}
\usepackage{cleveref}
\usepackage{enumerate}
\newcommand\rem[2][]{}
\usepackage{tikz}\usetikzlibrary{graphs,quotes,fit,positioning,matrix,calc,decorations.markings,angles,decorations.pathmorphing,decorations.pathreplacing,svg.path}
\usepackage{tikz-cd}
\tikzset{arrow/.style={postaction={decorate,thick,decoration={markings,mark = at position #1 with {\arrow{>}}}}},arrow/.default=0.5}
\tikzset{invarrow/.style={postaction={decorate,thick,decoration={markings,mark = at position #1 with {\arrow{<}}}}},invarrow/.default=0.5}
\usepackage{hyperref}
\allowdisplaybreaks[1]
\newcommand{\bra}[1]{\left\langle #1\right|}
\newcommand{\ket}[1]{\left|#1\right\rangle}

\renewcommand\ss{\scriptstyle}

\newcommand\CC{{\mathbb C}}
\newcommand\PP{{\mathbb P}}

\newcommand\ZZ{{\mathbb Z}}
\newcommand\wt{\mathrm{wt}}
\newcommand\mdeg{\operatorname{mdeg}}
\newcommand\codim{\operatorname{codim}}
\newcommand\spec{\operatorname{Spec}}
\newcommand\proj{\operatorname{Proj}}

\long\def\junk#1{}
\title[Shuffle algebras, lattice paths
and the commuting scheme]{Shuffle algebras, lattice paths\\
and the commuting scheme}
\address{A.~Garbali, P.~Zinn-Justin, School of Mathematics and Statistics, University of Melbourne, Victoria 3010, Australia.}
\author{Alexandr Garbali}
\email{alexandr.garbali@unimelb.edu.au}
\author{Paul Zinn-Justin}
\email{pzinn@unimelb.edu.au}
\thanks{We thank Allen Knutson, Peter MacNamara, Andrei Negut, Gufang Zhao for discussions.}
\date{\today}
\newtheorem{thm}{Theorem}
\newtheorem{conj}{Conjecture}
\newtheorem{ex}{Example}
\newtheorem{cor}{Corollary}
\newtheorem{prop}{Proposition}
\newtheorem{rmk}{Remark}
\newtheorem{lem}{Lemma}

\newcommand\F{\mathfrak F}
\newcommand\g{\mathfrak{gl}}
\renewcommand\H{{\mathcal H}}
\newcommand\Hn{{\mathcal H}_n}
\newcommand\FF{{\mathbb F}}
\newcommand\QQ{{\mathbb Q}}
\newcommand\Z{{\mathrm Z}}
\newcommand\A{{\mathrm A}}

\tikzset{arrow/.style={postaction={decorate,thick,decoration={markings,mark = at position #1 with {\arrow{>}}}}},arrow/.default=0.5}
\tikzset{invarrow/.style={postaction={decorate,thick,decoration={markings,mark = at position #1 with {\arrow{<}}}}},invarrow/.default=0.5}
\begin{document}
    \begin{abstract}
    The commutative trigonometric shuffle algebra $\A$ is a space of symmetric rational functions satisfying certain wheel conditions \cite{FO,FHHSY}. We describe a ring isomorphism between $\A$ and the center of the Hecke algebra
    using a realization of the elements of $\A$ as partition functions of coloured lattice paths associated to the $R$-matrix of $\mathcal U_{t^{1/2}}(\widehat{gl}_{\infty})$. As an application, we compute under certain conditions the Hilbert series of the {\em commuting scheme}\/ and identify it with a particular element of the shuffle algebra $\A$, thus providing a combinatorial formula for it as a ``domain wall'' type partition function of coloured lattice paths.
    \end{abstract}

\maketitle

\begin{flushright}\textit{
Dedicated to Jasper Stokman on the occasion of his 50th birthday
}
\end{flushright}
\section{Introduction}\label{sec:intro}
\subsection{Center of Hecke and shuffle products}\label{sec:intro-center}
Consider the {\em Hecke algebra} $\Hn$ ($n\in\ZZ_{\ge0}$), which we define for our purposes to be the algebra (over a field extension $\FF$ of $\QQ(t)$ to be specified below) with generators
$T_1,\ldots,T_{n-1}$ and relations
\[
T_iT_{i+1}T_i = T_{i+1}T_i T_{i+1}\qquad T_i^2=(t-1)T_i+t\qquad T_iT_j=T_jT_i\ (|i-j|\ge2)
\]

The standard basis $(T_w)_{w\in \mathcal S_{n}}$ of $\Hn$ is obtained by writing $T_w:=T_{i_1}\ldots T_{i_k}$
if $w=s_{i_1}\ldots s_{i_k}$ is a reduced word in the elementary transpositions $s_i$. Denote by $|w|=k$ the length of $w$.

Combine the centers $\Z(\Hn)$ into a graded algebra
\[
\Z := \bigoplus_{n\ge 0} \Z(\Hn)
\]
by defining a shuffle product
\begin{align}\label{eq:defZstar}
*: \Z(\H_k)\otimes \Z(\H_\ell) &\to \Z(\H_{k+\ell})
\\\notag
x &\mapsto \sum_{w\in \mathcal{S}^{k,\ell}}
t^{-|w|} T_{w} x T_{w^{-1}}
\end{align}
where $\H_k\otimes \H_\ell$ is embedded inside  $\H_{k+\ell}$ in the obvious way, and $\mathcal{S}^{k,\ell}$ is the set of shortest representatives of
cosets in $\mathcal{S}_{k+\ell}/ (\mathcal{S}_k\times\mathcal{S}_\ell)$.

Let us also consider $\Lambda$, the algebra of {\em symmetric functions}\/ with coefficients in $\FF$. One can for example define it as a polynomial ring in countably many variables: $\Lambda:=\FF[p_1,p_2,\ldots]$ where $\deg p_r=r$.

There is a graded algebra isomorphism  $\Phi$ between $\Z$ and $\Lambda$ \cite{Wan-Frob}, which will be reviewed in \S \ref{sec:hecke};
it is a deformation of the classical Frobenius map in the context of
the symmetric group.

In what follows we fix the base field to be
\[
\FF = \QQ(q,t)
\]
Define $\FF[x_1^\pm,\ldots,x_n^\pm]_0$ to be the space of Laurent polynomials with specific degree bounds:
\begin{equation}\label{eq:degbounds}
\FF[x_1^\pm,\ldots,x_n^\pm]_0 = 
\FF\text{-span of } \left\{
x_1^{i_1}\ldots x_n^{i_n}
,\ 
i_1,\ldots,i_n\in\ZZ:\ \left|\sum_{j=1}^r i_j\right|\le r(n-r),\ 0\le r\le n
\right\}
\end{equation}
(in particular, these are homogeneous polynomials of degree $0$)
and
$\FF[x_1^\pm,\ldots,x_n^\pm]^{\mathcal S_n}_0$ to be its subspace of symmetric Laurent polynomials.

Consider\footnote{$\A_n$ is not quite the shuffle algebra as normally defined, but is more convenient for us; the standard shuffle algebra, denoted by $\mathcal A_n$, will be given in \S\ref{sec:shuffle}.}
\[
\A_n = \left\{ P \in \FF[x_1^\pm,\ldots,x_n^\pm]^{\mathcal S_n}_0:
\ P(x,qx,tx,x_4,\ldots,x_n)=P(tqx,qx,tx,x_4,\ldots,x_n)=0\right\}
\]
One can again introduce a shuffle product
\begin{align*}
*: \A_k\otimes \A_\ell &\to \A_{k+\ell}
\\
P\otimes Q &\mapsto 
\sum_{w\in \mathcal{S}^{k,\ell}}
\left[P(x_{w(1)},\ldots,x_{w(k)})Q(x_{w(k+1)},\ldots,x_{w(k+\ell)})
\prod_{\substack{1\le i\le k\\k+1\le j\le k+\ell}}
\tilde\omega(x_{w(i)}/x_{w(j)})
\right]
\end{align*}
where 
$\tilde\omega(x)=\frac{(1-q^{-1}x)(1-qt^{-1}x)(1-t x)}{-x(1-x)}$
making
\[
\A=\bigoplus_{n\ge 0} \A_n
\]
a graded algebra. There is another graded algebra isomorphism $\Upsilon$ from $\A$ to $\Lambda$ \cite{FHHSY}, which will be reviewed in \S \ref{sec:macdo}.

It is natural to ask if there is an interpretation of the isomorphism $\Upsilon^{-1}\circ \Phi$ from $\Z$ to $\A$.
Note however that $\Lambda$ possesses graded algebra automorphisms, and in particular 
we shall allow ourselves to work modulo the automorphisms that rescale the variables $p_r$. 

\subsection{Lattice paths}\label{sec:intro-lattice}
\newcommand\conn{\text{conn}}
Consider a $n\times n$ square grid, and a set of $n$ paths entering from the left external edges of the grid and exiting at the top. 
We label incoming paths on the left $1,\ldots,n$ from top to bottom.
We say that a set of paths $P$ has connectivity $\conn(P)=w\in\mathcal S_n$
if the path labelled $i$ (i.e., the $i^{\rm th}$ incoming path on the left counted from the top) 
exits at the $w^{-1}(i)^{\rm th}$ location at the top (counted from the left). 

\begin{ex}Here are all the paths with connectivity $2431$:
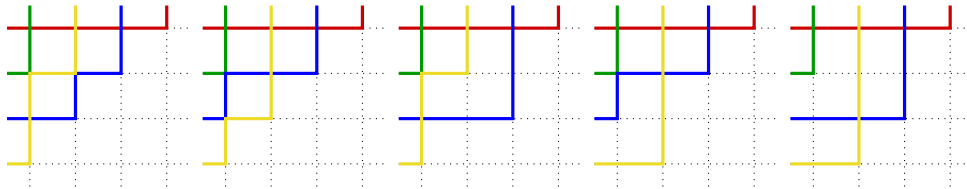
\begin{center}
\begin{tikzpicture}[scale=0.6]\draw[dotted] (0.5,0.5) grid (4.5,-3.5);\begin{scope}[very thick,shift={(0.5,0.5)},scale=1cm]\draw[red!80!black] svg {M 0 -.5 L 1 -.5 L 2 -.5 L 3 -.5 L 3.5 -.5 3.5 0}; \draw[green!60!black] svg {M 0 -1.5 L .5 -1.5 .5 -1 L .5 0}; \draw[blue] svg {M 0 -2.5 L 1 -2.5 L 1.5 -2.5 1.5 -2 L 1.5 -1.5 2 -1.5 L 2.5 -1.5 2.5 -1 L 2.5 0}; \draw[yellow!90!black] svg {M 0 -3.5 L .5 -3.5 .5 -3 L .5 -2 L .5 -1.5 1 -1.5 L 1.5 -1.5 1.5 -1 L 1.5 0}; \end{scope}\end{tikzpicture} 
\begin{tikzpicture}[scale=0.6]\draw[dotted] (0.5,0.5) grid (4.5,-3.5);\begin{scope}[very thick,shift={(0.5,0.5)},scale=1cm]\draw[red!80!black] svg {M 0 -.5 L 1 -.5 L 2 -.5 L 3 -.5 L 3.5 -.5 3.5 0}; \draw[green!60!black] svg {M 0 -1.5 L .5 -1.5 .5 -1 L .5 0}; \draw[blue] svg {M 0 -2.5 L .5 -2.5 .5 -2 L .5 -1.5 1 -1.5 L 2 -1.5 L 2.5 -1.5 2.5 -1 L 2.5 0}; \draw[yellow!90!black] svg {M 0 -3.5 L .5 -3.5 .5 -3 L .5 -2.5 1 -2.5 L 1.5 -2.5 1.5 -2 L 1.5 -1 L 1.5 0}; \end{scope}\end{tikzpicture} 
\begin{tikzpicture}[scale=0.6]\draw[dotted] (0.5,0.5) grid (4.5,-3.5);\begin{scope}[very thick,shift={(0.5,0.5)},scale=1cm]\draw[red!80!black] svg {M 0 -.5 L 1 -.5 L 2 -.5 L 3 -.5 L 3.5 -.5 3.5 0}; \draw[green!60!black] svg {M 0 -1.5 L .5 -1.5 .5 -1 L .5 0}; \draw[blue] svg {M 0 -2.5 L 1 -2.5 L 2 -2.5 L 2.5 -2.5 2.5 -2 L 2.5 -1 L 2.5 0}; \draw[yellow!90!black] svg {M 0 -3.5 L .5 -3.5 .5 -3 L .5 -2 L .5 -1.5 1 -1.5 L 1.5 -1.5 1.5 -1 L 1.5 0}; \end{scope}\end{tikzpicture} 
\begin{tikzpicture}[scale=0.6]\draw[dotted] (0.5,0.5) grid (4.5,-3.5);\begin{scope}[very thick,shift={(0.5,0.5)},scale=1cm]\draw[red!80!black] svg {M 0 -.5 L 1 -.5 L 2 -.5 L 3 -.5 L 3.5 -.5 3.5 0}; \draw[green!60!black] svg {M 0 -1.5 L .5 -1.5 .5 -1 L .5 0}; \draw[blue] svg {M 0 -2.5 L .5 -2.5 .5 -2 L .5 -1.5 1 -1.5 L 2 -1.5 L 2.5 -1.5 2.5 -1 L 2.5 0}; \draw[yellow!90!black] svg {M 0 -3.5 L 1 -3.5 L 1.5 -3.5 1.5 -3 L 1.5 -2 L 1.5 -1 L 1.5 0}; \end{scope}\end{tikzpicture} 
\begin{tikzpicture}[scale=0.6]\draw[dotted] (0.5,0.5) grid (4.5,-3.5);\begin{scope}[very thick,shift={(0.5,0.5)},scale=1cm]\draw[red!80!black] svg {M 0 -.5 L 1 -.5 L 2 -.5 L 3 -.5 L 3.5 -.5 3.5 0}; \draw[green!60!black] svg {M 0 -1.5 L .5 -1.5 .5 -1 L .5 0}; \draw[blue] svg {M 0 -2.5 L 1 -2.5 L 2 -2.5 L 2.5 -2.5 2.5 -2 L 2.5 -1 L 2.5 0}; \draw[yellow!90!black] svg {M 0 -3.5 L 1 -3.5 L 1.5 -3.5 1.5 -3 L 1.5 -2 L 1.5 -1 L 1.5 0}; \end{scope}\end{tikzpicture} 
\end{center}
\end{ex}

To each such set of lattice paths $P$ we associate a (Boltzmann) weight defined as the product over each vertex of the square grid
of a factor that depends on the configuration around that vertex:
\begin{equation}\label{eq:wt}
\wt(P)=\prod_{i,j=1}^n
\begin{cases}
1-t&
\begin{tikzpicture}[baseline=-3pt]
\draw[dotted] (-0.5,0) -- (0.5,0) (0,-0.5) -- (0,0.5);
\useasboundingbox (0,-0.7) -- (0,0.7);
\draw[very thick,red!80!black] (-0.5,0) -- (0,0) -- (0,0.5);
\end{tikzpicture}
,
\begin{tikzpicture}[baseline=-3pt]
\draw[dotted] (-0.5,0) -- (0.5,0) (0,-0.5) -- (0,0.5);
\useasboundingbox (0,-0.7) -- (0,0.7);
\draw[very thick,red!80!black] (-0.5,0) -- (0,0) -- (0,0.5);
\draw[very thick,green!60!black] (0.5,0) -- (0,0) -- (0,-0.5);
\end{tikzpicture}
\\
q^{-1}x_ix_j^{-1}(1-t)&
\begin{tikzpicture}[baseline=-3pt]
\draw[dotted] (-0.5,0) -- (0.5,0) (0,-0.5) -- (0,0.5);
\useasboundingbox (0,-0.7) -- (0,0.7);
\draw[very thick,red!80!black] (0.5,0) -- (0,0) -- (0,-0.5);
\end{tikzpicture}
,
\begin{tikzpicture}[baseline=-3pt]
\draw[dotted] (-0.5,0) -- (0.5,0) (0,-0.5) -- (0,0.5);
\useasboundingbox (0,-0.7) -- (0,0.7);
\draw[very thick,green!60!black] (-0.5,0) -- (0,0) -- (0,0.5);
\draw[very thick,red!80!black] (0.5,0) -- (0,0) -- (0,-0.5);
\end{tikzpicture}
\\
t(1-q^{-1} x_i x_j^{-1})&
\begin{tikzpicture}[baseline=-3pt]
\draw[dotted] (-0.5,0) -- (0.5,0) (0,-0.5) -- (0,0.5);
\useasboundingbox (0,-0.7) -- (0,0.7);
\draw[very thick,green!60!black] (0,-0.5) -- (0,0.5);
\draw[very thick,red!80!black] (-0.5,0) -- (0.5,0);
\end{tikzpicture}
,
\begin{tikzpicture}[baseline=-3pt]
\draw[dotted] (-0.5,0) -- (0.5,0) (0,-0.5) -- (0,0.5);
\useasboundingbox (0,-0.7) -- (0,0.7);
\draw[very thick,red!80!black] (-0.5,0) -- (0.5,0);
\end{tikzpicture}
\\
1-q^{-1} x_i x_j^{-1}&
\begin{tikzpicture}[baseline=-3pt]
\draw[dotted] (-0.5,0) -- (0.5,0) (0,-0.5) -- (0,0.5);
\useasboundingbox (0,-0.7) -- (0,0.7);
\draw[very thick,red!80!black] (0,-0.5) -- (0,0.5);
\draw[very thick,green!60!black] (-0.5,0) -- (0.5,0);
\end{tikzpicture}
,
\begin{tikzpicture}[baseline=-3pt]
\draw[dotted] (-0.5,0) -- (0.5,0) (0,-0.5) -- (0,0.5);
\useasboundingbox (0,-0.7) -- (0,0.7);
\draw[very thick,red!80!black] (0,-0.5) -- (0,0.5);
\end{tikzpicture}
\\
1-t q^{-1} x_ix_j^{-1}&
\begin{tikzpicture}[baseline=-3pt]
\draw[dotted] (-0.5,0) -- (0.5,0) (0,-0.5) -- (0,0.5);
\useasboundingbox (0,-0.7) -- (0,0.7);
\end{tikzpicture}
\end{cases}
\end{equation}
where the convention is that colours can be substituted as long as one preserves the ordering;
that is, if two paths are present at a vertex, red (resp.\ green) stands for the one with the lower (resp.\ greater) label.

To any element $c = \sum_{v\in\mathcal S_n} c_v T_v\in \Z(\H_n)$ ($c_v\in\FF$) of the center of the Hecke algebra we associate
the {\em partition function}
\begin{equation}\label{eq:deff}
f(c) := \alpha_n
\sum_{\substack{\text{lattice paths }P\\\text{on the }n\times n\text{ grid}}} \wt(P) c_{\conn(P)}
\end{equation}
where $\alpha_n=(q/t)^{\frac{n(n-1)}{2}}(1-t)^{-n}$ is a constant which we introduce for convenience. 

We can now formulate our first theorem:
\begin{thm}\label{thm:square}
There is a commuting square of graded algebra morphisms
\begin{center}
\begin{tikzcd}
\Z \arrow[r, "\Phi"] \arrow[d,"f"]
& \Lambda \arrow[d, "\sigma_{q^{-1}}\sigma_t^{-1}"] \\
\A \arrow[r, "\Upsilon"]
& \Lambda
\end{tikzcd}
\end{center}
where $\sigma_{q^{-1}}\sigma_t^{-1}$ is the automorphism that sends $p_r$ to $(-1)^{r-1}\frac{\ss (1-q)^r(1-t^r) }{\ss (1-q^r)(1-t)^r}p_r$.
\end{thm}

\subsection{Partition function with connectivity $12\dots n$}\label{sec:intro-identity}
We consider here the special case of the unit $1_n\in \Z(\H_n)$ of $\H_n$ which corresponds to the partition function with connectivity $12\dots n$
\[
f(1_n) = \alpha_n
\sum_{\substack{\text{lattice paths }P\\\text{on the }n\times n\text{ grid}\\\conn(P)=1_n}} \wt(P)
\]
\begin{ex}Here are the identity paths at $n=3$:
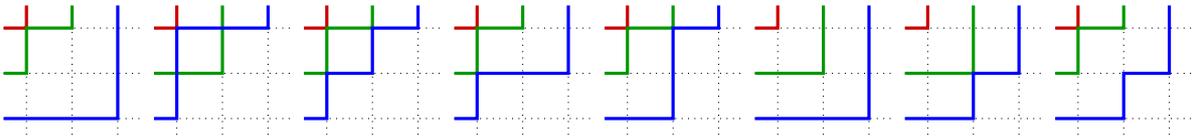
\begin{center}
\begin{tikzpicture}[scale=0.6]\draw[dotted] (0.5,0.5) grid (3.5,-2.5);\begin{scope}[very thick,scale=1cm]\draw[red!80!black] svg {M .5 0 L 1 0 1 .5}; \draw[green!60!black] svg {M .5 -1 L 1 -1 1 -.5 L 1 0 1.5 0 L 2 0 2 .5}; \draw[blue] svg {M .5 -2 L 1.5 -2 L 2.5 -2 L 3 -2 3 -1.5 L 3 -.5 L 3 .5}; \end{scope}\end{tikzpicture}
\begin{tikzpicture}[scale=0.6]\draw[dotted] (0.5,0.5) grid (3.5,-2.5);\begin{scope}[very thick,scale=1cm]\draw[red!80!black] svg {M .5 0 L 1 0 1 .5}; \draw[green!60!black] svg {M .5 -1 L 1.5 -1 L 2 -1 2 -.5 L 2 .5}; \draw[blue] svg {M .5 -2 L 1 -2 1 -1.5 L 1 -.5 L 1 0 1.5 0 L 2.5 0 L 3 0 3 .5}; \end{scope}\end{tikzpicture}
\begin{tikzpicture}[scale=0.6]\draw[dotted] (0.5,0.5) grid (3.5,-2.5);\begin{scope}[very thick,scale=1cm]\draw[red!80!black] svg {M .5 0 L 1 0 1 .5}; \draw[green!60!black] svg {M .5 -1 L 1 -1 1 -.5 L 1 0 1.5 0 L 2 0 2 .5}; \draw[blue] svg {M .5 -2 L 1 -2 1 -1.5 L 1 -1 1.5 -1 L 2 -1 2 -.5 L 2 0 2.5 0 L 3 0 3 .5}; \end{scope}\end{tikzpicture}
\begin{tikzpicture}[scale=0.6]\draw[dotted] (0.5,0.5) grid (3.5,-2.5);\begin{scope}[very thick,scale=1cm]\draw[red!80!black] svg {M .5 0 L 1 0 1 .5}; \draw[green!60!black] svg {M .5 -1 L 1 -1 1 -.5 L 1 0 1.5 0 L 2 0 2 .5}; \draw[blue] svg {M .5 -2 L 1 -2 1 -1.5 L 1 -1 1.5 -1 L 2.5 -1 L 3 -1 3 -.5 L 3 .5}; \end{scope}\end{tikzpicture}
\begin{tikzpicture}[scale=0.6]\draw[dotted] (0.5,0.5) grid (3.5,-2.5);\begin{scope}[very thick,scale=1cm]\draw[red!80!black] svg {M .5 0 L 1 0 1 .5}; \draw[green!60!black] svg {M .5 -1 L 1 -1 1 -.5 L 1 0 1.5 0 L 2 0 2 .5}; \draw[blue] svg {M .5 -2 L 1.5 -2 L 2 -2 2 -1.5 L 2 -.5 L 2 0 2.5 0 L 3 0 3 .5}; \end{scope}\end{tikzpicture}
\begin{tikzpicture}[scale=0.6]\draw[dotted] (0.5,0.5) grid (3.5,-2.5);\begin{scope}[very thick,scale=1cm]\draw[red!80!black] svg {M .5 0 L 1 0 1 .5}; \draw[green!60!black] svg {M .5 -1 L 1.5 -1 L 2 -1 2 -.5 L 2 .5}; \draw[blue] svg {M .5 -2 L 1.5 -2 L 2.5 -2 L 3 -2 3 -1.5 L 3 -.5 L 3 .5}; \end{scope}\end{tikzpicture}
\begin{tikzpicture}[scale=0.6]\draw[dotted] (0.5,0.5) grid (3.5,-2.5);\begin{scope}[very thick,scale=1cm]\draw[red!80!black] svg {M .5 0 L 1 0 1 .5}; \draw[green!60!black] svg {M .5 -1 L 1.5 -1 L 2 -1 2 -.5 L 2 .5}; \draw[blue] svg {M .5 -2 L 1.5 -2 L 2 -2 2 -1.5 L 2 -1 2.5 -1 L 3 -1 3 -.5 L 3 .5}; \end{scope}\end{tikzpicture}
\begin{tikzpicture}[scale=0.6]\draw[dotted] (0.5,0.5) grid (3.5,-2.5);\begin{scope}[very thick,scale=1cm]\draw[red!80!black] svg {M .5 0 L 1 0 1 .5}; \draw[green!60!black] svg {M .5 -1 L 1 -1 1 -.5 L 1 0 1.5 0 L 2 0 2 .5}; \draw[blue] svg {M .5 -2 L 1.5 -2 L 2 -2 2 -1.5 L 2 -1 2.5 -1 L 3 -1 3 -.5 L 3 .5}; \end{scope}\end{tikzpicture}
\end{center}
\end{ex}
\begin{thm}\label{thm:f1n}
The partition function $f(1_n)$, as an element of $\FF[x_1^\pm,\ldots,x_n^\pm]^{\mathcal S_n}_0$, is determined by $f(1_{0})=1$ and  the recurrence relation 
\begin{equation}\label{eq:frec}
(x_1 + \dots + x_n)  f(1_n)=
\sum_{i=1}^n 
x_i  f(1_{n-1})[\hat{x_i}]
\prod_{\substack{j=1 \\ j\neq i}}^n \frac{(1-q^{-1} x_j/x_i)(1- q t^{-1} x_j/x_i)
(1- t x_i/x_j)}{-x_j/x_i(1-x_j/x_i)}
\end{equation}
where $f(1_{n-1})[\hat{x_i}] \in \FF[x_1^\pm,\ldots,x_{i-1},x_{i+1},\ldots,x_n^\pm]^{\mathcal S_n}_0$.
\end{thm}
In \S \ref{sec:Lattice_path} we will show that the partition functions $f(c)$ can be interpreted using the $R$-matrices of $\mathcal U_{t^{1/2}}(\widehat{gl}_{\infty})$. This means that one can study them using the quantum inverse scattering method (QISM) \cite{KBI} (see also \cite{JLam}), and indeed the partition function $f(1_n)$ was already considered in \cite{BW-coloured}. We note, however, that the  recurrence relation \eqref{eq:frec} has an unusual form from the point of view of the QISM due to the division by $x_1 + \dots + x_j$ at every $j^{\rm th}$ step of the recursive computation of $f(1_n)$. It is not clear to us if one can prove \eqref{eq:frec} or find an alternative recursion relation for $f(1_n)$ using the QISM. Our proof of \eqref{eq:frec} is indirect; it follows as a corollary of our proof of Theorem~\ref{thm:square} given in \S\ref{sec:proof}.

\subsection{The commuting scheme}\label{sec:intro-commut}
\newcommand\C{\mathfrak C}
We now give a geometric interpretation of the partition function $f(1_n)$ of the previous section in terms of the commuting scheme. (The geometric interpretation of other $f(c)$ is left for future work.)

Given $n\in\ZZ_{\ge0}$, define the {\em commuting scheme}\/ $\C_n$ to be the affine scheme in
$\g_n(\CC)^2$ consisting of pairs of commuting $n\times n$ matrices:
\begin{equation}\label{eq:defE}
\C_n
:=
\{
(A,B)\in \g_n(\CC)\times \g_n(\CC):\ [A,B]=0
\}
\end{equation}
$\C_n$ is known to be irreducible, of dimension $n(n+1)$ \cite{MT-commut}, smooth in codimension 1 \cite{Popov-commut}, and we have the conjecture, generally attributed to Artin and Hochster in 1982:
\begin{conj}\label{conj:CM}
$\C_n$ is Cohen--Macaulay.
\end{conj}
which would imply that $\C_n$ is reduced and normal.
This conjecture will not be addressed in the present paper. \rem[gray]{but we'll use it to state our main theorem}

The group $GL_2(\CC)\times GL_n(\CC)$ acts on $\g_n(\CC)^2$ by
\[
(f,g)\in GL_2(\CC)\times GL_n(\CC)
: \begin{pmatrix}A\\B
\end{pmatrix}
\mapsto
f
\begin{pmatrix}
gAg^{-1}\\
gBg^{-1}
\end{pmatrix}
\]
and leaves $\C_n$ invariant.
In particular, the maximal torus $T=(\CC^\times)^{2+n}$ induces a $\ZZ^{2+n}$-grading of the corresponding
coordinate rings; explicitly,
\[
\deg A_{ij}=\varepsilon_1+\varepsilon_{i+2}-\varepsilon_{j+2}
\qquad
\deg B_{ij}=\varepsilon_2+\varepsilon_{i+2}-\varepsilon_{j+2}
\qquad
i,j=1,\ldots,n
\]
where $\varepsilon_i$ is the unit vector with nonzero coordinate $i$.

We consider the Hilbert series $\chi(\C_n)$ of the commuting scheme with
the multigrading above (i.e., the $T$-character of its coordinate ring);
in fact, it is convenient to consider instead
the {\em $K$-polynomial} (or Poincar\'e polynomial) of $\C_n$, defined by
\[
K_{n} = \frac{\chi(\C_n)}{\chi(\g_n(\CC)^2)}
\]
Explicitly, introducing formal variables $q_1,q_2,x_1,\ldots,x_n$ corresponding to $(\CC^{\times})^{2+n}$, one has
\[
\chi(\g_n(\CC)^2) = \prod_{i,j=1}^n 
\frac{1}{(1-q_1 x_i x_j^{-1})(1-q_2 x_i x_j^{-1})}
\]

$K_{n}$ is known to be a Laurent polynomial in $\ZZ[q_1^{\pm},q_2^{\pm},x_1^\pm,\ldots,x_n^\pm]$.
Furthermore, because of the $GL_2(\CC)\times GL_n(\CC)$-invariance of $\C_n$, 
$K_{n}$ is invariant under the action
of the corresponding Weyl group, i.e., it is a {\em symmetric}\/ Laurent polynomial in the variables
$q_1,q_2$ and $x_1,\ldots,x_n$ separately.

Finally, one can perform the substitution $q_i\mapsto 1-q_i,x_i\mapsto 1-x_i$ in $K_n$, expand in power series and keep only the
lowest degree terms in those variables.
We obtain this way a homogeneous polynomial of degree $\codim \C_n = n(n-1)$ in
$\ZZ[q_1,q_2,x_1,\ldots,x_n]$ (with its ordinary grading), called the {\em multidegree}\/ of $\C_n$ and denoted by
\[
D_n := \mdeg \C_n
\]

Our first result on the commuting scheme is the following:
\begin{thm}\label{thm:mainK}
Assuming Conjecture~\ref{conj:CM},
the following formula holds for the $K$-polynomial of the commuting scheme:
\begin{equation}\label{eq:mainK}
K_n = (q_1 q_2)^{\frac{n(n-1)}{2}} f(1_n)
\end{equation}
with the identification $q=q_1^{-1}$, $t=(q_1q_2)^{-1}$.
\end{thm}

Performing the substitution of variables above, it is convenient to slightly rearrange the Boltzmann weights of \S\ref{sec:intro-lattice} to:
\begin{equation}\label{eq:wtK}
\wt_K(P)=\prod_{i,j=1}^n
\begin{cases}
1-q_1q_2&
\begin{tikzpicture}[baseline=-3pt]
\draw[dotted] (-0.5,0) -- (0.5,0) (0,-0.5) -- (0,0.5);
\useasboundingbox (0,-0.7) -- (0,0.7);
\draw[very thick,red!80!black] (-0.5,0) -- (0,0) -- (0,0.5);
\end{tikzpicture}
,
\begin{tikzpicture}[baseline=-3pt]
\draw[dotted] (-0.5,0) -- (0.5,0) (0,-0.5) -- (0,0.5);
\useasboundingbox (0,-0.7) -- (0,0.7);
\draw[very thick,red!80!black] (0.5,0) -- (0,0) -- (0,-0.5);
\end{tikzpicture}
,
\begin{tikzpicture}[baseline=-3pt]
\draw[dotted] (-0.5,0) -- (0.5,0) (0,-0.5) -- (0,0.5);
\useasboundingbox (0,-0.7) -- (0,0.7);
\draw[very thick,red!80!black] (-0.5,0) -- (0,0) -- (0,0.5);
\draw[very thick,green!60!black] (0.5,0) -- (0,0) -- (0,-0.5);
\end{tikzpicture}
\\
q_1 (1-q_1q_2)&
\begin{tikzpicture}[baseline=-3pt]
\draw[dotted] (-0.5,0) -- (0.5,0) (0,-0.5) -- (0,0.5);
\useasboundingbox (0,-0.7) -- (0,0.7);
\draw[very thick,green!60!black] (-0.5,0) -- (0,0) -- (0,0.5);
\draw[very thick,red!80!black] (0.5,0) -- (0,0) -- (0,-0.5);
\end{tikzpicture}
\\
1-q_1 x_i x_j^{-1}&
\begin{tikzpicture}[baseline=-3pt]
\draw[dotted] (-0.5,0) -- (0.5,0) (0,-0.5) -- (0,0.5);
\useasboundingbox (0,-0.7) -- (0,0.7);
\draw[very thick,green!60!black] (0,-0.5) -- (0,0.5);
\draw[very thick,red!80!black] (-0.5,0) -- (0.5,0);
\end{tikzpicture},
\begin{tikzpicture}[baseline=-3pt]
\draw[dotted] (-0.5,0) -- (0.5,0) (0,-0.5) -- (0,0.5);
\useasboundingbox (0,-0.7) -- (0,0.7);
\draw[very thick,red!80!black] (-0.5,0) -- (0.5,0);
\end{tikzpicture}
\\
q_1q_2(1-q_1 x_i x_j^{-1})&
\begin{tikzpicture}[baseline=-3pt]
\draw[dotted] (-0.5,0) -- (0.5,0) (0,-0.5) -- (0,0.5);
\useasboundingbox (0,-0.7) -- (0,0.7);
\draw[very thick,red!80!black] (0,-0.5) -- (0,0.5);
\draw[very thick,green!60!black] (-0.5,0) -- (0.5,0);
\end{tikzpicture}
\\
q_2(1-q_1 x_i x_j^{-1})&
\begin{tikzpicture}[baseline=-3pt]
\draw[dotted] (-0.5,0) -- (0.5,0) (0,-0.5) -- (0,0.5);
\useasboundingbox (0,-0.7) -- (0,0.7);
\draw[very thick,red!80!black] (0,-0.5) -- (0,0.5);
\end{tikzpicture}
\\
x_ix_j^{-1}(1-q_2 x_i^{-1}x_j)&
\begin{tikzpicture}[baseline=-3pt]
\draw[dotted] (-0.5,0) -- (0.5,0) (0,-0.5) -- (0,0.5);
\useasboundingbox (0,-0.7) -- (0,0.7);
\end{tikzpicture}
\end{cases}
\end{equation}
Then \eqref{eq:mainK} can be formulated equivalently as
\[
K_n=
(1-q_1q_2)^{-n}\sum_{\substack{\text{lattice paths }P\\\text{on the }n\times n\text{ grid}\\\text{with identity connectivity}}} \wt_K(P)
\]

\begin{ex}At $n=2$, the weights are
\begin{align*}
\wt_K\left(
\begin{tikzpicture}[scale=0.75,baseline=(current  bounding  box.center)]\draw[dotted] (0.5,0.5) grid (2.5,-1.5);\begin{scope}[very thick,scale=1cm]\draw[red!80!black] svg {M .5 0 L 1 0 1 .5}; \draw[green!60!black] svg {M .5 -1 L 1 -1 1 -.5 L 1 0 1.5 0 L 2 0 2 .5}; \end{scope}\end{tikzpicture}\right)
&=(1-q_1q_2)^3(1-q_2)
\\
\wt_K\left(\begin{tikzpicture}[scale=0.75,baseline=(current  bounding  box.center)]\draw[dotted] (0.5,0.5) grid (2.5,-1.5);\begin{scope}[very thick,scale=1cm]\draw[red!80!black] svg {M .5 0 L 1 0 1 .5}; \draw[green!60!black] svg {M .5 -1 L 1.5 -1 L 2 -1 2 -.5 L 2 .5}; \end{scope}\end{tikzpicture}\right)
&=q_2(1-q_1q_2)^2(1-q_1\,x_1x_2^{-1})(1-q_1\,x_2x_1^{-1})
\end{align*}
and one can check that their sum divided by the common factor $(1-q_1q_2)^2$ indeed reproduces
the $K$-polynomial of $\C_2$
\[
K_2=
1+q_1^{2}q_2-q_1q_2x_{1}x_{2}^{-1}-q_1q_2-q_1q_2x_{1}^{-1}x_{2}+q_1q_2^{2}
\]
\end{ex}

Define another (Boltzmann) weight in a similar fashion:
\begin{equation}\label{eq:wtH}
\wt_H(P)=\prod_{i,j=1}^n
\begin{cases}
q_1+q_2&
\begin{tikzpicture}[baseline=-3pt]
\draw[dotted] (-0.5,0) -- (0.5,0) (0,-0.5) -- (0,0.5);
\useasboundingbox (0,-0.7) -- (0,0.7);
\draw[very thick,red!80!black] (-0.5,0) -- (0,0) -- (0,0.5);
\end{tikzpicture}
,
\begin{tikzpicture}[baseline=-3pt]
\draw[dotted] (-0.5,0) -- (0.5,0) (0,-0.5) -- (0,0.5);
\useasboundingbox (0,-0.7) -- (0,0.7);
\draw[very thick,red!80!black] (0.5,0) -- (0,0) -- (0,-0.5);
\end{tikzpicture}
,
\begin{tikzpicture}[baseline=-3pt]
\draw[dotted] (-0.5,0) -- (0.5,0) (0,-0.5) -- (0,0.5);
\useasboundingbox (0,-0.7) -- (0,0.7);
\draw[very thick,red!80!black] (-0.5,0) -- (0,0) -- (0,0.5);
\draw[very thick,green!60!black] (0.5,0) -- (0,0) -- (0,-0.5);
\end{tikzpicture}
,
\begin{tikzpicture}[baseline=-3pt]
\draw[dotted] (-0.5,0) -- (0.5,0) (0,-0.5) -- (0,0.5);
\useasboundingbox (0,-0.7) -- (0,0.7);
\draw[very thick,green!60!black] (-0.5,0) -- (0,0) -- (0,0.5);
\draw[very thick,red!80!black] (0.5,0) -- (0,0) -- (0,-0.5);
\end{tikzpicture}
\\
q_1 + x_i - x_j&
\begin{tikzpicture}[baseline=-3pt]
\draw[dotted] (-0.5,0) -- (0.5,0) (0,-0.5) -- (0,0.5);
\useasboundingbox (0,-0.7) -- (0,0.7);
\draw[very thick,green!60!black] (0,-0.5) -- (0,0.5);
\draw[very thick,red!80!black] (-0.5,0) -- (0.5,0);
\end{tikzpicture}
,
\begin{tikzpicture}[baseline=-3pt]
\draw[dotted] (-0.5,0) -- (0.5,0) (0,-0.5) -- (0,0.5);
\useasboundingbox (0,-0.7) -- (0,0.7);
\draw[very thick,red!80!black] (0,-0.5) -- (0,0.5);
\draw[very thick,green!60!black] (-0.5,0) -- (0.5,0);
\end{tikzpicture}
,
\begin{tikzpicture}[baseline=-3pt]
\draw[dotted] (-0.5,0) -- (0.5,0) (0,-0.5) -- (0,0.5);
\useasboundingbox (0,-0.7) -- (0,0.7);
\draw[very thick,red!80!black] (-0.5,0) -- (0.5,0);
\end{tikzpicture}
,
\begin{tikzpicture}[baseline=-3pt]
\draw[dotted] (-0.5,0) -- (0.5,0) (0,-0.5) -- (0,0.5);
\useasboundingbox (0,-0.7) -- (0,0.7);
\draw[very thick,red!80!black] (0,-0.5) -- (0,0.5);
\end{tikzpicture}
\\
q_2 - x_i + x_j&
\begin{tikzpicture}[baseline=-3pt]
\draw[dotted] (-0.5,0) -- (0.5,0) (0,-0.5) -- (0,0.5);
\useasboundingbox (0,-0.7) -- (0,0.7);
\end{tikzpicture}
\end{cases}
\end{equation}
where colours of paths can be substituted freely.

We can state our second result:
\begin{thm}\label{thm:mainH}
The following formula holds for the multidegree of the commuting scheme:
\begin{equation}\label{eq:mainH}
D_n := (q_1+q_2)^{-n}\sum_{\normalfont{\substack{\text{lattice paths }P\\\text{on the }n\times n\text{ grid}\\\text{with identity connectivity}}}}
\wt_H(P)
\end{equation}
\end{thm}
Note that Conjecture~\ref{conj:CM} and Theorem~\ref{thm:mainK} imply Theorem~\ref{thm:mainH}, as can be easily checked. However Theorem~\ref{thm:mainH} can be proven independently of Conjecture~\ref{conj:CM}.

\subsection{Plan of the paper}
The rest of the paper is organized as follows. 
In \S \ref{sec:hecke}, we provide details on the Hecke algebra construction of lattice paths partition functions,
and on the map $\Phi$ from the center of Hecke to symmetric functions. 
In \S \ref{sec:shuffle}, we give details on the shuffle algebra, the map $\Upsilon$ and derive the recurrence relations which are required for proving Theorems \ref{thm:square}-\ref{thm:mainK}. In \S \ref{sec:proof}, we prove our Theorem~\ref{thm:square} by an explicit comparison of $\Phi$ and $\Upsilon\circ f$, and obtain Theorem~\ref{thm:f1n} as a corollary of all that precedes.
In \S \ref{sec:commut}, we turn to the geometry of the commuting scheme
and prove Theorems~\ref{thm:mainK} and \ref{thm:mainH}, ending with some comments.

\section{Hecke algebra and lattice paths}\label{sec:hecke}
\subsection{The Hecke algebra and its center}

As in \S \ref{sec:intro-center}, we consider the Hecke algebra $\H_{n}$, an $\FF$-algebra with generators
$T_1,\ldots,T_{n-1}$ and relations
\[
T_iT_{i+1}T_i = T_{i+1}T_i T_{i+1}\qquad T_i^2=(t-1)T_i+t\qquad T_iT_j=T_jT_i\ (|i-j|\ge2)
\]
In this paper we choose $\FF$ to be $\QQ(q,t)$, though we shall not need the variable $q$ until \S \ref{sec:exch}.
Note $\H_0\cong \H_1\cong \FF$.

We have the simple lemma:
\begin{lem}\label{lem:basic}
\begin{align*}
T_i T_w &= \begin{cases}
T_{s_i w} & |s_i w| > |w|
\\
(t-1)T_w + t T_{s_i w} & |s_i w| < |w|
\end{cases}
\\
T_w T_i &= \begin{cases}
T_{ws_i} & |w s_i| > |w|
\\
(t-1)T_w + t T_{w s_i} & |w s_i| < |w|
\end{cases}
\end{align*}
\end{lem}
expressing the action of $T_i$ on the standard basis $(T_w)_{w\in \mathcal S_{n}}$, where $T_w:=T_{i_1}\ldots T_{i_k}$
if $w=s_{i_1}\ldots s_{i_k}$ is a reduced word in the elementary transpositions $s_i$, and $|w|=k$.

We shall often need two particular elements of the center $\Z(\H_n)$, namely
the complete symmetrizer
\begin{equation}\label{eq:defS}
S_n=\sum_{w\in \mathcal S_n} T_w
\end{equation}
and the complete antisymmetrizer
\begin{equation}\label{eq:defA}
A_n=\sum_{w\in \mathcal S_n} (-t)^{-|w|} T_w
\end{equation}
Both are primitive idempotents up to normalization. \rem[gray]{there'll be pictures expressing this in section \ref{sec:proof}}

Next, we consider the shuffle product structure:
\begin{prop}\label{prop:Zshuffle}
The product $*$ in \eqref{eq:defZstar} is a well-defined map, making $\Z=\bigoplus_{n\ge0} \Z(\H_n)$ a graded (associative) algebra.
\end{prop}
\rem[gray]{I removed commutative from the theorem because it's a pain to prove}
The nontrival part of the proof can be formulated by repeatedly applying Lemma~\ref{lem:basic}; instead,
we shall illustrate it by using the standard diagrammatic way to depict the Hecke algebra, as acting on $n$ strands:
\[
T_i=\begin{tikzpicture}[baseline=0.5cm]
\draw[invarrow] (0,0) -- (0,1);
\node at (1,0.5) {$\cdots$};
\draw[invarrow] (2,0) -- (2,1);
\draw[invarrow=0.75] (3,0) --  (4,1);
\draw[invarrow=0.75] (4,0) -- (3.6,0.4) (3.4,0.6) -- (3,1);
\draw[invarrow] (5,0) -- (5,1);
\node at (6,0.5) {$\cdots$};
\draw[invarrow] (7,0) -- (7,1);
\draw[decorate,decoration=brace] (2,-0.2) -- node[below] {$i-1$} (0,-0.2);
\draw[decorate,decoration=brace] (7,-0.2) -- node[below] {$n-i-1$} (5,-0.2);
\end{tikzpicture}
\]
The arrows keep track of the ordering of operators: moving forward w.r.t.\ the orientation of a line is reading an expression right to left. Sometimes we will omit arrows; then all lines are implicitly oriented downwards.
(The diagrammatic calculus that follows can be viewed as warming up for the considerably more difficult proof of our Theorem~\ref{thm:square}, which will also be diagrammatic.)
\begin{proof}
Associativity of the product is obvious: an $m$-fold product can be expressed as
\begin{equation}\label{eq:assoc}
x_1*\cdots*x_m = \sum_{w\in \mathcal S^{\lambda_1,\ldots,\lambda_m}}
t^{-|w|} T_w (x_1\otimes\cdots\otimes x_m) T_{w^{-1}}
\end{equation}
where $x_i\in \H_{\lambda_i}$ and $\mathcal S^{\lambda_1,\ldots,\lambda_m}$ is the set
of shortest representatives in $\mathcal S_n/\mathcal S_{\lambda_1,\ldots,\lambda_m}$,
$\mathcal S_{\lambda_1,\ldots,\lambda_m}=\mathcal S_{\lambda_1}\times\cdots\times \mathcal S_{\lambda_m}$.
\rem[gray]{note that this doesn't require centrality, so we have a shuffle product
on the whole of $\H$.}

The only nontrivial part is to show that $*$ is well-defined as a map from $\Z(\H_k)\otimes \Z(\H_\ell)$ to $\Z(\H_{k+\ell})$, i.e., we need to show that the right hand side of \eqref{eq:defZstar} is central, or equivalently that it commutes with all $T_i$s.

We recall that $\mathcal S^{k,\ell}$ is the set of shortest representatives
in $\mathcal S_{k+\ell}/(\mathcal S_k\times \mathcal S_\ell)$; equivalently,
it is the set of (Grassmannian) permutations of size $k+\ell$ with a unique possible descent at $(k,k+1)$.

Fixing $i=1,\ldots,n-1$, we distinguish four subsets of $\mathcal S^{k,\ell}$ according to the preimages of $i$ and $i+1$:
\[
\mathcal S_{\lessgtr,\lessgtr}=\{w\in \mathcal S^{k,\ell}:\ w^{-1}(i)\lessgtr k+1/2,\ w^{-1}(i+1)\lessgtr k+1/2\}
\]
Note that $w\mapsto s_i w$ induces a bijection between $\mathcal S_{<,>}$ and $\mathcal S_{>,<}$.
We now discuss the corresponding terms in \eqref{eq:defZstar} separately:
\begin{itemize}
\item If $w\in \mathcal S_{<,<}$, then $T_w (a\otimes b) T_{w^{-1}}$ commutes with $T_i$
provided $a\in \Z(\H_k)$, as shown diagrammatically:
\begin{multline*}
T_w (a\otimes b) T_{w^{-1}} T_i
=
\begin{tikzpicture}[rounded corners,scale=0.8,baseline=(a.base)]
\begin{scope}[every path/.style={draw=white,double=black,ultra thick, double distance=0.4pt}]
\draw (6,0.5) -- (6,2.5);
\draw (5,0.5) -- (4,2) -- (4,2.5);
\draw (4,0.5) -- (1,2) -- (1,2.5);
\draw (3,0.5) -- (5,2) -- (5,2.5);
\draw (2,0.5) -- (3,2) -- (2,2.5);
\draw (1,0.5) -- (2,2) -- (3,2.5);
\draw (6,0) -- (6,-1.5);
\draw (3,0) -- (5,-1.5);
\draw (5,0) -- (4,-1.5);
\draw (2,0) -- (3,-1.5);
\draw (1,0) -- (2,-1.5);
\draw (4,0) -- (1,-1.5);
\end{scope}
\draw[decorate,decoration=brace,sharp corners] (0.6,0.5) -- node[left] {$\ss T_{w^{-1}}$} (0.6,2);
\draw[decorate,decoration=brace,sharp corners] (0.6,-1.5) -- node[left] {$\ss T_{w}$} (0.6,0);
\node[draw=black,fill=white,rectangle,minimum height=0.4cm,minimum width=2.1cm,inner sep=0pt] at (2,0.25) (a) {$a$};
\node[draw=black,fill=white,rectangle,minimum height=0.4cm,minimum width=2.1cm,inner sep=0pt] at (5,0.25) {$b$};
\end{tikzpicture}
=
\begin{tikzpicture}[rounded corners,scale=0.8,baseline=(a.base)]
\begin{scope}[every path/.style={draw=white,double=black,ultra thick, double distance=0.4pt}]
\draw (6,0.5) -- (6,2.5);
\draw (5,0.5) -- (5,1) -- (4,2.5);
\draw (4,0.5) -- (4,1) -- (1,2.5);
\draw (3,0.5) -- (3,1) -- (5,2.5);
\draw (2,0.5) -- (1,1) -- (2,2.5);
\draw (1,0.5) -- (2,1) -- (3,2.5);
\draw (6,0) -- (6,-1.5);
\draw (3,0) -- (5,-1.5);
\draw (5,0) -- (4,-1.5);
\draw (2,0) -- (3,-1.5);
\draw (1,0) -- (2,-1.5);
\draw (4,0) -- (1,-1.5);
\end{scope}
\node[draw=black,fill=white,rectangle,minimum height=0.4cm,minimum width=2.1cm,inner sep=0pt] at (2,0.25) (a) {$a$};
\node[draw=black,fill=white,rectangle,minimum height=0.4cm,minimum width=2.1cm,inner sep=0pt] at (5,0.25) {$b$};
\end{tikzpicture}
\\
=
\begin{tikzpicture}[rounded corners,scale=0.8,baseline=(a.base)]
\begin{scope}[every path/.style={draw=white,double=black,ultra thick, double distance=0.4pt}]
\draw (6,1) -- (6,2.5);
\draw (5,1) -- (4,2.5);
\draw (4,1) -- (1,2.5);
\draw (3,1) -- (5,2.5);
\draw (1,1) -- (2,2.5);
\draw (2,1) -- (3,2.5);
\draw (6,0.5) -- (6,-1.5);
\draw (3,0.5) -- (3,0) -- (5,-1.5);
\draw (5,0.5) -- (5,0) -- (4,-1.5);
\draw (1,0.5) -- (2,0) -- (3,-1.5);
\draw (2,0.5) -- (1,0) -- (2,-1.5);
\draw (4,0.5) -- (4,0) -- (1,-1.5);
\end{scope}
\node[draw=black,fill=white,rectangle,minimum height=0.4cm,minimum width=2.1cm,inner sep=0pt] at (2,0.75) (a) {$a$};
\node[draw=black,fill=white,rectangle,minimum height=0.4cm,minimum width=2.1cm,inner sep=0pt] at (5,0.75) {$b$};
\end{tikzpicture}
=
\begin{tikzpicture}[rounded corners,scale=0.8,baseline=(a.base)]
\begin{scope}[every path/.style={draw=white,double=black,ultra thick, double distance=0.4pt}]
\draw (6,1) -- (6,2.5);
\draw (5,1) -- (4,2.5);
\draw (4,1) -- (1,2.5);
\draw (3,1) -- (5,2.5);
\draw (1,1) -- (2,2.5);
\draw (2,1) -- (3,2.5);
\draw (6,0.5) -- (6,-1.5);
\draw (3,0.5) -- (5,-1) -- (5,-1.5);
\draw (5,0.5) -- (4,-1) -- (4,-1.5);
\draw (1,0.5) -- (2,-1) -- (3,-1.5);
\draw (2,0.5) -- (3,-1) -- (2,-1.5);
\draw (4,0.5) -- (1,-1) -- (1,-1.5);
\end{scope}
\node[draw=black,fill=white,rectangle,minimum height=0.4cm,minimum width=2.1cm,inner sep=0pt] at (2,0.75) (a) {$a$};
\node[draw=black,fill=white,rectangle,minimum height=0.4cm,minimum width=2.1cm,inner sep=0pt] at (5,0.75) {$b$};
\end{tikzpicture}
= T_i T_w(a\otimes b)T_{w^{-1}}
\end{multline*}
\item Similarly, if $w\in \mathcal S_{>,>}$, then $T_w (a\otimes b) T_{w^{-1}}$ commutes with $T_i$
provided $b\in \Z(\H_\ell)$.
\item Finally, we group together the two terms corresponding to $w\in \mathcal S_{<,>}$ and $s_i w\in\mathcal S_{>,<}$ to obtain:
\begin{multline*}
(T_w(a\otimes b)T_{w^{-1}}+t^{-1}T_{s_iw}(a\otimes b)T_{w^{-1}s_i})T_i
\\
=
\begin{tikzpicture}[rounded corners,scale=0.8,baseline=(a.base)]
\begin{scope}[every path/.style={draw=white,double=black,ultra thick, double distance=0.4pt}]
\draw (6,0.5) -- (6,2.5);
\draw (5,0.5) -- (3,2) -- (2,2.5);
\draw (4,0.5) -- (1,2) -- (1,2.5);
\draw (3,0.5) -- (5,2) -- (5,2.5);
\draw (2,0.5) -- (4,2) -- (4,2.5);
\draw (1,0.5) -- (2,2) -- (3,2.5);
\draw (6,0) -- (6,-1.5);
\draw (3,0) -- (5,-1.5);
\draw (2,0) -- (4,-1.5);
\draw (5,0) -- (3,-1.5);
\draw (1,0) -- (2,-1.5);
\draw (4,0) -- (1,-1.5);
\end{scope}
\draw[decorate,decoration=brace,sharp corners] (0.6,0.5) -- node[left] {$\ss T_{w^{-1}}$} (0.6,2);
\draw[decorate,decoration=brace,sharp corners] (0.6,-1.5) -- node[left] {$\ss T_{w}$} (0.6,0);
\node[draw=black,fill=white,rectangle,minimum height=0.4cm,minimum width=2.1cm,inner sep=0pt] at (2,0.25) (a) {$a$};
\node[draw=black,fill=white,rectangle,minimum height=0.4cm,minimum width=2.1cm,inner sep=0pt] at (5,0.25) {$b$};
\end{tikzpicture}
+t^{-1}
\begin{tikzpicture}[rounded corners,scale=0.8,baseline=(a.base)]
\begin{scope}[every path/.style={draw=white,double=black,ultra thick, double distance=0.4pt}]
\draw (6,0.5) -- (6,2.5);
\draw (5,0.5) -- (2,2) -- (2.2,2.1); \draw (2.8,2.1) -- (2,2.5);
\draw (4,0.5) -- (1,2) -- (1,2.5);
\draw (3,0.5) -- (5,2) -- (5,2.5);
\draw (2,0.5) -- (4,2) -- (4,2.5);
\draw (1,0.5) -- (3,2) -- (2.8,2.1); \draw (2.2,2.1) -- (3,2.5);
\draw (6,0) -- (6,-1.5);
\draw (3,0) -- (5,-1.5);
\draw (2,0) -- (4,-1.5);
\draw (1,0) -- (3,-1.5);
\draw (5,0) -- (2,-1.5);
\draw (4,0) -- (1,-1.5);
\end{scope}
\draw[decorate,decoration=brace,sharp corners] (0.75,0.5) -- node[left] {$\ss T_{w^{-1}s_i}$} (0.75,2);
\draw[decorate,decoration=brace,sharp corners] (0.75,-1.5) -- node[left] {$\ss T_{s_iw}$} (0.75,0);
\node[draw=black,fill=white,rectangle,minimum height=0.4cm,minimum width=2.1cm,inner sep=0pt] at (2,0.25) (a) {$a$};
\node[draw=black,fill=white,rectangle,minimum height=0.4cm,minimum width=2.1cm,inner sep=0pt] at (5,0.25) {$b$};
\end{tikzpicture}
\\
=
\begin{tikzpicture}[rounded corners,scale=0.8,baseline=(a.base)]
\begin{scope}[every path/.style={draw=white,double=black,ultra thick, double distance=0.4pt}]
\draw (6,0.5) -- (6,2);
\draw (5,0.5) -- (2,2);
\draw (4,0.5) -- (1,2);
\draw (3,0.5) -- (5,2);
\draw (2,0.5) -- (4,2);
\draw (1,0.5) -- (3,2);
\draw (6,0) -- (6,-1.5);
\draw (3,0) -- (5,-1.5);
\draw (2,0) -- (4,-1.5);
\draw (5,0) -- (3,-1.5);
\draw (1,0) -- (2,-1.5);
\draw (4,0) -- (1,-1.5);
\end{scope}
\node[draw=black,fill=white,rectangle,minimum height=0.4cm,minimum width=2.1cm,inner sep=0pt] at (2,0.25) (a) {$a$};
\node[draw=black,fill=white,rectangle,minimum height=0.4cm,minimum width=2.1cm,inner sep=0pt] at (5,0.25) {$b$};
\end{tikzpicture}
+
\begin{tikzpicture}[rounded corners,scale=0.8,baseline=(a.base)]
\begin{scope}[every path/.style={draw=white,double=black,ultra thick, double distance=0.4pt}]
\draw (6,0.5) -- (6,2);
\draw (5,0.5) -- (3,2);
\draw (4,0.5) -- (1,2);
\draw (3,0.5) -- (5,2);
\draw (2,0.5) -- (4,2);
\draw (1,0.5) -- (2,2);
\draw (6,0) -- (6,-1.5);
\draw (3,0) -- (5,-1.5);
\draw (2,0) -- (4,-1.5);
\draw (1,0) -- (3,-1.5);
\draw (5,0) -- (2,-1.5);
\draw (4,0) -- (1,-1.5);
\end{scope}
\node[draw=black,fill=white,rectangle,minimum height=0.4cm,minimum width=2.1cm,inner sep=0pt] at (2,0.25) (a) {$a$};
\node[draw=black,fill=white,rectangle,minimum height=0.4cm,minimum width=2.1cm,inner sep=0pt] at (5,0.25) {$b$};
\end{tikzpicture}
+(1-t^{-1})
\begin{tikzpicture}[rounded corners,scale=0.8,baseline=(a.base)]
\begin{scope}[every path/.style={draw=white,double=black,ultra thick, double distance=0.4pt}]
\draw (6,0.5) -- (6,2);
\draw (5,0.5) -- (2,2);
\draw (4,0.5) -- (1,2);
\draw (3,0.5) -- (5,2);
\draw (2,0.5) -- (4,2);
\draw (1,0.5) -- (3,2);
\draw (6,0) -- (6,-1.5);
\draw (3,0) -- (5,-1.5);
\draw (2,0) -- (4,-1.5);
\draw (1,0) -- (3,-1.5);
\draw (5,0) -- (2,-1.5);
\draw (4,0) -- (1,-1.5);
\end{scope}
\node[draw=black,fill=white,rectangle,minimum height=0.4cm,minimum width=2.1cm,inner sep=0pt] at (2,0.25) (a) {$a$};
\node[draw=black,fill=white,rectangle,minimum height=0.4cm,minimum width=2.1cm,inner sep=0pt] at (5,0.25) {$b$};
\end{tikzpicture}
\\
=T_w(a\otimes b)T_{w^{-1}s_i}
+T_{s_iw}(a\otimes b)T_{w^{-1}}
+(1-t^{-1})T_{s_iw}(a\otimes b)T_{w^{-1}s_i}
\end{multline*}
Clearly, the same result is obtained by multiplying by $T_i$ on the left.
\end{itemize}
By multiplying by $t^{-|w|}$ and summing over $w$, we obtain the desired commutation property for the right hand side of \eqref{eq:defZstar}.
\end{proof}

We define an algebra morphism $\Psi$ from $\Lambda=\FF[p_1,p_2,\ldots]$ to $\Z$ as follows.
Define Jucys--Murphy elements in $\H_n$ (see e.g.~\cite{Ram-seminormal}) to be
\[
J_{j,n} = \sum_{i=1}^{j-1} t^{i-j+1} T_i T_{i+1}\ldots T_{j-1} T_{j-2} \ldots T_i,
\qquad j=2,\ldots,n
\]
They form a commutative subalgebra of $\H_n$, and it is well-known that symmetric polynomials in the $J_{j,n}$ are central. Then $\Psi$ sends $p_n$ to the element \rem[gray]{which is just a $t$-deformation of ($n\times$) the sum of cycles}
\[
\Psi(p_n)=[n]_{t^{-1}}\prod_{j=2}^n J_{j,n}
\]
where $[n]_{t^{-1}}=1+\cdots+t^{-(n-1)}$.
\rem[gray]{this agrees with Frobenius map at $t=1$ up to a conventional $n!$ which is related to the choice of isomorphism between class functions and center.}

\begin{prop}[\cite{Wan-Frob}]
$\Psi$ is a graded algebra isomorphism from $\Lambda$ to $\Z$.
\end{prop}
Its inverse is the $\Phi$ of \S \ref{sec:intro-center}.
As a corollary, $\Z$ is commutative.

Denote by
\begin{align}
\tilde h_n &= \Phi(1_n)
\\
\tilde e_n &= \Phi(t^{-\frac{n(n-1)}{2}}T_{w_0}^2)
\\
h_n &= \Phi(S_n)
\\\label{eq:PhiAn}
e_n &= \Phi(A_n)
\end{align}
where $1_n$ is the identity of $\H_n$, and $w_0$ the longest element of $\mathcal S_n$. We can identify these symmetric functions.
Given $u\in \FF$, define $\sigma_u$ to be the automorphism of $\Lambda$
\begin{equation}\label{eq:defsigma}
\sigma_u:\ p_r \mapsto \frac{(1-u)^r}{1-u^r}\,p_r
\end{equation}
\begin{prop}[\cite{Lascoux-Frob}]\label{prop:identsym}
$h_n$ (resp.\ $e_n)$ is the complete (resp.\ elementary) symmetric function of degree $n$.

Furthermore, $\tilde h_n=\sigma_t(h_n)$ and $\tilde e_n=\sigma_t(e_n)$.
\end{prop}
\rem[gray]{is there an elementary proof?}

We also denote by $\sigma_\infty:\ p_r\mapsto (-1)^{r-1}p_r$.
Note $\sigma_{u^{-1}}=\sigma_\infty\sigma_u$. It is well-known that $\sigma_\infty$ exchanges $h_n$ and $e_n$ (and similarly for $\tilde h_n$ and $\tilde e_n$).

\rem[gray]{note a subtlety: at $t=1$, $1_n=\sum_\lambda \frac{1}{hook(\lambda)}s_\lambda$ and so $\tilde h_n=p_1^n$. not so at $t\ne 1$ because hooks get $t$-deformed! this is one way of understanding why the mdeg formula is easier than the K-poly formula. by now this is spelled out in \S \ref{sec:commut}}

\subsection{Partition functions from the Hecke algebra}\label{sec:Pf_Hecke}
\newcommand\M{\mathcal M}
We now work inside the algebra $\H_{2n}$. There are two obvious embeddings of $\H_n$ into $\H_{2n}$,
given by $T_i\mapsto T_i$ and $T_i\mapsto T_{i+n}$ respectively, and we denote the corresponding subalgebras by
$\H_n^{(1)}$ and $\H_n^{(2)}$.
Similarly, denote by $\mathcal S^{(1)}_n$ and $\mathcal S^{(2)}_n$ the two subgroups of $\mathcal S_{2n}$ (both isomorphic to $\mathcal S_n$) acting nontrivially on $\{1,\ldots,n\}$ (resp.\ $\{n+1,\ldots,2n\}$) only.

Also define the {\em $R$-matrix}
\begin{equation}\label{eq:defR}
\check R_i(u)=
1-t + (1-u) T_i
\qquad
i=1,\ldots,2n-1
\end{equation}
as an element of $\H_{2n}[u^\pm]$.

Extending the diagrammatic calculus introduced in the proof of Proposition~\ref{prop:Zshuffle}, we depict $R$-matrices as (flat) crossings of two lines carrying variables (so-called spectral parameters), the parameter of the $R$-matrix being the ratio of spectral parameters:
\[
\check R_i(u) = 
\begin{tikzpicture}[baseline=0.5cm]
\draw[invarrow] (0,0) -- (0,1);
\node at (1,0.5) {$\cdots$};
\draw[invarrow] (2,0) -- (2,1);
\draw[invarrow=0.75] (3,0) -- node[pos=0.75,right] {$\ss u''$} (4,1);
\draw[invarrow=0.75] (4,0) -- node[pos=0.75,left] {$\ss u'$} (3,1);
\draw[invarrow] (5,0) -- (5,1);
\node at (6,0.5) {$\cdots$};
\draw[invarrow] (7,0) -- (7,1);
\draw[decorate,decoration=brace] (2,-0.2) -- node[below] {$i-1$} (0,-0.2);
\draw[decorate,decoration=brace] (7,-0.2) -- node[below] {$2n-i-1$} (5,-0.2);
\end{tikzpicture}
\qquad
u=u''/u'
\]
so that \eqref{eq:defR} can be expressed as:
\[
\begin{tikzpicture}[baseline=0.5cm]
\draw[invarrow=0.75] (3,0)-- node[pos=0.75,right] {$\ss u''$} (4,1);
\draw[invarrow=0.75] (4,0) -- node[pos=0.75,left] {$\ss u'$} (3,1);
\end{tikzpicture}
=
-
\begin{tikzpicture}[baseline=0.5cm]
\draw[invarrow=0.75] (3,0) --  (3.4,0.4) (3.6,0.6) -- (4,1);
\draw[invarrow=0.75] (4,0) -- (3,1);
\end{tikzpicture}
-u
\begin{tikzpicture}[baseline=0.5cm]
\draw[invarrow=0.75] (3,0) --  (4,1);
\draw[invarrow=0.75] (4,0) -- (3.6,0.4) (3.4,0.6) -- (3,1);
\end{tikzpicture}
\qquad
u=u''/u'
\]
\rem[gray]{would be nicer if I changed the sign of $\check R_i$. oh well.}

The {\em Yang--Baxter equation} is known to hold:
\begin{align}\label{eq:ybe}
\check R_i(u) \check R_{i+1}(uv)\check R_i(v) &= \check R_{i+1}(v)\check R_i(uv) \check R_{i+1}(u)
\qquad
i=1,\ldots,2n-2
\\[1mm]\notag
\begin{tikzpicture}[baseline=1.5cm,rounded corners,yscale=0.75]
\draw[invarrow=0.9] (1,0) -- (3,2) -- (3,3);
\draw[invarrow=0.9] (2,0) -- (1,1) -- (1,2) -- (2,3);
\draw[invarrow=0.9] (3,0) -- (3,1) -- (1,3);
\end{tikzpicture}
&=
\begin{tikzpicture}[baseline=1.5cm,rounded corners,yscale=0.75]
\draw[invarrow=0.9] (1,0) -- (1,1) -- (3,3);
\draw[invarrow=0.9] (2,0) -- (3,1) -- (3,2) -- (2,3);
\draw[invarrow=0.9] (3,0) -- (1,2) -- (1,3);
\end{tikzpicture}
\end{align}
as an identity in $\H_{2n}[u^\pm,v^\pm]$, as well as the {\em unitarity equation}
\begin{align}\label{eq:unit}
\check R_{i}(u^{-1})\check R_i(u)
&=(1-tu)(1-tu^{-1})
\\\notag
\begin{tikzpicture}[baseline=1cm,rounded corners]
\draw[invarrow=0.9] (1,0) -- (2,1) -- node[left,pos=0.6] {$\ss u'$} (1,2); \draw[invarrow=0.9] (2,0) -- (1,1) -- node[right,pos=0.6] {$\ss u''$} (2,2); \end{tikzpicture}
&=
(1-tu)(1-tu^{-1})
\begin{tikzpicture}[baseline=1cm,rounded corners]
\draw[invarrow=0.8] (1,0) -- node[left,pos=0.75] {$\ss u'$} (1,2);
\draw[invarrow=0.8] (2,0) -- node[right,pos=0.75] {$\ss u''$} (2,2);
\end{tikzpicture}
\qquad
u=u''/u'
\end{align}

The central object of study of this section is the ``formal partition function'', which is the following element of $\H_{2n}[x_1^\pm,\ldots,x_n^\pm,y_1^\pm,\ldots,y_n^\pm]$:
\begin{align}\label{eq:defZ}
Z(x_1,\ldots,x_n,y_1,\ldots,y_n)
&=\check R_n(x_n/y_1)\check R_{n-1}(x_{n-1}/y_1)\check R_{n+1}(x_n/y_2)
\ldots
\check R_n(x_1/y_n)
\\
&=\notag
\begin{tikzpicture}[baseline=(current  bounding  box.center)]
\foreach\j/\lab in {1/y_1,2/y_2,3/\cdots,4/y_n}
\draw[invarrow=0.95] (\j,0.5) -- node[pos=0.95,left] {$\ss \lab$} (\j,4.5);
\foreach\i/\lab in {4/x_1,3/x_2,2/\vdots,1/x_n}
\draw[invarrow=0.95] (0.5,\i) -- node[pos=0.95,below] {$\ss \lab$} (4.5,\i);
\end{tikzpicture}
\end{align}
We also write $Z = Z(x_1,\ldots,x_n,y_1,\ldots,y_n)$ when there is no need to emphasize the dependence on the variables.

$Z$ has an expansion of the form
\begin{equation}\label{eq:exp}
Z = \sum_{w\in\mathcal{S}_{2n}} Z_w T_w
\end{equation}
where each 
$Z_w \in \ZZ[t^\pm,x_1^\pm,\ldots,x_n^\pm,y_1^\pm,\ldots,y_n^\pm]$.

\subsection{Exchange relation and symmetry}\label{sec:exch}
Consider the operator $\tau_i$ (resp.\ $\tau'_j$) that permutes variables $x_i$ and $x_{i+1}$
(resp.\ $y_j$ and $y_{j+1}$), $i,j=1,\ldots,n-1$.

We have the following fundamental lemma:
\begin{lem}\label{lem:exch}
$Z$ satisfies the {\em exchange relations}:
\begin{align}\label{eq:exch-x}
(\tau_i Z) \check R_{n+i}(x_{i+1}/x_i)
&=
\check R_i(x_{i+1}/x_i) Z
\qquad i=1,\ldots,n-1
\\\notag
\begin{tikzpicture}[rounded corners,baseline=(current  bounding  box.center)]
\foreach\j/\lab in {1/y_1,2/y_2,3/\cdots,4/y_n}
\draw[invarrow=0.95] (\j,0.5) -- node[pos=0.95,left] {$\ss \lab$} (\j,4.5);
\foreach\i/\lab in {4/x_1,1/x_n}
\draw[invarrow=0.95] (0.5,\i) -- node[pos=0.95,below] {$\ss \lab$} (4.5,\i);
\node at (4.25,3.5) {$\ss\vdots$};
\node at (4.25,1.5) {$\ss\vdots$};
\draw[invarrow=0.95] (0.5,3) -- (4.5,3) -- node[pos=0.7,below=1mm] {$\ss x_{i+1}$} (5.5,2);
\draw[invarrow=0.95] (0.5,2) -- (4.5,2) -- node[pos=0.7,above=1mm] {$\ss x_{i}$} (5.5,3);
\end{tikzpicture}
&=
\begin{tikzpicture}[rounded corners,baseline=(current  bounding  box.center)]
\foreach\j/\lab in {1/y_1,2/y_2,3/\cdots,4/y_n}
\draw[invarrow=0.95] (\j,0.5) -- node[pos=0.95,left] {$\ss \lab$} (\j,4.5);
\foreach\i/\lab in {4/x_1,1/x_n}
\draw[invarrow=0.95] (0.5,\i) -- node[pos=0.95,below] {$\ss \lab$} (4.5,\i);
\node at (4.25,3.5) {$\ss\vdots$};
\node at (4.25,1.5) {$\ss\vdots$};
\draw[invarrow=0.97] (-0.5,2) -- (0.5,3) -- node[pos=0.95,below] {$\ss x_{i}$} (4.5,3); 
\draw[invarrow=0.97] (-0.5,3) -- (0.5,2) --  node[pos=0.95,below] {$\ss x_{i+1}$} (4.5,2);
\end{tikzpicture}
\end{align}
\begin{align}\label{eq:exch-y}
\check R_{n+j}(y_j/y_{j+1})
\tau'_j Z
&=
Z\check R_j(y_j/y_{j+1})
\qquad j=1,\ldots,n-1
\\\notag
\begin{tikzpicture}[rounded corners,baseline=2.5cm]
\foreach\j/\lab in {1/{y_1\rlap{\ $\ss\cdots$}},4/{\cdots\ y_n}}
\draw[invarrow=0.95] (\j,0.5) -- node[pos=0.95,left] {$\ss \lab$} (\j,4.5);
\draw[invarrow=0.95] (2,-0.5) -- (3,0.5) -- node[pos=0.95,left] {$\ss y_j$} (3,4.5);
\draw[invarrow=0.95] (3,-0.5) -- (2,0.5) -- node[pos=0.95,left] {$\ss y_{j+1}$} (2,4.5);
\foreach\i/\lab in {4/x_1,3/x_2,2/\vdots,1/x_n}
\draw[invarrow=0.95] (0.5,\i) -- node[pos=0.95,below] {$\ss \lab$} (4.5,\i);
\end{tikzpicture}
\qquad&=\qquad
\begin{tikzpicture}[rounded corners,baseline=2.5cm]
\foreach\j/\lab in {1/{y_1\rlap{\quad$\ss\cdots$}},4/{\cdots\ y_n}}
\draw[invarrow=0.95] (\j,0.5) -- node[pos=0.95,left] {$\ss \lab$} (\j,4.5);
\draw[invarrow=0.95] (2,0.5) -- (2,4.5) -- node[pos=0.75,right] {$\ss y_j$} (3,5.5);
\draw[invarrow=0.95] (3,0.5) -- (3,4.5) -- node[pos=0.75,left] {$\ss y_{j+1}$} (2,5.5);
\foreach\i/\lab in {4/x_1,3/x_2,2/\vdots,1/x_n}
\draw[invarrow=0.95] (0.5,\i) -- node[pos=0.95,below] {$\ss \lab$} (4.5,\i);
\end{tikzpicture}
\end{align}
\end{lem}
\begin{proof}
Both equalities result from repeated application of the Yang--Baxter equation \eqref{eq:ybe}. 
\end{proof}

We now consider the specialization $y_i=qx_i$, $i=1,\ldots,n$.
We first note
\begin{lem}\label{lem:deg}
$Z_w|_{y_i=qx_i}\in \FF[x_1^\pm,\ldots,x_n^\pm]_0$.
\end{lem}
\begin{proof}
Recall that each of the $n^2$ factors of $Z$ corresponds to some vertex $(i,j)$ on the diagram of \eqref{eq:defZ}, and is a polynomial of degree $1$ in the parameter $x_i/y_j$, so after specialization of the Lemma, in $x_i/(qx_j)$. Now fix $r\in\{0,\ldots,n\}$ and consider the degree $i_1+\cdots+i_r$ in the first $r$ variables $x_1,\ldots,x_r$. 
Among the factors of $Z$, the ones corresponding in the diagram to the top-left $r\times r$ square have degree zero, and similarly for the bottom $(n-r)\times (n-r)$ square. The ones for the top-right $r\times(n-r)$ square can have monomials of degree $0$ or $1$, whereas for the bottom-left $(n-r)\times r$ square they can be of degree $0$ or $-1$. Summing these degrees, we find the desired inequality $|i_1+\cdots+i_r|\le r(n-r)$.
\end{proof}

Of course the same proof would work with any $r$-subset of variables,
suggesting that we should study the behaviour of $Z(x_1,\ldots,x_n,qx_1,\ldots,qx_n)$ under the interchange of the variables $x_i$. We have the following condition:
\begin{prop}\label{prop:charsym}
Let $\chi$ be a linear form on $\H_{2n}$. 
$\chi(Z(x_1,\ldots,x_n,qx_1,\ldots,qx_n))$ is a symmetric Laurent polynomial in the $x_i$s provided $\chi$ satisfies
\begin{equation}\label{eq:char}
\chi(ab)=\chi(ba)\qquad \forall a\in\H_{2n},\ b\in\H_n^{(1)}\otimes \H_n^{(2)}
\end{equation}
\end{prop}
\begin{proof}
We assume $n\ge 2$, otherwise the Proposition is trivial.
By definition exchanging $x_i$ and $x_{i+1}$ in $\chi(Z(x_1,\ldots,x_n,qx_1,\ldots,qx_n))$ amounts to considering
$
\chi(\tau_i \tau'_i Z)
$
and then substituting $y_i=qx_i$. One has
\[
\chi(\tau_i \tau'_i Z) = 
\chi(
\check R_i(x_{i+1}/x_i) 
\check R_{n+i}(y_i/y_{i+1})^{-1}
Z
\check R_i(y_i/y_{i+1})
\check R_{n+i}(x_{i+1}/x_i)^{-1})
\]
Assuming the property \eqref{eq:char} of the Proposition, one can rewrite this
\[
\chi(\tau_i \tau'_i Z) = 
\chi(
\check R_i(y_i/y_{i+1})
\check R_i(x_{i+1}/x_i) 
Z
\check R_{n+i}(y_i/y_{i+1})^{-1}
\check R_{n+i}(x_{i+1}/x_i)^{-1})
\]
Substituting $y_i=qx_i$ and simplifying using unitarity equation \eqref{eq:unit}, we obtain the desired symmetry of $\chi(Z(x_1,\ldots,x_n,qx_1,\ldots,qx_n))$.

\end{proof}

In what follows, we restrict ourselves to the following class of linear forms.
Denote by $\left<\bullet\right>$ the linear form that extracts the coefficient of $1$ in a linear combination of $T_w$, i.e., $\left<T_w\right>:=\delta_{w,1}$. One checks that 
\[
\left<ab\right>=\left<ba\right>=\sum_{w\in\mathcal S_{2n}} t^{|w|} a_w b_{w^{-1}},\qquad a=\sum_{w\in\mathcal S_{2n}} a_w T_w,\ b=\sum_{w\in\mathcal S_{2n}} b_w T_w
\]

Given $c\in \Z(\H_n)$, we then define
\begin{equation}\label{eq:deff2}
f(c) := \alpha_n \left<(c\otimes S_n)
Z(x_1,\ldots,x_n,qx_1,\ldots,qx_n)
\right>
\end{equation}
where $S_n$ was defined in \eqref{eq:defS}, and $\alpha_n=(q/t)^{\frac{n(n-1)}{2}}(1-t)^{-n}$ is the same constant as in the introduction.
Clearly, $\chi=\left<(c\otimes S_n)\bullet\right>$ fulfills the condition of Proposition~\ref{prop:charsym}.

\begin{ex}
Three important examples are:
\begin{itemize}
    \item the identity: $f(1_n)$ is the object of \S \ref{sec:intro-identity} and \S \ref{sec:commut}.
    \item the complete antisymmetrizer: $f(A_n)$ will be computed in \S \ref{sec:proof}.
    \item the complete symmetrizer: $f(S_n)$ is also interesting, since the symmetrization makes the model ``color-blind'', and thus equal to the famous Izergin determinant \cite{Iz-6v} for the partition function of the six-vertex model with domain wall boundary conditions, as can be shown by the usual symmetry and recurrence arguments \cite{Kor}. We leave the proof to the interested reader.
\end{itemize}
\rem[gray]{what about the 4th, corresponding to $\tilde h_n$? which of these are Schur-positive, assuming some embedding space (i.e., denominator)?}
\end{ex}

We finally introduce a lattice path representation of $f(c)$, reconnecting
it to the content of \S \ref{sec:intro-lattice}.

\subsection{Lattice path representation}\label{sec:Lattice_path}
Consider the following representation of $\H_{2n}$. 
As a vector space, $\M=(\FF^{n+1})^{\otimes 2n}$.
Letting $(e_i)_{i=1,\ldots,n+1}$ be the standard basis of $\FF^{n+1}$, define the action of $\H_{2n}$ by
\[
T_i (e_{k_1}\otimes\cdots\otimes e_{k_{2n}}) = \sum_{k'_i,k'_{i+1}=1}^{n+1}
\begin{bmatrix}&k_i&\\k'_i&&k_{i+1}\\&k'_{i+1}&
\end{bmatrix}
e_{k_1}\otimes\cdots\otimes e_{k_{i-1}}\otimes e_{k'_i}\otimes e_{k'_{i+1}}
\otimes e_{k_{i+2}}\otimes\cdots\otimes e_{k_{2n}}
\]
where
\begin{equation}\label{eq:rep}
\begin{bmatrix}&i&\\j&&k\\&\ell&
\end{bmatrix}
=
\begin{cases}
t-1 & i=j>k=\ell\\
1 & j=k>i=\ell\\
t & j=k\le i=\ell\\
0 & \text{else}
\end{cases}
\end{equation}

\newcommand\cosets{\mathcal{S}^{(2)}_n\backslash\mathcal{S}_{2n}}
Given $v\in\mathcal{S}_{2n}$, define
\[
\ket{v}=e_{\omega(v(1))}\otimes\cdots\otimes e_{\omega(v(2n))}
\]
where $\omega(i)=i$ for $1\le i\le n$ and $\omega(i)=n+1$ for $n+1\le i\le 2n$.
Note that $\ket{v}$ only depends on the class of $v$ inside $\cosets$; we define this way
a basis indexed by $\cosets$ of the submodule of $\M$ of interest to us.

Similarly, define dual basis elements
\[
\bra{v}=e^*_{\omega(v(1))}\otimes\cdots\otimes e^*_{\omega(v(2n))}
\]
in $\mathcal M^*$.
\rem[gray]{here I choose the convention that time goes SW. the convention of my Schubert papers, but doesn't match with permutations
of pipe dreams. $\ket{w}=w(1),w(2),\ldots,w(n),$blank} 

We have the easy lemma
\begin{lem}\label{lem:coset}
For all $w\in\mathcal{S}_{2n}$,
\begin{enumerate}[(a)]
\item
$\bra{1} T_w = t^{|w|} \bra{w}$.
\item $T_w\ket{1}=t^{|w_2|} \ket{w^{-1}}$ where $w_2$ is defined by $w^{-1}=w_1w_2$,
$w_1\in\mathcal{S}^{(2)}_n$ and $w_2$ a minimal representative in $\cosets$.
\end{enumerate}
\end{lem}
\begin{proof}
(a) Write $w=w_1w_2$ where $w_1\in\mathcal{S}^{(2)}_n$ and $w_2$ minimal representative in $\cosets$.
One has directly $\bra{1}T_{w_1}=t^{|w_1|}\bra{1}$. Then use a reduced decomposition of $w_2$ and note that
in \eqref{eq:rep}, one only ever uses the case $j<\ell$ which leads to $\bra{1}T_{w_2}=t^{|w_2|}\bra{w_2}$ inductively. At both stages, the only nonzero contributions come from the third line of \eqref{eq:rep}.
Case (b) is similar except the first (resp.\ second) stage involves the second (resp.\ third) line of \eqref{eq:rep}.

\end{proof}

\rem[gray]{we always choose $x_i$ for rows vs $y_j$ for columns}

Define, noting that $\mathcal S^{(1)}_n$ naturally sits inside $\cosets$,
\begin{equation}\label{eq:defF}
F_v(x_1,\ldots,x_n,y_1,\ldots,y_n) = 
\bra{1} Z(x_1,\ldots,x_n,y_1,\ldots,y_n)\ket{v}\qquad v\in\mathcal{S}^{(1)}_n
\end{equation}
This is a slightly more general definition than we need for the interpretation of the right hand side of \eqref{eq:deff}; \rem[gray]{because of $x$s and $y$s} indeed, we have:
\begin{prop}\label{prop:latticepath}
\[
\sum_{\substack{\text{lattice paths }P\\\text{on the }n\times n\text{ grid}\\\conn(P)=v}} \wt(P)
 = F_v(x_1,\ldots,x_n,q x_1,\ldots,q x_n)
\]
\end{prop}
\begin{proof}
Consider the diagrammatic representation \eqref{eq:defZ} of $Z$ and apply the representation described at the start of this section; this amounts to assigning to each edge a label in $\{1,\ldots,n+1\}$.
$F_v$ corresponds, in accordance with \eqref{eq:defF}, to further imposing that labels of external edges be fixed: $1,\ldots,n$ on the left side, $v(1),\ldots,v(n)$ on the top side, and $n+1$ on the right and bottom sides. Now identify labels $1,\ldots,n$ with the corresponding lattice paths, and $n+1$ with ``empty''. Then this produces precisely a lattice path with connectivity $P$. Finally, what needs to be checked is that the weight assigned to a lattice path coincides with the entry of $Z$. This can be done locally at the level of each individual crossing ($R$-matrix); and indeed, combining \eqref{eq:rep} with \eqref{eq:defR} and setting $y_i=qx_i$ leads to the weights in \eqref{eq:wt}.
\end{proof}

Finally, we can show that the definition \eqref{eq:deff2} given above is equivalent to the one \eqref{eq:deff} in the introduction:
\begin{lem}\label{lem:latticepath}
If $c=\sum_{v\in\mathcal S_n} c_v T_v \in \Z(\H_n)$, then $c_v=c_{v^{-1}}$ for all $v\in\mathcal S_n$ and
\[
f(c) =\alpha_n \sum_{v\in\mathcal S_n} c_{v} F_v(x_1,\ldots,x_n,qx_1,\ldots,qx_n)
\]
\end{lem}
\begin{proof}
First we observe that
\begin{equation}\label{eq:scalprod}
\left<x(1\otimes S_n)\right> = \bra{1}x\ket{1}\qquad\forall x\in \H_{2n}
\end{equation}
by testing this on $x=T_w$, $w\in\mathcal S_{2n}$, and noting that
both sides of the equality are $0$ unless $w\in\mathcal{S}_{2n}^{(2)}$,
in which case they are equal to $t^{|w|}$ (using the definition \eqref{eq:defS} for the left hand side and Lemma~\ref{lem:coset} for the right hand side).
We now start from the definition \eqref{eq:deff2} of $f(c)$ and apply the equality above, as well as Lemma~\ref{lem:coset}~(b):
\[
f(c) =\alpha_n \bra{1}Zc\ket{1} = \alpha_n \sum_{v\in\mathcal S_n} c_{v^{-1}} F_v(x_1,\ldots,x_n,qx_1,\ldots,qx_n)
\]
This is almost \eqref{eq:deff}, except a small simplification was made:
any central element satisfies $c_v=c_{v^{-1}}$
(short proof: it is obviously true for the generators $S_n$ of $\Z$, and $*$ preserves this property).
\end{proof}

We conclude this section by reformulating more explicitly the exchange relations of Lemma~\ref{lem:exch} in terms of the $F_v$. The next proposition will not be used in what follows, but is included to reconnect to the existing literature \cite{BW-coloured}.
\begin{prop}\label{prop:exch}
For $v\in\mathcal S_n^{(1)}$ and $i,j=1,\ldots,n-1$,
\begin{align*} (1-t x_{i+1}/x_{i})\tau_i F_v &=
\begin{cases} (1-t) F_v + (1-x_{i+1}/x_{i}) F_{s_i v}& s_i v > v \\
(1-t)(x_{i+1}/x_{i}) F_v + t(1-x_{i+1}/x_{i}) F_{s_i v}& s_i v < v
\end{cases} \\ (1-t y_j/y_{j+1})\tau'_j F_v &=
\begin{cases} (1-t) F_v + (1-y_j/y_{j+1}) F_{v s_j}& v s_j > v \\
(1-t)(y_j/y_{j+1}) F_v + t(1-y_j/y_{j+1}) F_{v s_j}& v s_j < v
\end{cases}
\end{align*}
\end{prop}
\begin{proof}
First, note that \eqref{eq:exp}, \eqref{eq:defF}, and Lemma~\ref{lem:coset}~(a) imply
\begin{equation}\label{eq:expF}
F_v = \sum_{w\in \mathcal{S}_n^{(2)} v} t^{|w|} Z_w
\end{equation}
where the summation is over the coset of $v$ in $\cosets$.

We then start from the identity \eqref{eq:exch-x} of Lemma~\ref{lem:exch}.

Using \eqref{eq:exp} and Lemma~\ref{lem:basic},
we obtain for the right hand side:
\[
\check R_i(x_{i+1}/x_i) Z = \sum_{w\in\mathcal{S}_{2n}} Z_w 
\begin{cases}
(1-t) T_w + (1-x_{i+1}/x_i) T_{s_iw}
& s_iw > w
\\
(1-t)(x_{i+1}/x_i) T_w + t(1-x_{i+1}/x_i) T_{s_iw}
& s_iw < w
\end{cases}
\]
We now take the bra-ket $\bra{1}\cdot \ket{v}$ and note that only $w$ such that either $w \in \mathcal{S}_n^{(2)} v$ or $w  \in \mathcal{S}_n^{(2)} s_iv$ contribute.
We get for $s_i v > v$
\begin{align*}
\bra{1}
\check R_i(x_{i+1}/x_i) Z
\ket{v}
&=
(1-t)\sum_{w  \in \mathcal{S}_n^{(2)} v} t^{|w|}Z_w + t(1-x_{i+1}/x_i)\sum_{w \in \mathcal{S}_n^{(2)} s_i v} t^{|s_i w|} Z_w
\\
&=
(1-t) F_v + (1-x_{i+1}/x_i) F_{s_i v} 
\end{align*}
and for $s_i v<v$
\begin{align*}
\bra{1}
\check R_i(x_{i+1}/x_i) Z
\ket{v}
&=
(1-t)(x_{i+1}/x_i)\sum_{w \in \mathcal{S}_n^{(2)} v} t^{|w|}Z_w + (1-x_{i+1}/x_i)\sum_{w  \in \mathcal{S}_n^{(2)} s_i v} t^{|s_i w|} Z_w
\\
&=
(1-t)(x_{i+1}/x_i) F_v + t(1-x_{i+1}/x_i) F_{s_i v} 
\end{align*}

The bra-ket $\bra{1}\cdot\ket{v}$ of the left hand side is simple to evaluate: we note that $\check R_{n+i}(x_{i+1}/x_i)\ket{v}
= (1-t x_{i+1}/x_i)\ket{v}$ because the $i^{\rm th}$ and $(i+1)^{\rm th}$ factors of $\ket{v}$ are $e_{n+1}\otimes e_{n+1}$;
therefore
\[
\bra{1} (\tau_i Z) \check R_{n+i}(x_{i+1}/x_i)\ket{v}=
(1-t x_{i+1}/x_i) \tau_i F_v
\]
This proves the first identity of the Proposition.

The second identity is actually simpler to prove (there isn't a full symmetry between rows and columns).
Start analogously from \eqref{eq:exch-y} of Lemma~\ref{lem:exch}.
Now apply the bra-ket $\bra{1}\cdot\ket{v}$ directly. One has $\bra{1}\check R_{n+j}(y_j/y_{j+1})=(1-ty_j/y_{j+1})\bra{1}$ 
by the same argument as above;
but one also has from \eqref{eq:rep}
\[
\check R_j(y_j/y_{j+1})
\ket{v} =
\begin{cases}
(1-t) \ket{v} + (1-y_j/y_{j+1}) \ket{vs_j}& vs_j > v
\\
(y_j/y_{j+1})(1-t)\ket{v}+t(1-y_j/y_{j+1})\ket{vs_j}& vs_j < v
\end{cases}
\]
which leads to the desired identity.
\end{proof}
\rem[gray]{$K_w=(-1)^{n^2-|w|} q_1^{n(n+1)/2}q_2^{n^2} (1-q_1q_2)^{-n} \prod_{i=1}^n (y_i/x_i)^{i-1} F_w$ (where $K_w$ -- $K$-polynomial of component of lower-upper scheme -- is a generalization of $K_n$ to arbitrary permutations and distinct alphabets; reduces to $K_n$ when $w=1$ and $x_i=y_i$).
the annoying sign is because geometrically we'd want to change the sign of $a$ and $c$ weights.}

These relations can be used to determine the $F_v$ inductively. More precisely, consider first the computation of $F_{w_0}$ where $w_0$ is the longest element of $\mathcal S_n^{(1)}$. There is a unique lattice path with such connectivity, of the form
\[
\begin{tikzpicture}[scale=0.75]\draw[dotted] (0.5,0.5) grid (4.5,-3.5);\begin{scope}[very thick,shift={(0.5,0.5)},scale=1cm]\draw[red!80!black] svg {M 0 -.5 L 1 -.5 L 2 -.5 L 3 -.5 L 3.5 -.5 3.5 0}; \draw[green!60!black] svg {M 0 -1.5 L 1 -1.5 L 2 -1.5 L 2.5 -1.5 2.5 -1 L 2.5 0}; \draw[blue] svg {M 0 -2.5 L 1 -2.5 L 1.5 -2.5 1.5 -2 L 1.5 -1 L 1.5 0}; \draw[yellow!90!black] svg {M 0 -3.5 L .5 -3.5 .5 -3 L .5 -2 L .5 -1 L .5 0}; \end{scope}\end{tikzpicture} 
\]
from which we conclude
\begin{equation}\label{eq:w0}
F_{w_0}=(1-t)^n t^{\frac{n(n-1)}{2}} \prod_{\substack{1\le i,j\le n\\ i+j<n+1}}(1-x_i/y_j) \prod_{\substack{1\le i,j\le n\\ i+j>n+1}} (1-tx_i/y_j)
\end{equation}
One can then apply {\em either}\/ the second relation or the fourth relation of Proposition~\ref{prop:exch} to define
the $F_w$ inductively.

\section{The shuffle algebra}\label{sec:shuffle}
This section is devoted to the trigonometric Feigin--Odesskii shuffle algebra \cite{FO}. It was studied notably in \cite{FHHSY,FT-shuffle,SV-Hall,Ng}. We will review some known results and derive several formulas which are required for proving Theorems \ref{thm:square}-\ref{thm:mainK}. We will follow the conventions of \cite{FHHSY} and define the shuffle algebra  $\mathcal{A}$ which is slightly different from the algebra $\A$ in \S \ref{sec:intro-center}. The precise connection between $\mathcal{A}$ and $\A$ is given in Remark \ref{rmk:A}. For the treatment of the shuffle algebra $\mathcal{A}$ it is convenient to use three parameters $q_1,q_2,q_3$, expressed via $q$ and $t$ as
\begin{align}
    q_1= q^{-1} \qquad
    q_2 = q t^{-1}  \qquad
    q_3 = (q_1 q_2)^{-1} = t
\end{align}

Let us outline the goals of this section:
\begin{itemize}
    \item Discuss the shuffle algebra $\mathcal{A}^+$, its commutative subalgebra $\mathcal{A}$ and the isomorphism $\Upsilon$ of Theorem \ref{thm:square}
    \item Give a characterization of the distinguished basis elements $\epsilon_\lambda$ of the shuffle algebra $\mathcal{A}$ via specializations which will be used in \S \ref{sec:proof} to prove Theorem \ref{thm:square}
    \item Introduce a special shuffle element $\kappa_n$ which will be used  in \S \ref{sec:proof} to prove Theorem \ref{thm:f1n} and in \S \ref{sec:proof_mainK} to prove Theorem \ref{thm:mainK}
\end{itemize}

\subsection{The shuffle algebra $\mathcal{A}^+$}
The shuffle algebra $\mathcal{A}^+$ is a vector space whose elements are symmetric rational functions of the form
\begin{align}
\label{eq:A_element}
    P(x_1,\dots ,x_n) =
    \frac{p(x_1,\dots ,x_n)}{\prod_{1\leq i<j\leq n}(x_i-x_j)^2}
    \qquad 
    p(x_1,\dots ,x_n) \in \mathbb{F}[x_1^{\pm 1},\dots , x_n^{\pm 1}]^{\mathcal{S}_n}
\end{align}
where $p(x_1,\dots ,x_n)$ satisfies the {\em wheel condition}
\begin{align}
\label{eq:p_wheel}
p(x_1,\dots,x_n) = 0
\quad \text{if} \quad 
(x_i,x_j,x_k) = (x,q_1 x, q_1 q_2 x) 
\quad \text{or} \quad
(x_i,x_j,x_k) = (x,q_2 x, q_1 q_2 x)
\end{align}
The number of arguments $(x_1,\dots,x_n)$ gives a grading
\begin{align*}
    \mathcal{A}^+ = \bigoplus_{n\geq 0} \mathcal{A}_{n}^+
\end{align*}
We set $\mathcal{A}_{0}^+ = \mathbb{F}$ and $\mathcal{A}_{1}^+ = \mathbb{F}[x^{\pm 1}]$. 
The shuffle algebra $\mathcal{A}^+$ is closed under the {\it shuffle product} $*$, defined for two elements $F \in \mathcal{A}_k^+$ and $G \in \mathcal{A}_l^+$ by
\begin{align}
\label{eq:shuffle_prod}
    (F * G)(x_1,\dots, x_{k+l}) 
    = \frac{1}{k!l!}\text{Sym}\left[ 
    F(x_1,\dots, x_{k}) G(x_{k+1},\dots, x_{k+l}) 
    \prod_{\substack{ 1\leq i\leq k\\ 1\leq j\leq l}}
    \omega(x_{k+j},x_i)
    \right]
\end{align}
where $\text{Sym}$ is the symmetrization over all arguments $x_i$
\begin{align}
    \label{eq:Sym}
    \text{Sym} \left[ f(x_1,\dots, x_{n})\right]
    =  \sum_{w \in \mathcal{S}_n} f(x_{w(1)},\dots, x_{w(n)})
\end{align}
and $\omega(x,y)$ is defined as follows 
\begin{align}
    \label{eq:omega}
    \omega(x,y) = \frac{(x-q_1 y)(x-q_2 y)(x-q_3 y)}{(x-y)^3}
\end{align}
It was shown in \cite{Ng} that the algebra $\mathcal{A}^+$ is generated by the elements $x^i \in \mathcal{A}_1^+$ for $i\in \mathbb{Z}$.

\subsection{The commutative subalgebra $\mathcal{A}$}\label{sec:commut_A}
Consider a subalgebra $\mathcal{A}\subset \mathcal{A}^+$ of the elements $P\in \mathcal{A}_n^+$ for which the two limits
\begin{align}
\label{eq:limit1}
    &\lim_{\xi \rightarrow 0}P(\xi x_1,\dots ,\xi x_{r},  x_{r+1}, \dots , x_{n}),\\
\label{eq:limit2}
    &\lim_{\xi \rightarrow \infty}
    P(\xi x_1,\dots ,\xi x_{r},  x_{r+1}, \dots , x_{n})
\end{align}
exist and coincide for all fixed $r=1,\dots ,n$. This subalgebra splits into components of fixed degree in the same way as $\mathcal{A}^+$
\begin{align*}
    \mathcal{A} = \bigoplus_{n\in \mathbb{Z}_{\ge 0}} \mathcal{A}_{n}
\end{align*}
\begin{rmk}\label{rmk:A}
The shuffle algebras $\mathcal{A}_n$ and  $\A_n$ are isomorphic. For $P(x_1,\dots,x_n)\in \mathcal{A}_n$ and $Q(x_1,\dots,x_n)\in \A_n$, s.t. $Q\mapsto P$, we have
\begin{align}
    \label{eq:AA}
    P(x_1,\dots,x_n) = \frac{1}{V_n(x_1,\dots,x_n)}Q(x_1,\dots,x_n)
\end{align}
where $V_n=V_n(x_1,\dots,x_n)$ is defined by
\begin{equation}\label{eq:defV}
    V_n := \prod_{1\leq i\neq j\leq n} (1-x_i/x_j)
\end{equation}
\end{rmk}
\begin{proof}
The product $V_n^{-1}Q(x_1,\dots,x_n)$ is of the form (\ref{eq:A_element}) and the wheel conditions of $\mathcal{A}_n^+$ and $\A_n$ coincide. The fact that the limiting conditions (\ref{eq:limit1}) and (\ref{eq:limit2}) of $\mathcal{A}_n$ are equivalent to the degree constraints (\ref{eq:degbounds}) of $\A_n$ was shown in \cite{FHHSY}.
\end{proof}

Another important result of \cite{FHHSY} is the following proposition.
\begin{prop}
The algebra $(\mathcal{A}, *)$ is commutative and the dimension of the graded subspace $\mathcal{A}_n$ is equal to the number of partitions of $n$.
\end{prop}
There are several distinguished bases in the space $\mathcal{A}$. These bases are built using the shuffle elements
\begin{align}
\label{eq:epsilons}
    \epsilon_n(x_1,\dots,x_n;q_k)
    =\prod_{1\leq i<j\leq n} \frac{(x_i -q_k x_j)(x_i -q_k^{-1} x_j)}{(x_i -x_j)^2}
    \qquad k=1,2,3
\end{align}
In the following we will abbreviate them as $\epsilon_n(q_k)=\epsilon_n(x_1,\dots,x_n;q_k)$. 
Due to commutativity the products $\epsilon_i(q_k)\, *\, \epsilon_j(q_k)$ can be ordered such that $i>j$. Thus for a partition $\lambda$ the elements 
\begin{align}
\label{eq:epsilon_basis}
    \epsilon_\lambda(q_k)
    =\epsilon_{\lambda_1}(q_k) *
    \epsilon_{\lambda_2}(q_k) *
    \dots
    * \epsilon_{\lambda_{\ell(\lambda)}}(q_k)
\end{align}
give three bases, one for each $k=1,2,3$.

\subsection{Macdonald operators and the isomorphism $\Upsilon$}\label{sec:macdo}
Define the Heisenberg algebra $\mathfrak{h}$ with generators $a_{\pm n}$, $n>0$, satisfying the relations
\begin{align}
    \label{eq:Heisenberg}
    [a_m, a_n] = \delta_{m,-n} m \frac{1-q^{|m|}}{1-t^{|m|}}
\end{align}
We have two commutative subalgebras generated by $\{a_{\pm n}\}_{n>0}$ and containing the elements  $a_{\pm \lambda}= a_{\pm \lambda_1}\dots a_{\pm  \lambda_{\ell(\lambda)}}$. 
Define the Fock space $\mathcal{F}$ by introducing the vacuum vector $\ket{\emptyset}$ which is annihilated by the positive modes $a_n \ket{\emptyset} = 0$, for $n>0$. The action of the negative modes creates a basis
\begin{align}
\ket{a_\lambda} = a_{-\lambda} \ket{\emptyset}
\end{align}
for all partitions $\lambda$. 

Let $x=(x_1,x_2,\dots)$ be the alphabet for the ring of symmetric functions $\Lambda$ with the base field $\mathbb{F}$. Recall the power sum basis of $\Lambda$
\begin{align}
    \label{eq:psums}
    p_{\lambda} = p_{\lambda_1}
    \dots p_{\lambda_{\ell(\lambda)}}
    \qquad
    p_{n} = x_1^n +x_2^n + \dots
\end{align}
The Fock space $\mathcal{F}$ is isomorphic to $\Lambda$ where the isomorphism is explicitly given by
\begin{align*}
    \ket{a_{\lambda}} \mapsto p_\lambda
\end{align*}
In  $\Lambda$ a symmetric function is considered in the power sum basis and the above isomorphism replaces it with a vector in the Fock space. A distinguished basis in $\Lambda$ is given by the Macdonald symmetric functions $P_\lambda$ (see \cite{Macdonald2}). We view it as a basis in the Fock space denoted by $\ket{P_\lambda}$. 

The Macdonald operators are defined as the operators which act diagonally on  $P_\lambda$. The {\it first} Macdonald operator $E$ \cite[Ch. \textrm{VI}]{Macdonald2} acts as follows
\begin{align}
    E \ket{P_\lambda}= \sum_{i\geq 1} (q^{\lambda_i}-1) t^{-i} \ket{P_\lambda}
\end{align}
This operator is realized on the Fock space with the help of the vertex operator \cite{Sh}
\begin{align}
    \eta(z) : = \exp \left( \sum_{n>0} 
    \frac{1-t^{-n}}{n} a_{-n} z^n \right)
    \exp \left(-  \sum_{n>0} 
    \frac{1-t^{n}}{n} a_{n} z^{-n} \right)
\end{align}
The Fourier modes of this operator $\eta_n$ ($n\in \mathbb{Z}$) are given by
\begin{align*}
    \eta(z) = \sum_{n\in \mathbb{Z}} \eta_n z^{-n}
\end{align*}
and have a well defined action on $\mathcal{F}$. The operator $E$ is given by the zero mode \cite{Sh}
\begin{align}
    \eta_0 = (t-1) E + 1
\end{align}
We can write the eigenvalue equation as follows 
\begin{align*}
\frac{t^{-1}}{1-t^{-1}}\eta_0 \ket{P_\lambda}= 
u_\lambda(e_1)\ket{P_\lambda}
\end{align*}
where 
$$u_\lambda(x_i)= q^{\lambda_i}t^{-i}
$$ 
and $e_1$ is the first elementary symmetric function in the infinite alphabet $e_1 = x_1 +x_2 + \dots$. One can build higher order operators in a systematic way using the algebra $\mathcal{A}$ and the operators $\eta$. 

The product of Heisenberg operators is {\it normally ordered} if all positive modes are on the right and negative modes on the left. The product $\eta(z) \eta(w)$ can be ordered at the expense of a rational function 
\begin{align}
\eta(z) \eta(w)     
=
\frac{(1-w/z)(1-q t^{-1} w/z)}{(1-q w/z)(1-t^{-1}w/z)}
:\eta(z) \eta(w):
\end{align}
where the operators inside $::$ are normally ordered. 
For $f\in \mathcal{A}_n$ define the mapping $\mathcal{O}: \mathcal{A}\rightarrow \text{End}_{\mathbb{F}}(\mathcal{F})$ by
\begin{align}
    \label{eq:Omap0}
    \mathcal{O}(f)=\frac{1}{(t-1)^n n!}\oint \frac{\text{d}z_1 \cdots \text{d}z_n}{z_1 \cdots z_n}
    \frac{f(z_1,\dots, z_n)}{\epsilon_n(z_1,\dots,z_n;q_1)\epsilon_n(z_1,\dots,z_n;q_3)}
    :\eta(z_1)\dots \eta(z_n):
\end{align}
where the integration contours are given by $|z_1|=\dots = |z_n|=1$ and $|t|, |q| <1$. It was shown in \cite{FHHSY} that $\mathcal{O}$ is compatible with the shuffle product, i.e.
\begin{align}
    \label{eq:Ofg}
    \mathcal{O}(f * g)=\mathcal{O}(f)\mathcal{O}(g)
\end{align}
Since $\mathcal{A}$ is commutative this implies the commutativity of the operators
$[\mathcal{O}(f),\mathcal{O}(g)]=0$ for any two elements $f,g\in \mathcal{A}$. The map $\mathcal{O}$ was used in \cite{FHHSY} to define an isomorphism $\Upsilon': \mathcal{A}\rightarrow \Lambda$: 
\begin{prop}\label{prop:epsP}
Given $f\in \mathcal A_n$, there exists an $S\in \Lambda_n$ such that for any partition $\lambda$,
\begin{align}
\label{eq:Omap}
     \mathcal{O}(f)\ket{P_\lambda}&=
     u_\lambda(S)
    \ket{P_\lambda}
\end{align}
and $\Upsilon': f\mapsto S$ thus defined is a ring isomorphism from $\mathcal A$ to $\Lambda$.
\end{prop}
 By inserting a $V_n^{-1}$ in (\ref{eq:Omap}), cf Remark~\ref{rmk:A}, we obtain the isomorphism $\Upsilon$ from $\A$ to $\Lambda$ that is stated in the introduction and used in Theorem~\ref{thm:square}. 
 The symmetric functions in $\Lambda$ which correspond to the shuffle elements $\epsilon_n(q_k)$, $k=1,2,3$ are all given by a plethystic substitution applied to the elementary symmetric functions $e_n$. Recalling \eqref{eq:defsigma} we set 
 \begin{align}
\label{eq:esp}
    e_n^{(1)} = e_n
    \qquad
    e_n^{(2)} = \sigma_{q^{-1}}\sigma^{-1}_{qt^{-1}}(e_n)
    \qquad
    e_n^{(3)} =\sigma_{q^{-1}}\sigma_t^{-1}(e_n)
    \qquad
\end{align}
Note that the plethystic substitution in $e_n^{(3)}$ appears in the statement of Theorem~\ref{thm:square}. This element will play a special role. We have \cite{FHHSY}
\begin{equation}\label{eq:Ueps}
\Upsilon'(\epsilon_n(q_k))=e^{(k)}_n\qquad k=1,2,3
\end{equation}
or equivalently
\begin{align}
\label{eq:epsP}
     \mathcal{O}(\epsilon_n(q_k))\ket{P_\lambda}&=
     u_\lambda(e^{(k)}_n)
    \ket{P_\lambda}
    \qquad k=1,2,3
\end{align}
The eigenvalue equation (\ref{eq:epsP}) for $k=1$ shows that the operator $\mathcal{O}(\epsilon_n(q_1))$ coincides with the $n^{\rm th}$ Macdonald operator \cite{Macdonald2}. Thus one shows (\ref{eq:epsP}) by matching the Macdonald difference operators $E_n$ with $\mathcal{O}(\epsilon_n(q_1))$. Based on this result one can show the other two eigenvalue equations for $k=2,3$ using the Wronski relations
\begin{align}
\label{eq:Wronski}
    \sum_{j=0}^n (q_k^j-q_l^{n-j})\frac{(1-q_l)^j}{(1-q_k)^j}
    \epsilon_{n-j}(q_k) * \epsilon_{j}(q_l) 
    =0
    \qquad
    k,l=1,2,3
\end{align}
In \cite{FHHSY} the authors proved these Wronski relations. Applying $\mathcal{O}$ to (\ref{eq:Wronski}) with $l=1,k=2,3$ gives a relation between the eigenvalues of $\mathcal{O}(\epsilon_n(q_1))$ and $\mathcal{O}(\epsilon_n(q_k))$ which determines the eigenvalue in (\ref{eq:epsP}) for $k=2,3$.

We finish this section by giving a lemma which will be used at a later stage.
\begin{lem}
\label{lem:recs}
The following identities hold 
\begin{align}
    \label{eq:epsilon_rec}
        \left(x_1+\dots + x_n \right) \epsilon_n(q_k) &= \frac{1}{1-q_k}
    \left(
    x_1 * \epsilon_{n-1}(q_k) - q_k  \epsilon_{n-1}(q_k) * x_1
    \right)
    \qquad k=1,2,3
\end{align}
\end{lem}
\begin{proof}
For simplicity we fix $k=1$. Let us denote the right hand side of (\ref{eq:epsilon_rec}) by $C_n$ and rewrite it using the definition of the shuffle product (\ref{eq:shuffle_prod})
\begin{align*}
C_n=
&\frac{1}{1-q_1}
 \sum_{i=1}^n 
x_i  \epsilon_{n-1}(q_1)[\hat{x_i}]\\
&\left(
\prod_{\substack{j=1 \\ j\neq i}}^n \frac{(x_j-q_1 x_i)(x_j- q_2 x_i)
(x_j- q_3 x_i)}{(x_j-x_i)^3}
-
q_1
\prod_{\substack{j=1 \\ j\neq i}}^n \frac{(x_i-q_1 x_j)(x_i- q_2 x_j)
(x_i- q_3 x_j)}{(x_i-x_j)^3}
\right)
\end{align*}
where $\epsilon_{n-1}(q_1)[\hat{x_i}]$ depends on the $n-1$ arguments $(x_1,\dots,x_{i-1},x_{i+1},\dots, x_n)$. 
Noting that 
\begin{align*}
    \epsilon_{n-1}(q_1)[\hat{x_i}]
    \prod_{\substack{j=1 \\ j\neq i}}^n \frac{(x_j-q_1 x_i)}{(x_j-x_i)}
    &=\epsilon_{n}(q_1)
    \prod_{\substack{j=1 \\ j\neq i}}^n \frac{(x_j-x_i)}{(x_j-q_1^{-1} x_i)}
    \\
    \epsilon_{n-1}(q_1)[\hat{x_i}]
    \prod_{\substack{j=1 \\ j\neq i}}^n \frac{(x_i-q_1 x_j)}{(x_i-x_j)}
    &=\epsilon_{n}(q_1)
    \prod_{\substack{j=1 \\ j\neq i}}^n \frac{(x_i-x_j)}{(x_i-q_1^{-1} x_j)}
\end{align*}
leads to 
\begin{align*}
C_n=
\epsilon_{n}(q_1)
\frac{1}{1-q_1}
 \sum_{i=1}^n 
x_i
\left(
\prod_{\substack{j=1 \\ j\neq i}}^n \frac{(x_j- q_2 x_i)
(x_j- q_3 x_i)}{(x_j-q_1^{-1} x_i)(x_j-x_i)}
-
q_1
\prod_{\substack{j=1 \\ j\neq i}}^n \frac{(x_i- q_2 x_j)
(x_i- q_3 x_j)}{(x_i-q_1^{-1} x_j)(x_i-x_j)}
\right)
\end{align*}
The factor next to $\epsilon_n(q_1)$ does not have poles at $x_i=x_j$ or $x_i=q_1^{-1} x_j$ because of the symmetry in $x_i$ and vanishing residues at $x_i=q_1^{-1} x_j$ which can be easily verified. It means that this factor must be a symmetric polynomial of degree $1$. Thus it is proportional to the elementary symmetric polynomial $e_1$ in $n$ arguments
\begin{align*}
\frac{1}{1-q_1}
 \sum_{i=1}^n 
x_i 
\left(
\prod_{\substack{j=1 \\ j\neq i}}^n \frac{(x_i- q_2 x_j)
(x_i- q_3 x_j)}{(x_i-q_1^{-1} x_j)(x_i-x_j)}
-
q_1\prod_{\substack{j=1 \\ j\neq i}}^n \frac{(x_j- q_2 x_i)
(x_j- q_3 x_i)}{(x_j-q_1^{-1} x_i)(x_j-x_i)}
\right)=C \sum_{i=1}^n  x_i
\end{align*}
We need to show that the proportionality factor $C=1$. This can be verified by setting $x_i=0$ for  $i>1$ in the above equation. This completes the proof of (\ref{eq:epsilon_rec}).
\end{proof}

\subsection{Shuffle elements of $\mathcal{A}$ and specializations}
In this section we look at the problem of determining the shuffle algebra elements of $\mathcal{A}$ recursively via specializations of the arguments $x_i$. This problem was  investigated in \cite{Ng,Ng-Pieri} using a coproduct map of $\mathcal{A}^+$. Our treatment of this problem is appropriate for the goal of matching shuffle elements with partition functions in \S \ref{sec:proof}. 

Recall that the shuffle elements of $\mathcal{A}$ have the form 
\begin{align}
\label{eq:A_element_X}
    P(x_1,\dots ,x_n) =
    \frac{p(x_1,\dots ,x_n)}{\prod_{1\leq i<j\leq n}(x_i-x_j)^2}
\end{align}
where $p(x_1,\dots ,x_n)$ is a symmetric polynomial. We start with a lemma which gives us information about the degree of $p(x_1,\dots ,x_n)$.

\begin{lem}\label{lem:degA}
Let $P(x_1,\dots,x_n)$ be an element of $\mathcal{A}_n$ as in (\ref{eq:A_element_X}). Let $x$ be an indeterminate and $c_1,\dots,c_r\in \mathbb{F}$ then $p(c_1 x ,\dots,c_r x ,x_{r+1},\dots,x_n)\in \mathbb{F}[x,x_{r+1},\dots,x_n]$ is of the form 
\begin{align}
\label{eq:pspec}
x^{r(r-1)} G(x,x_{r+1},\ldots,x_n;c_1,\dots,c_r)
\end{align}
and  $G(x,x_{r+1},\ldots,x_n;c_1,\dots,c_r)$ is a polynomial of degree at most $2r(n-r)$ in $x$.
\end{lem}
\begin{proof}
Consider the specialization $x_k=c_k x$ $(k=1,\dots,r)$ of  (\ref{eq:A_element_X})
\begin{multline*}
P(c_1 x ,\dots,c_r x ,x_{r+1},\dots,x_n)\\
=
    \frac{
    p(c_1 x ,\dots,c_r x ,x_{r+1},\dots,x_n)
    }{x^{r(r-1)}
    \prod_{i=1}^r\prod_{j=r+1}^n(c_i x -x_j)^2
    \prod_{1\leq i<j\leq r}(c_i-c_j)^2
    \prod_{r+1\leq i<j\leq n}(x_i-x_j)^2}
\end{multline*}
The existence of the limit (\ref{eq:limit1}) implies that, for $x\rightarrow 0$, the polynomial in the numerator $p(c_1 x ,\dots,c_r x ,x_{r+1},\dots,x_n)$ contains an overall factor $x^{r(r-1)}$ to compensate the same factor in the denominator. Thus we established the form (\ref{eq:pspec}) and
\begin{multline*}
P(c_1 x ,\dots,c_r x ,x_{r+1},\dots,x_n)\\
=
    \frac{
    G(x,x_{r+1},\dots,x_n;c_1,\dots,c_r)
    }{
    \prod_{i=1}^r\prod_{j=r+1}^n(c_i x -x_j)^2
    \prod_{1\leq i<j\leq r}(c_i-c_j)^2
    \prod_{r+1\leq i<j\leq n}(x_i-x_j)^2}
\end{multline*}
The existence of the limit (\ref{eq:limit2}) implies that, for $x\rightarrow \infty$, the degree in $x$ of the polynomial $G(x ,x_{r+1},\dots,x_n)$ is controlled by $\prod_{i=1}^r\prod_{j=r+1}^n(c_i x -x_j)^2$, i.e. bounded by $2r(n-r)$. 
\end{proof}
In the following we will only be interested in the situations when $c_i=q_k^{i-1}$. 
The next statement gives a recipe for a recursive computation of shuffle elements of $\mathcal{A}$.
\begin{prop}\label{prop:P_specs}
Let $r\in\{1,\ldots,n\}$, and $P(x_1,\dots,x_n)$ be an element of $\mathcal{A}_n$ as in (\ref{eq:A_element_X}). Fix $k\in\{1,2,3\}$, then $P(x_1,\dots,x_n)$ is determined by the following specializations:
\begin{enumerate}
    \item[\normalfont{$(1.r)$}] $\lim_{x_i\to0}\cdots\lim_{x_1\to0} P(x_1,\ldots,x_i,x_{i+1},\ldots,x_n)$, $i=1,\ldots,r-1$.
    \item[\normalfont{$(2.r)$}] $P(x,q_k x,\ldots,q_k^{r-1}x,x_{r+1},\ldots,x_n)$.
\end{enumerate}
\end{prop}
\begin{proof}
We fix $k',k''\in \{1,2,3\}$ such that $k,k',k''$ are all distinct.
The proof is by induction on $r$. The case $r=1$ is trivial.
Now assume the property true at $r\ge1$, i.e. knowing $(1.r)$ and $(2.r)$ allows us to compute $P(x_1,\dots,x_n)$. Let us show that the specializations $(1.r+1)$ and $(2.r+1)$ allow us to compute the specializations $(1.r)$ and $(2.r)$. Since the specializations $(1.r)$ are a subset of the specializations $(1.r+1)$ we only need to focus on computing $(2.r)$
$$
P(x,q_k x,\ldots,q_k^{r-1}x,x_{r+1},\ldots,x_n)
$$
According to Lemma~\ref{lem:degA} with $c_i=q_k^{i-1}$ for $1\leq i\leq r$ we have
\begin{multline*}
P( x ,q_k x,\dots,q_k^{r-1} x ,x_{r+1},\dots,x_n)\\
=
    \frac{
    G(x,x_{r+1},\dots,x_n)
    }{
    \prod_{i=1}^r\prod_{j=r+1}^n(q_k^{i-1} x -x_j)^2
    \prod_{1\leq i<j\leq r}(q_k^{i-1}-q_k^{j-1})^2
    \prod_{r+1\leq i<j\leq n}(x_i-x_j)^2}
\end{multline*}
where $G$ is a polynomial of degree $2r(n-r)$ in $x$. 
According to the wheel conditions (\ref{eq:p_wheel}), it has the following factorization:
\[
G(x,x_{r+1},\ldots,x_n)=\prod_{i=0}^{r-2} \prod_{j=r+1}^{n}
(q_{k'} x_j- q_k^i x)(q_{k''} x_j-q_k^{i} x)
H(x,x_r,\ldots,x_n)
\]
with the degree of $H$ in $x$ being $2r(n-r)-2(r-1)(n-r)=2(n-r)$.
To determine $H$, we therefore need to know its values at $2(n-r)+1$ distinct
specializations of $x$. The value of $H$ at $x=0$ gives one specialization and can be computed from the knowledge of
$$
\lim_{x_r\to0}\cdots\lim_{x_1\to0} P(x_1,\ldots,x_r,x_{r+1},\ldots,x_n)
$$
The specialization (2), using the fact that $P$ is a symmetric polynomial, is equivalent to computing $H$ at $x=q_k^{-r}x_j$ and $x=q_k x_j$, $j=r+1,\ldots,n$.
This is the required number of specializations, which means $G$ is determined and with it $P(x,q_k x,\ldots,q_k^{r-1}x,x_{r+1},\ldots,x_n)$.
\end{proof}
Let us apply this result to the computation of the basis elements $\epsilon_\lambda(q_k)$ with $\ell(\lambda)=m$ and $\lambda_1+\dots+\lambda_m=n$. For convenience we only consider the case $k=3$ and denote by
$$
\epsilon_\lambda=\epsilon_\lambda(q_3)=\epsilon_\lambda(t)
$$
We apply Proposition \ref{prop:P_specs} with $r=n$ and $k=3$ to the problem of computing $\epsilon_\lambda$. This means that in order to determine $\epsilon_\lambda$ we need to know the values of
\begin{align}
\label{eq:eps_cond_0}
    &\lim_{x_i\to0}\cdots\lim_{x_1\to0} \epsilon_{\lambda}(x_1,\ldots,x_i,x_{i+1},\ldots,x_n),
    \qquad i=1,\dots,n-1\\
\label{eq:eps_cond_t}
    &\epsilon_{\lambda}(x,t x,\ldots,t^{n-1}x)
\end{align}
One can compute the limits in (\ref{eq:eps_cond_0}) using the following lemma.
\begin{lem}\label{lem:epsrec}
The elements $\epsilon_{\lambda_1,\dots,\lambda_m}$, with $\lambda_1+\dots+\lambda_m=n$, satisfy the recurrence relation
\begin{align}
\label{eq:epskk_rec}
    \lim_{x_n\rightarrow 0}\epsilon_{\lambda_1,\dots,\lambda_m}(x_1,\dots,x_{n-1},x_n) = 
    \sum_{i=1}^{m} 
    \epsilon_{\lambda_1,\dots,\lambda_i-1,\dots,\lambda_m}(x_1,\dots,x_{n-1})
\end{align}
\end{lem}
\begin{proof}
We start by writing $\epsilon_{\lambda_1,\dots,\lambda_m}$ explicitly as a shuffle product of $\epsilon_{\lambda_l}$
\begin{align}\label{eq:epsk1km}
    \epsilon_{\lambda_1,\dots,\lambda_m} =
    &\sum_{w\in\mathcal{S}_{n}/\mathcal{S}_{\lambda_1,\dots,\lambda_m}}
    \prod_{l=1}^m\epsilon_{\lambda_l}(x_{w(1+\sum_{i=1}^{l-1}\lambda_i)},\dots,x_{w(\sum_{i=1}^{l}\lambda_i)})
    \nonumber\\
    &\prod_{1\leq a<b\leq m}\prod_{i\in I_a,
    j\in I_b}
    \omega(x_{w(j)},x_{w(i)})
\end{align}
where we defined the subsets $I_a$ by 
\begin{equation}\label{eq:defI}
    I_a := \left\{\sum_{s=1}^{a-1} \lambda_s +1,\dots,\sum_{s=1}^a \lambda_s \right\}
\end{equation}
From definitions (\ref{eq:epsilons}) and (\ref{eq:omega}) we have two obvious identities 
\begin{align*}
    \epsilon_{k}(x_1,\dots,x_{k-1},0)=
\epsilon_{k-1}(x_1,\dots,x_{k-1})
\qquad
\omega(x,0)=\omega(0,x)=1
\end{align*}
Fix $r\in \{1,\ldots,m\}$ and consider the summands of (\ref{eq:epsk1km}) with $w$ such that $n\in w(I_r)$. The factor $\epsilon_{\lambda_r}$ will be replaced by $\epsilon_{\lambda_r-1}$ upon setting $x_n=0$ while the factors of $\omega$ containing $x_n$ will disappear. The summation over $w\in\mathcal{S}_{n}/\mathcal{S}_{\lambda_1,\ldots,\lambda_m}$ (with $n\in w(I_r)$) reduces to the summation over $w'\in\mathcal{S}_{n-1}/\mathcal{S}_{\lambda_1,\ldots,\lambda_r-1,\dots,\lambda_m}$ which is equal to the shuffle product of $\epsilon_{\lambda_1},\ldots, \epsilon_{\lambda_r-1},\dots, \epsilon_{\lambda_m}$ similarly to (\ref{eq:epsk1km}). Treating in this way every summand with $n\in w(I_r)$ for $r\in \{1,\ldots,m\}$ gives rise to the formula (\ref{eq:epskk_rec}).
\end{proof}
It is clear that computing (\ref{eq:eps_cond_0}) can be done inductively using (\ref{eq:epskk_rec}).
The specializations in (\ref{eq:eps_cond_t}) are not difficult to compute and in fact follow from more general specializations considered in \cite{FHHSY}. We have
\begin{align}
\label{eq:epsA_0}
    &\epsilon_{\lambda}(x,t x,\ldots,t^{n-1}x) =0,\qquad \lambda\neq (1,1,\dots,1)\\
\label{eq:epsB_0}
    &\epsilon_{(1,1,\dots,1)}(x,t x,\ldots,t^{n-1}x) =
    \prod_{1\leq i<j\leq n}
    \frac{(q t^i-t^j)(t^{i+1}-q t^j)(t^i-t^{j+1})}{q t(t^i-t^j)^3}
\end{align}
Let us explain (\ref{eq:epsB_0}). We use the explicit expression (\ref{eq:epsk1km}) with $\lambda_i=1$ in which case all $\epsilon_{\lambda_i}=1$ and all sets $I_a$ have a single element. Due to the vanishing of  $\omega(tx,x)=0$ only one term in the summation is non-vanishing. This term must have $x$'s with fully ordered  indices as follows
\begin{align*}
    \prod_{1\leq i<j\leq n}\omega(x_i,x_j)|_{x_k=t^{k-1}x}
\end{align*}
which is equal to the product on the right hand side of (\ref{eq:epsB_0}).

\subsection{The shuffle element \texorpdfstring{$\kappa_n$}{κ_n}}
The results of this section are gathered in the following proposition:
\begin{prop}\label{prop:kappa}
Let $\kappa_0=1 \in \mathcal{A}_0$. Each of the following formulas uniquely determines the same sequence of shuffle algebra elements $\kappa_n\in \mathcal{A}_n$:
\begin{align}
    \label{eq:kappa_ev}
    &\Upsilon'(\kappa_n)=
     \sigma_q(e_n)\\
\label{eq:Wronski_lem}
    &\kappa_n = \sum_{j=0}^{n}
    q_k^{-n} (q_k-1)^{n-j}  \kappa_j * \epsilon_{n-j}(q_k)
    \qquad k=1,2,3\\
    \label{eq:kappa_x}
    &\left(x_1+\dots + x_n \right) \kappa_n
    = \kappa_{n-1} * x_1
\end{align}
where in (\ref{eq:Wronski_lem}) any one choice $k\in \{1,2,3\}$ suffices.
\end{prop}
\begin{proof}
It is clear that each of these equations uniquely determines the element $\kappa_n$, we only need to show their consistency. We can take (\ref{eq:kappa_x}) together with $\kappa_0=1$ as the definition of $\kappa_n$ and prove that it implies (\ref{eq:Wronski_lem}) and (\ref{eq:kappa_ev}). 

The derivation of (\ref{eq:Wronski_lem}) is recursive. The case $n=0$ of (\ref{eq:Wronski_lem}) obviously holds. Using (\ref{eq:kappa_x}) we will derive (\ref{eq:Wronski_lem}) for $\kappa_n$ assuming that (\ref{eq:Wronski_lem}) holds for $n-1$. We act with $*\, x_1$ on both sides of (\ref{eq:Wronski_lem}) where $n$ is replaced by $n-1$
\begin{align}
    \label{eq:Wronski_n-1}
    \kappa_{n-1} * x_1 = \sum_{j=0}^{n-1}
    q_k^{-n+1} (q_k-1)^{n-1-j}  \kappa_j * \epsilon_{n-j-1}(q_k) * x_1
\end{align}
On the left hand side we apply (\ref{eq:kappa_x}) and on the right hand side (\ref{eq:epsilon_rec})
\begin{align*}
    \sum_{i=1}^n x_i \kappa_{n}
    = 
    &\sum_{j=0}^{n-1} q_k^{-n}
    (q_k-1)^{n-j}  \kappa_j * \left( \sum_{i=1}^{n-j} x_i \epsilon_{n-j}(q_k) \right)\\
    +
    &\sum_{j=0}^{n-1}
    q_k^{-n} (q_k-1)^{n-j-1}  \kappa_j * x_1 * \epsilon_{n-j-1}(q_k) 
\end{align*}
On the right hand side we apply (\ref{eq:kappa_x}) to $\kappa_j * x_1$
\begin{align*}
    \sum_{i=1}^n x_i \kappa_{n}
    = 
    &\sum_{j=0}^{n-1} q_k^{-n}
    (q_k-1)^{n-j}  \kappa_j * \left( \sum_{i=1}^{n-j} x_i \epsilon_{n-j}(q_k) \right)\\
    +
    &\sum_{j=0}^{n-1}
     q_k^{-n}(q_k-1)^{n-j-1} \left(\sum_{i=1}^{j+1} x_i
    \kappa_{j+1} \right) * \epsilon_{n-j-1}(q_k)
\end{align*}
In the last summation we can shift the index $j$ 
\begin{align}
\label{eq:epskap0}
\sum_{i=1}^n x_i \kappa_{n}
    = 
    &\sum_{j=0}^{n-1} q_k^{-n}
    (q_k-1)^{n-j}  \kappa_j * \left( \sum_{i=1}^{n-j} x_i \epsilon_{n-j}(q_k) \right)\\
    +
    &\sum_{j=1}^{n}
     q_k^{-n}(q_k-1)^{n-j} \left( \sum_{i=1}^{j} x_i
    \kappa_{j} \right) * \epsilon_{n-j}(q_k)  \nonumber
\end{align}
From the definition of the shuffle product (\ref{eq:shuffle_prod}) it is easy to check that
\begin{align*}
    \kappa_j * \left( \sum_{i=1}^{n-j} x_i \epsilon_{n-j}(q_k) \right)
    +
     \left( \sum_{i=1}^{j} x_i
    \kappa_{j} \right) * \epsilon_{n-j}(q_k) =
    \left(\sum_{i=1}^{n} x_i\right) 
    \kappa_{j} * \epsilon_{n-j}(q_k)
\end{align*}
Using this identity in (\ref{eq:epskap0}) for every pair of summands with $j=1,\dots,n-1$ leads to
\begin{align*}
\sum_{i=1}^n x_i \kappa_{n}
    = 
    &\sum_{i=1}^{n} x_i 
    q_k^{-n}
    (q_k-1)^{n} \kappa_0 * \epsilon_{n}(q_k)
    +
    \sum_{i=1}^{n} x_i
     q_k^{-n} 
    \kappa_{n} * \epsilon_{0}(q_k) 
    \\
    +
    &
    \sum_{i=1}^{n} x_i
    \sum_{j=1}^{n-1}
     q_k^{-n}(q_k-1)^{n-j} 
    \kappa_{j} * \epsilon_{n-j}(q_k) 
\end{align*}
After removing the common factor $\sum_{i=1}^{n} x_i$ and putting all terms together on the right hand side we obtain (\ref{eq:Wronski_lem}) for $\kappa_n$.

Next we show that (\ref{eq:kappa_ev}) holds.  For that we will use (\ref{eq:Wronski_lem}).  It is sufficient to consider $k=1$ in (\ref{eq:Wronski_lem}). We apply the map $\Upsilon'$ to this equation
\begin{align}
\label{eq:ekap0}
        &\Upsilon'(\kappa_n) = \sum_{j=0}^{n}
    q^{n} (q^{-1}-1)^{n-j}  \Upsilon'(\kappa_j) \Upsilon'(\epsilon_{n-j}(q^{-1}))
\end{align}
From (\ref{eq:Ueps}) we have $\Upsilon'(\epsilon_{n-j}(q^{-1}))=e_{n-j}$. Let us denote by $e'_n=\Upsilon'(\kappa_n)$, then (\ref{eq:ekap0}) becomes
\begin{align}
\label{eq:ekap1}
        &e'_n = \sum_{j=0}^{n}
    q^{j} (1-q)^{n-j}  e'_j e_{n-j} 
\end{align}
We need to show that $e'_n=\sigma_q(e_n)$. Note that, together with the initial condition, (\ref{eq:ekap1}) uniquely determines $e'_n$. Let us solve (\ref{eq:ekap1}) using generating functions, set 
\begin{align}
    \label{eq:gen}
    E'(z):= \sum_{n=0}^{\infty}  e'_n z^n\qquad 
    E(z):= \sum_{n=0}^{\infty}  e_n z^n
    =\exp\left( -\sum_{j>0}\frac{(-1)^j}{j} p_j z^j
    \right)
\end{align}
Multiply both sides of  (\ref{eq:ekap1}) by $z^n$ and sum over $n\geq 0$
\begin{align*}
    E'(z)=E'(q z) E((1-q)z)
\end{align*}
By iterating this equation we find 
\begin{align*}
    E'(z)=&\prod_{k=0}^{\infty} E((1-q)q^k z)
    =\exp\left( -\sum_{k\geq 0}\sum_{j>0}\frac{(-1)^j(1-q)^{j}}{j} p_j q^{j k} z^j
    \right)\\
    =&\exp\left( -\sum_{j>0}\frac{(-1)^j(1-q)^{j}}{j(1-q^j)} p_j z^j
    \right)=
    \exp\left( -\sum_{j>0}\frac{(-1)^j}{j} \sigma_q(p_j) z^j
    \right)
    =\sum_{n=0}^{\infty}  \sigma_q(e_n)z^n
\end{align*}
Comparing the last result with (\ref{eq:gen}) shows that $e'_n=\sigma_q(e_n)$.
\end{proof}

\section{Matching the partition functions with shuffle elements}\label{sec:proof}
We now investigate the partition functions associated to the elements
$A_{\lambda_1}*\cdots*A_{\lambda_m}$ (where $A_n$ was defined in \eqref{eq:defA}) that were studied in the previous section on the shuffle algebra side. 

\begin{prop}\label{prop:zerospec}
Given $\lambda_1,\ldots,\lambda_m\in \ZZ_{>0}$ with $n=\sum_{i=1}^m \lambda_i$, one has the limit
\begin{equation}\label{eq:zerospec}
\lim_{x_n\to0}x_n^{n-1}
f(A_{\lambda_1}*\cdots*A_{\lambda_m})
=(-1)^{n-1}
\prod_{i=1}^{n-1} x_i\,
\sum_{i=1}^m
f(A_{\lambda_1}*\cdots*A_{\lambda_i-1}*A_{\lambda_m})
\end{equation}
\end{prop}
Note that by commutativity,
$\lambda$ can be chosen to be a partition.
\begin{proof}
We start from the definition of $A_{\lambda_1}*\cdots*A_{\lambda_m}$ (cf \eqref{eq:assoc}):
\[
A_{\lambda_1}*\cdots*A_{\lambda_m}
= \sum_{v\in \mathcal S_{\lambda_1,\ldots,\lambda_m}} \sum_{w\in\mathcal S^{\lambda_1,\ldots,\lambda_m}}
(-1)^{|v|} t^{-(|v|+|w|)} 
T_w T_v T_{w^{-1}}
\]
By definition of the groups involved, one has $T_w T_v = T_{wv}$ (and similarly $T_v T_{w^{-1}}=T_{vw^{-1}}$, but we choose to use the first identity).

Plugging this into \eqref{eq:deff2}, we have
\begin{align*}
\alpha_n^{-1} f(A_{\lambda_1}*\cdots*A_{\lambda_m})
&= \sum_{v\in \mathcal S_{\lambda_1,\ldots,\lambda_m}} \sum_{w\in\mathcal S^{\lambda_1,\ldots,\lambda_m}}
(-1)^{|v|} t^{-(|v|+|w|)} 
\left<
Z|_{y_i=qx_i} (T_w T_v T_{w^{-1}}\otimes S_n)
\right>
\\
&= \sum_{v\in \mathcal S_{\lambda_1,\ldots,\lambda_m}} \sum_{w\in\mathcal S^{\lambda_1,\ldots,\lambda_m}}
(-1)^{|v|} t^{-(|v|+|w|)} 
\left<
T_{w^{-1}} Z|_{y_i=qx_i} T_{wv}(1\otimes S_n)
\right>
\end{align*}
Using \eqref{eq:scalprod}, we can rewrite this as
\begin{align}\notag
\alpha_n^{-1}
f(A_{\lambda_1}*\cdots*A_{\lambda_m})
&=
\sum_{v\in \mathcal S_{\lambda_1,\ldots,\lambda_m}} \sum_{w\in\mathcal S^{\lambda_1,\ldots,\lambda_m}}
(-1)^{|v|} t^{-(|v|+|w|)} 
\bra{1}
T_{w^{-1}} Z|_{y_i=qx_i} T_{wv}
\ket{1}
\\\label{eq:intermlim0}
&=
\sum_{v\in \mathcal S_{\lambda_1,\ldots,\lambda_m}} \sum_{w\in\mathcal S^{\lambda_1,\ldots,\lambda_m}}
(-t)^{-|v|} 
\bra{ w^{-1}} Z|_{y_i=qx_i} \ket{v^{-1}w^{-1}}
\end{align}
This has the following interpretation: \rem[gray]{which really could be made for a general shuffle product}
$\alpha_n^{-1}f(A_{\lambda_1}*\cdots*A_{\lambda_m})$ is the sum over $v$ and $w$ of partition functions of lattice paths,
where the paths are labelled $w^{-1}(1),\ldots,w^{-1}(n)$ on the left side and
$v^{-1}(w^{-1}(1)),\ldots,v^{-1}(w^{-1}(n))$ on the top side (with an extra weight $(-t)^{-|v|}$).

We can now take the limit $x_n\to0$. In this paragraph we denote by $a=w^{-1}$, $b=v^{-1}w^{-1}$,
which are two arbitrary permutations, and consider $\bra{a}Z|_{y_i=qx_i}\ket{b}$ as $x_n\to 0$.
We are only interested for the purposes of deriving \eqref{eq:zerospec} in the terms of order $x_n^{-n+1}$;
from Lemma~\ref{lem:deg} we know that this is the minimum order in $x_n$ and in fact, as in the proof of the lemma,
we know that the only dependence on $x_n$ comes from the last row and column of the diagram of $Z$. Let us examine
these in more detail with our boundary conditions $\bra{a}$ and $\ket{b}$.
We look at the destiny of the path $b(n)$ which ends at the top right: to maximize the degree in $x_n^{-1}$, we must
pick up a $x_n^{-1}$ at each vertex $(i,n)$, $i=1,\ldots,n-1$; this means, according to \eqref{eq:wt} that the first
type of vertex (a path making a bend when the right edge is either empty or has a higher labelled path) must be excluded.
Since all right edges are empty in the last column, one can prove inductively that the paths must go straight through
every vertex $(i,n)$,  $i=1,\ldots,n-1$, incurring a weight of $-q^{-1}x_ix_n^{-1}$.
At vertex $(n,n)$, because both bottom and right edges are empty, we must have the first type of vertex in \eqref{eq:wt},
incurring a $1-t$. From there, the path can only go straight left to its destination with label $a(n)$, incurring
a weight $t^{n-1}$ and imposing $a(n)=b(n)$. Combining all this, we find
\begin{multline}\label{eq:limit0}
\bra{a} Z(x_1,\ldots,x_n,qx_1,\ldots,qx_n) \ket{b}
= \delta_{a(n),b(n)} (1-t) t^{n-1} x_n^{-n+1} 
\prod_{i=1}^{n-1} (-q^{-1}x_i)
\\\times \bra{a(1),\ldots,a(n-1)} Z(x_1,\ldots,x_{n-1},qx_1,\ldots,qx_{n-1}) \ket{b(1),\ldots,b(n-1)}
+O(x_n^{-n+2})\qquad a,b\in\mathcal S_n^{(1)}
\end{multline}

We apply this formula to $f(A_{\lambda_1}*\cdots*A_{\lambda_m})$.
We fix $w\in\mathcal S_{\lambda_1,\ldots,\lambda_m}$ and denote by $k=w^{-1}(n)$. Because of the form of $w$, $k$ must be the last element of its block $I_a$, $1\le a\le m$ ($I_a$ was defined in \eqref{eq:defI}).
According to \eqref{eq:limit0},
the summation can be restricted to $v$ such that $v(k)=k$. Consider now the permutation $w'$ with $w'(i)=w(i)$, $i=1,\ldots,k-1$, $w'(i)=w(i+1)$, $i=k,\ldots,n-1$, and
the permutation $v'$ with $v'(i)=i$ for $i=1,\ldots,k-1$ and $v'(i)=v(i+1)-1$ for $i=k,\ldots,n-1$. Because $k$ is last of its block,
$|v'|=|v|$ and we have
\begin{multline*}
\alpha_n^{-1} f(A_{\lambda_1}*\cdots*A_{\lambda_m})
\overset{x_n\to 0}{\sim} 
(1-t)t^{n-1} x_n^{-n+1} 
\prod_{i=1}^{n-1} (-q^{-1}x_i)
\sum_{v'\in\mathcal S_{\lambda_1,\ldots,\lambda_a-1,\ldots,\lambda_m}}
(-t)^{-|v'|}
\sum_{w'\in\mathcal S^{\lambda_1,\ldots,\lambda_a-1,\ldots,\lambda_m}}
\\\bra{w'^{-1}(1),\ldots,w'^{-1}(n-1)}
Z(x_1,\ldots,x_{n-1},qx_1,\ldots,qx_{n-1}) 
\ket{v^{-1}(w'^{-1}(1)),\ldots,v'^{-1}(w'^{-1}(n-1))}
\end{multline*}
Comparing with \eqref{eq:intermlim0} and using
$\alpha_n^{-1}\alpha_{n-1}=(1-t)(t/q)^{n-1}$, we finally find the desired
limit \eqref{eq:zerospec}.
\end{proof}

In what follows, we shall use the diagrammatic language.
We shall repeatedly use the following simple identities:
\begin{lem}
One has for $k\le\ell,m\ge2$:
\begin{align}\label{eq:id1}
\begin{tikzpicture}[baseline=(current  bounding  box.center),rounded corners,scale=0.8]
\node[draw=black,rectangle,minimum height=0.4cm,minimum width=1.5cm,inner sep=0pt] at (2.5,0) {$\ss A_k$};
\node[draw=black,rectangle,minimum height=0.4cm,minimum width=3cm,inner sep=0pt] at (2.5,1) {$\ss A_\ell$};
\foreach\x in {1,4}
\draw[invarrow=0.93] (\x,-0.75) -- (\x,0.75) (\x,1.25) -- (\x,1.75);
\foreach\x in {2,3}
\draw[invarrow=0.93] (\x,-0.75) -- (\x,-0.25) (\x,0.25) -- (\x,0.75) (\x,1.25) -- (\x,1.75);
\foreach\x/\y in {1.5/1.5,2.5/1.5,3.5/1.5,2.5/0.5,2.5/-0.5}
\node at (\x,\y) {$\ss\cdots$};
\end{tikzpicture}
&=
\begin{tikzpicture}[baseline=(current  bounding  box.center),rounded corners,scale=0.8]
\node[draw=black,rectangle,minimum height=0.4cm,minimum width=1.5cm,inner sep=0pt] at (2.5,1) {$\ss A_k$};
\node[draw=black,rectangle,minimum height=0.4cm,minimum width=3cm,inner sep=0pt] at (2.5,0) {$\ss A_\ell$};
\foreach\x in {1,4}
\draw[invarrow=0.93] (\x,-0.75) -- (\x,-0.25) (\x,0.25) -- (\x,1.75);
\foreach\x in {2,3}
\draw[invarrow=0.93] (\x,-0.75) -- (\x,-0.25) (\x,0.25) -- (\x,0.75) (\x,1.25) -- (\x,1.75);
\foreach\x/\y in {1.5/1.5,2.5/1.5,3.5/1.5,2.5/0.5,2.5/-0.5}
\node at (\x,\y) {$\ss\cdots$};
\end{tikzpicture}
=
[k]_{t^{-1}}!\ 
\begin{tikzpicture}[baseline=(current  bounding  box.center),rounded corners,scale=0.8]
\node[draw=black,rectangle,minimum height=0.4cm,minimum width=3cm,inner sep=0pt] at (2.5,1) {$\ss A_\ell$};
\foreach\x in {1,...,4}
\draw[invarrow=0.75] (\x,0.25) -- (\x,0.75) (\x,1.25) -- (\x,1.75);
\foreach\x/\y in {2.5/1.5,2.5/0.5}
\node at (\x,\y) {$\ss\cdots$};
\end{tikzpicture}
\\\label{eq:id2}
\begin{tikzpicture}[baseline=(current  bounding  box.center),rounded corners,scale=0.8]
\node[draw=black,rectangle,minimum height=0.4cm,minimum width=1.5cm,inner sep=0pt] at (2.5,0) {$\ss S_k$};
\node[draw=black,rectangle,minimum height=0.4cm,minimum width=3cm,inner sep=0pt] at (2.5,1) {$\ss S_\ell$};
\foreach\x in {1,4}
\draw[invarrow=0.93] (\x,-0.75) -- (\x,0.75) (\x,1.25) -- (\x,1.75);
\foreach\x in {2,3}
\draw[invarrow=0.93] (\x,-0.75) -- (\x,-0.25) (\x,0.25) -- (\x,0.75) (\x,1.25) -- (\x,1.75);
\foreach\x/\y in {1.5/1.5,2.5/1.5,3.5/1.5,2.5/0.5,2.5/-0.5}
\node at (\x,\y) {$\ss\cdots$};
\end{tikzpicture}
&=
\begin{tikzpicture}[baseline=(current  bounding  box.center),rounded corners,scale=0.8]
\node[draw=black,rectangle,minimum height=0.4cm,minimum width=1.5cm,inner sep=0pt] at (2.5,1) {$\ss S_k$};
\node[draw=black,rectangle,minimum height=0.4cm,minimum width=3cm,inner sep=0pt] at (2.5,0) {$\ss S_\ell$};
\foreach\x in {1,4}
\draw[invarrow=0.93] (\x,-0.75) -- (\x,-0.25) (\x,0.25) -- (\x,1.75);
\foreach\x in {2,3}
\draw[invarrow=0.93] (\x,-0.75) -- (\x,-0.25) (\x,0.25) -- (\x,0.75) (\x,1.25) -- (\x,1.75);
\foreach\x/\y in {1.5/1.5,2.5/1.5,3.5/1.5,2.5/0.5,2.5/-0.5}
\node at (\x,\y) {$\ss\cdots$};
\end{tikzpicture}
=
[k]_t!\ 
\begin{tikzpicture}[baseline=(current  bounding  box.center),rounded corners,scale=0.8]
\node[draw=black,rectangle,minimum height=0.4cm,minimum width=3cm,inner sep=0pt] at (2.5,1) {$\ss S_\ell$};
\foreach\x in {1,...,4}
\draw[invarrow=0.75] (\x,0.25) -- (\x,0.75) (\x,1.25) -- (\x,1.75);
\foreach\x/\y in {2.5/1.5,2.5/0.5}
\node at (\x,\y) {$\ss\cdots$};
\end{tikzpicture}
\\\label{eq:id3}
\begin{tikzpicture}[baseline=(current  bounding  box.center),rounded corners,scale=0.8]
\node[draw=black,rectangle,minimum height=0.4cm,minimum width=3cm,inner sep=0pt] at (2.5,-1) {$\ss A_m$};
\node[draw=black,rectangle,minimum height=0.4cm,minimum width=3cm,inner sep=0pt] at (2.5,1) {$\ss S_\ell$};
\foreach\x in {1,4}
\draw[invarrow=0.93] (\x,-1.75) -- (\x,-1.25) (\x,-0.75) -- (\x,-0.25) (\x,0.25) -- (\x,0.75) (\x,1.25) -- (\x,1.75);
\foreach\x in {2,3}
\draw[invarrow=0.93] (\x,-1.75) -- (\x,-1.25) (\x,-0.75) -- (\x,0.75) (\x,1.25) -- (\x,1.75);
\foreach\x/\y in {2.5/1.5,2.5/-1.5}
\node at (\x,\y) {$\ss\cdots$};
\end{tikzpicture}
&=
\begin{tikzpicture}[baseline=(current  bounding  box.center),rounded corners,scale=0.8]
\node[draw=black,rectangle,minimum height=0.4cm,minimum width=3cm,inner sep=0pt] at (2.5,-1) {$\ss S_\ell$};
\node[draw=black,rectangle,minimum height=0.4cm,minimum width=3cm,inner sep=0pt] at (2.5,1) {$\ss A_m$};
\foreach\x in {1,4}
\draw[invarrow=0.93] (\x,-1.75) -- (\x,-1.25) (\x,-0.75) -- (\x,-0.25) (\x,0.25) -- (\x,0.75) (\x,1.25) -- (\x,1.75);
\foreach\x in {2,3}
\draw[invarrow=0.93] (\x,-1.75) -- (\x,-1.25) (\x,-0.75) -- (\x,0.75) (\x,1.25) -- (\x,1.75);
\foreach\x/\y in {2.5/1.5,2.5/-1.5}
\node at (\x,\y) {$\ss\cdots$};
\end{tikzpicture}
=0
\end{align}
where $[k]_t!=\prod_{i=1}^k [i]_t$ and $[i]_t=1+\cdots+t^{i-1}$.
\end{lem}

For general $c\in\Z(\H_n)$, $f(c)$ can be described as
\begin{equation}\label{eq:diagf}
f(c) = \alpha_n
\left<\quad
\begin{tikzpicture}[baseline=(current  bounding  box.center),rounded corners,scale=0.8]
\foreach\j/\lab in {6/qx_n,5/,4/\cdots,3/,2/qx_2,1/qx_1} {
\draw[invarrow=0.97] (\j,0.5) -- node[pos=0.97,right] {$\ss \lab$} (\j,6.5);
\draw (\j,0) -- (\j,-0.5);
}
\foreach\i/\lab in {1/x_1,2/x_2,3/,4/\vdots,5/,6/x_n} {
\draw[invarrow=0.97] (0.5,7-\i) -- node[pos=0.96,below] {$\ss \lab$} (6.5,7-\i);
\draw (-0.5,7-\i) -- (0,7-\i);
}
\node[draw=black,rectangle,minimum height=0.4cm,minimum width=4.5cm,inner sep=0pt] at (3.5,0.25) {$\ss S_n$};
\node[draw=black,rectangle,minimum height=0.4cm,minimum width=4.5cm,inner sep=0pt,rotate=-90] at (0.25,3.5) {$\ss c$};
\end{tikzpicture}
\right>
\end{equation}
This immediately implies
\begin{lem}\label{lem:wheel}
For any $c\in\Z$, $f(c)$ satisfies the wheel conditions
\[
f(c)(x,tx,qx,x_4,\ldots,x_n)=f(c)(tqx,tx,qx,x_4,\ldots,x_n)=0
\]
\end{lem}
\begin{proof}
$f(c)$ being symmetric in its arguments according to Proposition~\ref{prop:charsym},
we can pick any three variables and specialize them.
Let us then consider the diagram above and set $x_{n-2}=x$, $x_{n-1}=tx$, $x_n=qx$.
Zooming in on the bottom right corner, we have, using $\check R(t^{-1})=(1-t)A_2$,
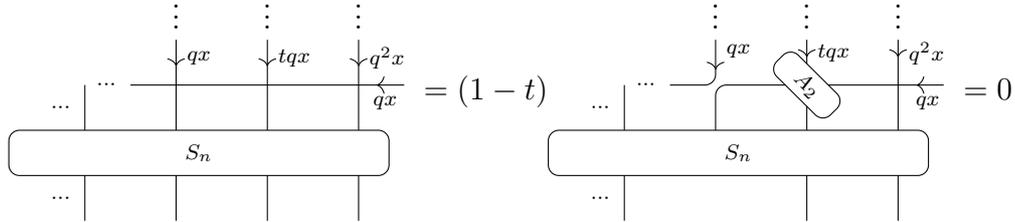
\[
\begin{tikzpicture}[baseline=(current  bounding  box.center),rounded corners,scale=1.2]
\foreach\j/\lab in {6/q^2x,5/tqx,4/qx} {
\draw[invarrow=0.8] (\j,0.5) -- node[pos=0.8,right] {$\ss \lab$} (\j,1.5) node[above] {$\ss\vdots$};
\draw (\j,-0.5) -- (\j,0);
}
\draw[invarrow=0.93] (3.5,1) node[left] {$\ss\cdots$} -- node[pos=0.93,below] {$\ss qx$} (6.5,1);
\node[draw=black,rectangle,minimum height=0.6cm,minimum width=5cm,inner sep=0pt] at (4.25,0.25) {$\ss S_n$};
\draw (3,0.5) -- node[left] {$\ss\cdots$} (3,1);
\draw (3,-0.5) -- node[left] {$\ss\cdots$} (3,0);
\end{tikzpicture}
=(1-t)
\begin{tikzpicture}[baseline=(current  bounding  box.center),rounded corners,scale=1.2]
\foreach\j/\lab in {6/q^2x,5/tqx} {
\draw[invarrow=0.90] (\j,0.5) -- node[pos=0.85,right] {$\ss \lab$} (\j,1.5) node[above] {$\ss\vdots$};
\draw (\j,-0.5) -- (\j,0);
}
\draw[invarrow=0.75] (3.5,1) node[left] {$\ss\cdots$} -- (4,1) -- node[pos=0.75,right] {$\ss qx$} (4,1.5) node[above] {$\ss\vdots$};
\draw[invarrow=0.93]  (4,0.5) -- (4,1) -- node[pos=0.93,below] {$\ss qx$} (6.5,1);
\node[draw=black,rectangle,minimum height=0.6cm,minimum width=5cm,inner sep=0pt] at (4.25,0.25) {$\ss S_n$};
\node[draw=black,fill=white,rectangle,minimum height=0.4cm,minimum width=1cm,inner sep=0pt,rotate=-45] at (5,1) {$\ss A_2$};
\draw (3,0.5) -- node[left] {$\ss\cdots$} (3,1);
\draw (3,-0.5) -- node[left] {$\ss\cdots$} (3,0);
\end{tikzpicture}
=0
\]
Switching rows and columns results similarly in the second wheel condition.
\end{proof}

We now go back to $A_{\lambda_1}*\cdots* A_{\lambda_m}$.
Starting from the expression
\[
f(A_{\lambda_1}*\cdots*A_{\lambda_m})
=\alpha_n \sum_{w\in \mathcal S^{\lambda_1,\ldots,\lambda_m}}t^{-|w|}
\left<
(T_w(A_{\lambda_1}\otimes\cdots\otimes A_{\lambda_m})T_{w^{-1}}\otimes S_n)Z
\right>|_{y_i=qx_i,\ i=1,\ldots,n}
\]
and using the cyclicity of $\left<\bullet\right>$, we can describe the expression inside it
as the diagram: (compare with \eqref{eq:diagf})
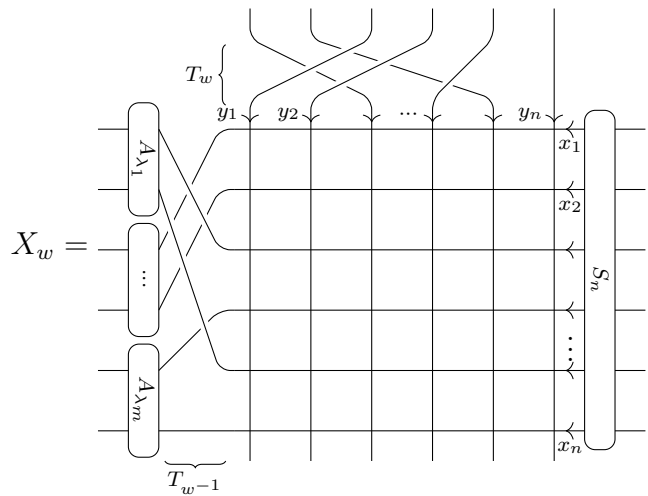
\[
X_w
=
\begin{tikzpicture}[baseline=(current  bounding  box.center),rounded corners,scale=0.8]
\foreach\j/\k/\lab in {6/6/y_n,5/2/,4/5/\cdots,3/1/,2/4/y_2,1/3/y_1} {
\draw[draw=white,double=black,ultra thick, double distance=0.4pt] (\j,6.25) -- (\j,6.5) -- (\k,7.5) -- (\k,8);
\draw[invarrow=1] (\j,0.5) -- (\j,6.25) node[left] {$\ss \lab$};
}
\foreach\i/\k/\lab in {1/3/x_1,2/4/x_2,3/1/,4/5/\vdots,5/2/,6/6/x_n} {
\draw[draw=white,double=black,ultra thick, double distance=0.4pt] (0.75,7-\i) -- (0.5,7-\i) -- (-0.5,7-\k);
\draw[invarrow=0.97] (0.75,7-\i) -- node[pos=0.96,below] {$\ss \lab$} (6.5,7-\i);
\draw (-1,7-\i) -- (-1.5,7-\i);
\draw (7,7-\i) -- (7.5,7-\i);
}
\node[draw=black,rectangle,minimum height=0.4cm,minimum width=1.5cm,inner sep=0pt,rotate=-90] at (-0.75,5.5) {$\ss A_{\lambda_1}$};
\node[draw=black,rectangle,minimum height=0.4cm,minimum width=1.5cm,inner sep=0pt,rotate=-90] at (-0.75,3.5) {$\ss \cdots$};
\node[draw=black,rectangle,minimum height=0.4cm,minimum width=1.5cm,inner sep=0pt,rotate=-90] at (-0.75,1.5) {$\ss A_{\lambda_m}$};
\node[draw=black,rectangle,minimum height=0.4cm,minimum width=4.5cm,inner sep=0pt,rotate=-90] at (6.75,3.5) {$\ss S_n$};
\draw[decorate,decoration=brace,sharp corners] (0.6,0.5) -- node[below] {$\ss T_{w^{-1}}$} (-0.4,0.5);
\draw[decorate,decoration=brace,sharp corners] (0.6,6.4) -- node[left] {$\ss T_{w}$} (0.6,7.4);
\end{tikzpicture}
\]
As a side remark, using the idempotency of $A_n$ and $S_n$, we can introduce
the more symmetric diagram
\begin{equation}\label{eq:diagram}
\tilde X_w :=
\begin{tikzpicture}[baseline=(current  bounding  box.center),rounded corners,scale=0.8]
\foreach\j/\k/\lab in {6/6/y_n,5/2/,4/5/\cdots,3/1/,2/4/y_2,1/3/y_1} {
\draw[draw=white,double=black,ultra thick, double distance=0.4pt] (\j,6.25) -- (\j,6.5) -- (\k,7.5);
\draw[invarrow=1] (\j,0.5) -- (\j,6.25) node[left] {$\ss \lab$};
\draw (\j,8) -- (\j,8.5);
\draw (\j,0) -- (\j,-0.5);
}
\foreach\i/\k/\lab in {1/3/x_1,2/4/x_2,3/1/,4/5/\vdots,5/2/,6/6/x_n} {
\draw[draw=white,double=black,ultra thick, double distance=0.4pt] (0.75,7-\i) -- (0.5,7-\i) -- (-0.5,7-\k);
\draw[invarrow=0.97] (0.75,7-\i) -- node[pos=0.96,below] {$\ss \lab$} (6.5,7-\i);
\draw (-1,7-\i) -- (-1.5,7-\i);
\draw (7,7-\i) -- (7.5,7-\i);
}
\node[draw=black,rectangle,minimum height=0.4cm,minimum width=1.5cm,inner sep=0pt] at (1.5,7.75) {$\ss A_{\lambda_1}$};
\node[draw=black,rectangle,minimum height=0.4cm,minimum width=1.5cm,inner sep=0pt] at (3.5,7.75) {$\ss\cdots$};
\node[draw=black,rectangle,minimum height=0.4cm,minimum width=1.5cm,inner sep=0pt] at (5.5,7.75) {$\ss A_{\lambda_m}$};
\node[draw=black,rectangle,minimum height=0.4cm,minimum width=1.5cm,inner sep=0pt,rotate=-90] at (-0.75,5.5) {$\ss A_{\lambda_1}$};
\node[draw=black,rectangle,minimum height=0.4cm,minimum width=1.5cm,inner sep=0pt,rotate=-90] at (-0.75,3.5) {$\ss \cdots$};
\node[draw=black,rectangle,minimum height=0.4cm,minimum width=1.5cm,inner sep=0pt,rotate=-90] at (-0.75,1.5) {$\ss A_{\lambda_m}$};
\node[draw=black,rectangle,minimum height=0.4cm,minimum width=4.5cm,inner sep=0pt] at (3.5,0.25) {$\ss S_n$};
\node[draw=black,rectangle,minimum height=0.4cm,minimum width=4.5cm,inner sep=0pt,rotate=-90] at (6.75,3.5) {$\ss S_n$};
\draw[decorate,decoration=brace,sharp corners] (0.6,0.5) -- node[below] {$\ss T_{w^{-1}}$} (-0.4,0.5);
\draw[decorate,decoration=brace,sharp corners] (0.6,6.4) -- node[left] {$\ss T_{w}$} (0.6,7.4);
\end{tikzpicture}
\end{equation}
with $\tilde X_w = \left<X_w\right> A_{\lambda_1}\otimes\cdots\otimes A_{\lambda_m}\otimes S_n$. We shall not use $\tilde X_w$ in what follows.

\begin{prop}\label{prop:nasty}
\begin{itemize}
    \item For $m<n$, one has the vanishing condition:
    \begin{equation}\label{eq:vanish}
    f(A_{\lambda_1}*\cdots*A_{\lambda_m})(x,t x, \ldots,t^m x,x_{m+2},\ldots,x_n)=0
    \end{equation}
    \item
The partition function $f(A_n)$ has the fully factorized form
\begin{equation}\label{eq:m1}
f(A_n)= 
t^{-\frac{n(n-1)}{2}}\prod_{\substack{i,j=1\\i\ne j}}^n (1-t x_i/x_j)
\end{equation}
\item One has the specialization
\begin{equation}\label{eq:specA1}
f(A_1*\cdots*A_1)(1,t,\ldots,t^{n-1})=
[n]_{t^{-1}}!
\prod_{1\le i<j\le n}(q-t^{j-i})(q-t^{i-j+1})
\end{equation}
\end{itemize}
\end{prop}

\begin{proof}
We first consider $m=1$, and investigate the effect of setting $x_1=t x$, $x_2=x$ in $X_w$, where here $w=1$.
Noting that $\check R_1(t)=(1-t)(1+T_1)=(1-t)S_2$, and using \eqref{eq:id2}, we can insert a crossing at the right of the top two horizontal lines and then
 apply the Yang--Baxter equation repeatedly to move the crossing to the left,
resulting in zero according to \eqref{eq:id3}:
\begin{align*}
X_1|_{x_1=t x,\ x_2=x}
&=\frac{1}{1+t}\ 
\begin{tikzpicture}[baseline=(current  bounding  box.center),rounded corners,scale=0.8]
\foreach\j/\lab in {6/y_n,5/,4/\cdots,3/,2/y_2,1/y_1} {
\draw[invarrow=0.97] (\j,0.5) -- node[pos=0.97,left] {$\ss \lab$} (\j,6.5) -- (\j,7);
}
\foreach\i/\lab in {1/tx,2/x,3/,4/\vdots,5/,6/x_n} {
\draw (-0.5,7-\i) -- (0,7-\i);
\draw[invarrow=0.85] (0.5,7-\i) -- node[pos=0.85,below] {$\ss \lab$} (7.25,7-\i);
\draw (7.75,7-\i) -- (8.25,7-\i);
}
\node[draw=black,rectangle,minimum height=0.4cm,minimum width=4.5cm,inner sep=0pt,rotate=-90] at (0.25,3.5) {$\ss A_n$};
\node[draw=black,rectangle,minimum height=0.4cm,minimum width=4.5cm,inner sep=0pt,rotate=-90] at (7.5,3.5) {$\ss S_n$};
\node[draw=black,fill=white,rectangle,minimum height=0.4cm,minimum width=1.2cm,inner sep=0pt,rotate=-90] at (6.75,5.5) {$\ss S_2$};
\end{tikzpicture}
\\
&=\frac{1}{1+t}\ 
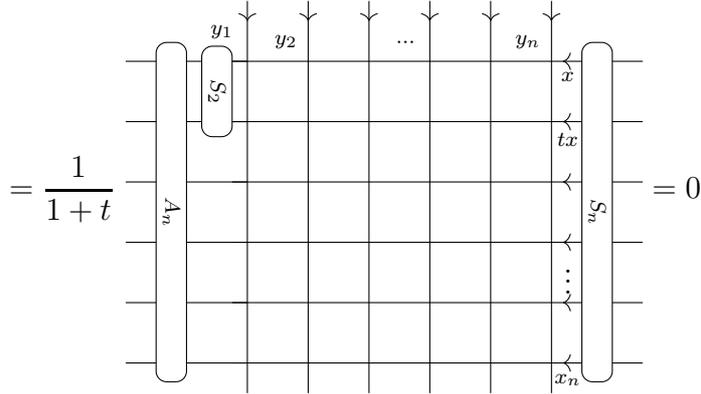
\begin{tikzpicture}[baseline=(current  bounding  box.center),rounded corners,scale=0.8]
\foreach\j/\lab in {6/y_n,5/,4/\cdots,3/,2/y_2,1/{y_1\atop}} {
\draw[invarrow=0.97] (\j,0.5) -- node[pos=0.97,left] {$\ss \lab$} (\j,6.5) -- (\j,7);
}
\foreach\i/\lab in {1/x,2/tx,3/,4/\vdots,5/,6/x_n} {
\draw (1,7-\i) -- (0,7-\i);
\draw[invarrow=0.97] (0.75,7-\i) -- node[pos=0.96,below] {$\ss \lab$} (6.5,7-\i);
\draw (-0.5,7-\i) -- (-1,7-\i);
\draw (7,7-\i) -- (7.5,7-\i);
}
\node[draw=black,rectangle,minimum height=0.4cm,minimum width=4.5cm,inner sep=0pt,rotate=-90] at (-0.25,3.5) {$\ss A_n$};
\node[draw=black,rectangle,minimum height=0.4cm,minimum width=4.5cm,inner sep=0pt,rotate=-90] at (6.75,3.5) {$\ss S_n$};
\node[draw=black,fill=white,rectangle,minimum height=0.4cm,minimum width=1.2cm,inner sep=0pt,rotate=-90] at (0.5,5.5) {$\ss S_2$};
\end{tikzpicture}
=0
\end{align*}

We conclude that $x_1-tx_2$ divides $f(A_n)$, and by symmetry,
$f(A_n)$ is a multiple of $\prod_{i\ne j} (1-t x_i/x_j)$. The latter has degree range $[-(n-1),(n-1)]$ in each $x_i$, which is the maximum allowed according to \eqref{eq:degbounds}, so only a constant in $\FF$ remains to be determined. The latter is determined inductively by using Proposition~\ref{prop:zerospec}.

We proceed identically for $1<m<n$. We set $x_1=t^m x$, \dots, $x_m=tx$, $x_{m+1}=x$.
We note that the ``$R$-matrix associated to the longest element $w_0$ of $\mathcal S_{m+1}$'', i.e., the product of $R$-matrices whose diagram reproduces (any) wiring diagram of $w_0$, with spectral parameters set to $t^{m},\ldots,1$, is nothing but 
$(1-t)^{\frac{m(m+1)}{2}}\prod_{i=1}^m [i]_t!\,
S_{m+1}$ (see e.g. \cite[Lem.~3]{Lascoux-Frob} or \cite[3.1]{IO-shuffle}), and move it across to the left using the Yang--Baxter equation:
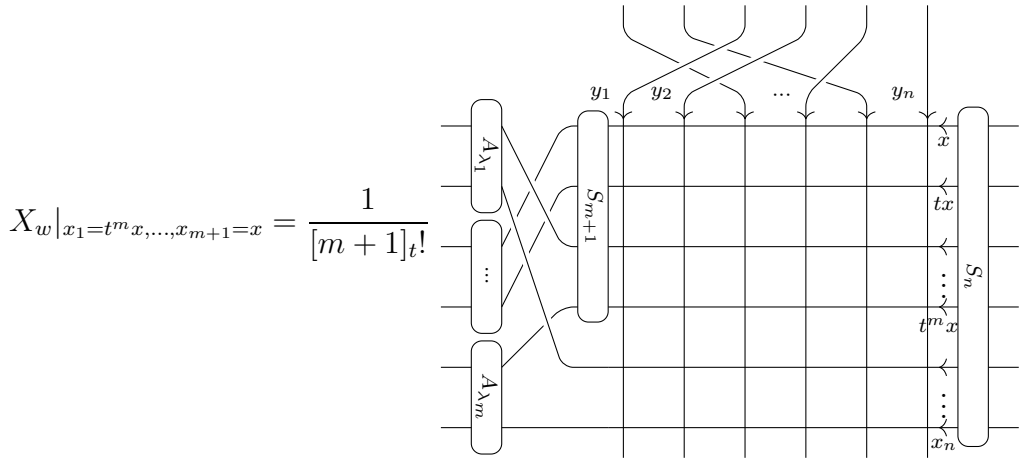
\[
X_w|_{x_1=t^m x,\ldots,x_{m+1}=x}
=\frac{1}{[m+1]_t!}\ 
\begin{tikzpicture}[baseline=(current  bounding  box.center),rounded corners,scale=0.8]
\foreach\j/\k/\lab in {6/6/y_n,5/2/,4/5/\cdots,3/1/,2/4/y_2,1/3/y_1} {
\draw[draw=white,double=black,ultra thick, double distance=0.4pt] (\j,6.25) -- (\j,6.5) -- (\k,7.5) -- (\k,8);
\draw[invarrow=1] (\j,0.5) -- (\j,6.25) node[above left] {$\ss \lab$};
}
\foreach\i/\k/\lab in {1/3/x,2/4/tx,3/1/\vdots,4/5/{t^{m}x\ },5/2/\vdots,6/6/x_n} {
\draw[draw=white,double=black,ultra thick, double distance=0.4pt] (0.75,7-\i) -- (0,7-\i) -- (-1,7-\k);
\draw[invarrow=0.97] (0.75,7-\i) -- node[pos=0.96,below] {$\ss \lab$} (6.5,7-\i);
\draw (-1.5,7-\i) -- (-2,7-\i);
\draw (7,7-\i) -- (7.5,7-\i);
}
\node[draw=black,rectangle,minimum height=0.4cm,minimum width=1.5cm,inner sep=0pt,rotate=-90] at (-1.25,5.5) {$\ss A_{\lambda_1}$};
\node[draw=black,rectangle,minimum height=0.4cm,minimum width=1.5cm,inner sep=0pt,rotate=-90] at (-1.25,3.5) {$\ss \cdots$};
\node[draw=black,rectangle,minimum height=0.4cm,minimum width=1.5cm,inner sep=0pt,rotate=-90] at (-1.25,1.5) {$\ss A_{\lambda_m}$};
\node[draw=black,rectangle,minimum height=0.4cm,minimum width=4.5cm,inner sep=0pt,rotate=-90] at (6.75,3.5) {$\ss S_n$};
\node[draw=black,fill=white,rectangle,minimum height=0.4cm,minimum width=2.8cm,inner sep=0pt,rotate=-90] at (0.5,4.5) {$\ss S_{m+1}$};
\end{tikzpicture}
\]
The pigeonhole principle tells us there will be two numbers between $1$ and $m+1$ whose preimages under $w$
lie in the same block (i.e., are associated to the same part of $\lambda$). Since $w\in \mathcal S^{\lambda_1,\ldots,\lambda_m}$, all numbers whose preimages are in the same block
are consecutive, and so we can apply \eqref{eq:id3} to conclude that the result is zero. By summing over $w$ and using symmetry in the exchange of variables, we obtain the result \eqref{eq:vanish} as stated in the Proposition.

The only case left is $m=n$, i.e., $A_1*\cdots*A_1$. We set $x_1=t^{n-1}x,\ldots,x_{n-1}=tx,x_n=x$ but this time the result
does not vanish. \rem[gray]{in fact, we could play the same trick with $r=m$ and get a factorized expression for any $m$,
which is how I did it originally.} However, noting once again that $S_n$ is up to normalization
the product of $R$-matrices associated to the longest permutation $w_0$ of $\mathcal S_n$,
we can move it to the left and absorb any crossing from $T_w$:
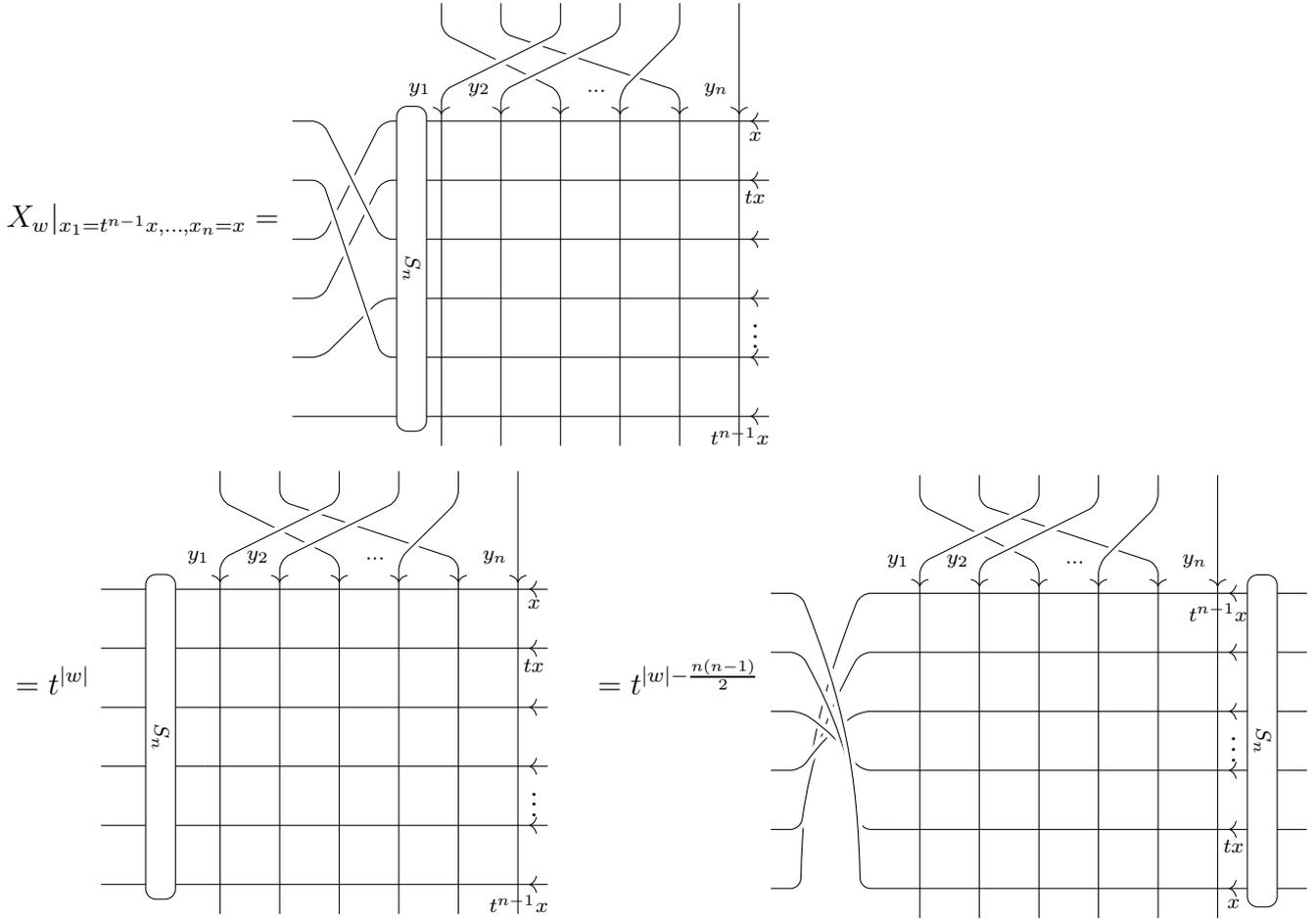
\begin{multline*}
X_w|_{x_1=t^{n-1} x,\ldots,x_{n}=x}
=
\begin{tikzpicture}[baseline=(current  bounding  box.center),rounded corners,scale=0.8]
\foreach\j/\k/\lab in {6/6/y_n,5/2/,4/5/\cdots,3/1/,2/4/y_2,1/3/y_1} {
\draw[draw=white,double=black,ultra thick, double distance=0.4pt] (\j,6.25) -- (\j,6.5) -- (\k,7.5) -- (\k,8);
\draw[invarrow=1] (\j,0.5) -- (\j,6.25) node[above left] {$\ss \lab$};
}
\foreach\i/\k/\lab in {1/3/x,2/4/tx,3/1/,4/5/\vdots,5/2/,6/6/{t^{n-1}x\quad}} {
\draw[draw=white,double=black,ultra thick, double distance=0.4pt] (0.75,7-\i) -- (0,7-\i) -- (-1,7-\k) -- (-1.5,7-\k);
\draw[invarrow=0.97] (0.75,7-\i) -- node[pos=0.96,below] {$\ss \lab$} (6.5,7-\i);
}
\node[draw=black,fill=white,rectangle,minimum height=0.4cm,minimum width=4.4cm,inner sep=0pt,rotate=-90] at (0.5,3.5) {$\ss S_n$};
\end{tikzpicture}
\\
=t^{|w|}\ 
\begin{tikzpicture}[baseline=(current  bounding  box.center),rounded corners,scale=0.8]
\foreach\j/\k/\lab in {6/6/y_n,5/2/,4/5/\cdots,3/1/,2/4/y_2,1/3/y_1} {
\draw[draw=white,double=black,ultra thick, double distance=0.4pt] (\j,6.25) -- (\j,6.5) -- (\k,7.5) -- (\k,8);
\draw[invarrow=1] (\j,0.5) -- (\j,6.25) node[above left] {$\ss \lab$};
}
\foreach\i/\lab in {1/x,2/tx,3/,4/\vdots,5/,6/{t^{n-1}x\quad}} {
\draw[draw=white,double=black,ultra thick, double distance=0.4pt] (0.75,7-\i) -- (0,7-\i) -- (-1,7-\i);
\draw[invarrow=0.97] (0.75,7-\i) -- node[pos=0.96,below] {$\ss \lab$} (6.5,7-\i);
}
\node[draw=black,fill=white,rectangle,minimum height=0.4cm,minimum width=4.4cm,inner sep=0pt,rotate=-90] at (0,3.5) {$\ss S_n$};
\end{tikzpicture}
=t^{|w|-\frac{n(n-1)}{2}}\ 
\begin{tikzpicture}[baseline=(current  bounding  box.center),rounded corners,scale=0.8]
\foreach\j/\k/\lab in {6/6/y_n,5/2/,4/5/\cdots,3/1/,2/4/y_2,1/3/y_1} {
\draw[draw=white,double=black,ultra thick, double distance=0.4pt] (\j,6.25) -- (\j,6.5) -- (\k,7.5) -- (\k,8);
\draw[invarrow=1] (\j,0.5) -- (\j,6.25) node[above left] {$\ss \lab$};
}
\foreach\i/\k/\lab in {1/6/{t^{n-1}x\quad},2/5/,3/4/\vdots,4/3/,5/2/tx,6/1/x} {
\draw[draw=white,double=black,ultra thick, double distance=0.4pt] (0.75,7-\i) -- (0,7-\i) [bend right=10] to (-1,7-\k) -- (-1.5,7-\k);
\draw[invarrow=0.97] (0.75,7-\i) -- node[pos=0.96,below] {$\ss \lab$} (6.5,7-\i);
\draw (7,7-\i) -- (7.5,7-\i);
}
\node[draw=black,rectangle,minimum height=0.4cm,minimum width=4.5cm,inner sep=0pt,rotate=-90] at (6.75,3.5) {$\ss S_n$};
\end{tikzpicture}
\end{multline*}
where in the last line we find it convenient to reintroduce the longest element $w_0$, before moving back $S_n$ to the right.
Performing the summation over $w$ (one could instead absorb the top crossings as well, but the result would be the same), we find
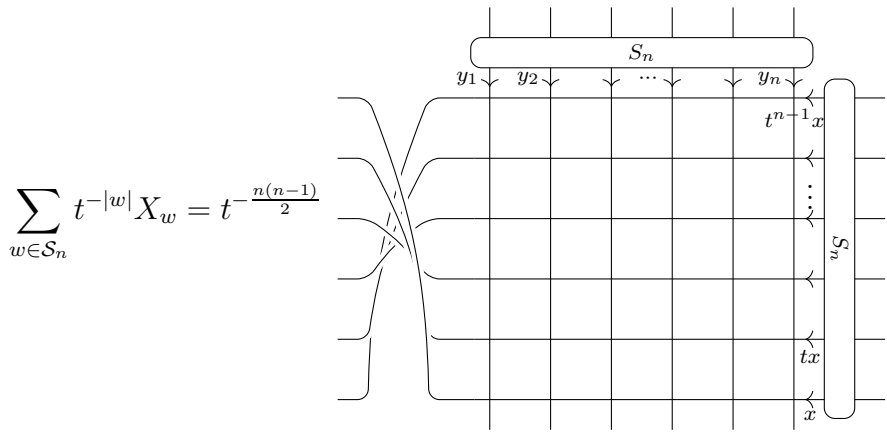
\[
\sum_{w\in\mathcal S_n} t^{-|w|}X_w
=t^{-\frac{n(n-1)}{2}}\ 
\begin{tikzpicture}[baseline=(current  bounding  box.center),rounded corners,scale=0.8]
\foreach\j/\lab in {6/y_n,5/,4/\cdots,3/,2/y_2,1/y_1} {
\draw[invarrow=0.97] (\j,0.5) -- node[pos=0.97,left] {$\ss \lab$} (\j,6.5);
\draw (\j,7) -- (\j,7.5);
}
\foreach\i/\k/\lab in {1/6/{t^{n-1}x\quad},2/5/\vdots,3/4/,4/3/,5/2/tx,6/1/x} {
\draw[draw=white,double=black,ultra thick, double distance=0.4pt] (0.75,7-\i) -- (0,7-\i) [bend right=10] to (-1,7-\k) -- (-1.5,7-\k);
\draw[invarrow=0.97] (0.75,7-\i) -- node[pos=0.96,below] {$\ss \lab$} (6.5,7-\i);
\draw (7,7-\i) -- (7.5,7-\i);
}
\node[draw=black,rectangle,minimum height=0.4cm,minimum width=4.5cm,inner sep=0pt,rotate=-90] at (6.75,3.5) {$\ss S_n$};
\node[draw=black,rectangle,minimum height=0.4cm,minimum width=4.5cm,inner sep=0pt] at (3.5,6.75) {$\ss S_n$};
\end{tikzpicture}
\]
We then impose $y_i=q x_i$ and note that up to normalization, the $S_n$ at the top is nothing but the $R$-matrix
associated to the longest element of $\mathcal S_n^{(2)}$; we can pull it to the bottom:
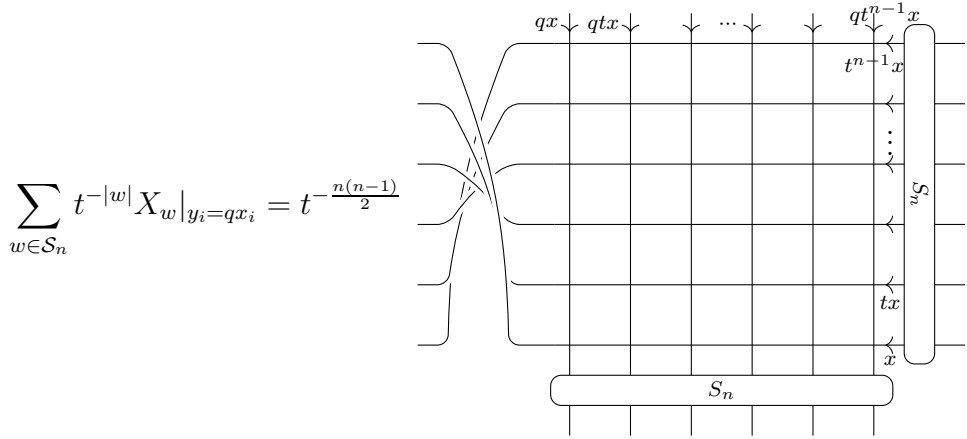
\[
\sum_{w\in\mathcal S_n} t^{-|w|}X_w|_{y_i=qx_i}
=t^{-\frac{n(n-1)}{2}}\ 
\begin{tikzpicture}[baseline=(current  bounding  box.center),rounded corners,scale=0.8]
\foreach\j/\lab in {6/{qt^{n-1}x\hspace{-0.8cm}\atop},5/,4/\cdots,3/,2/qtx,1/qx} {
\draw[invarrow=0.97] (\j,0.5) -- node[pos=0.97,left] {$\ss \lab$} (\j,6.5);
\draw (\j,-0.5) -- (\j,0);
}
\foreach\i/\k/\lab in {1/6/{t^{n-1}x\quad},2/5/\vdots,3/4/,4/3/,5/2/tx,6/1/x} {
\draw[draw=white,double=black,ultra thick, double distance=0.4pt] (0.75,7-\i) -- (0,7-\i) [bend right=10] to (-1,7-\k) -- (-1.5,7-\k);
\draw[invarrow=0.97] (0.75,7-\i) -- node[pos=0.96,below] {$\ss \lab$} (6.5,7-\i);
\draw (7,7-\i) -- (7.5,7-\i);
}
\node[draw=black,rectangle,minimum height=0.4cm,minimum width=4.5cm,inner sep=0pt,rotate=-90] at (6.75,3.5) {$\ss S_n$};
\node[draw=black,rectangle,minimum height=0.4cm,minimum width=4.5cm,inner sep=0pt] at (3.5,0.25) {$\ss S_n$};
\end{tikzpicture}
\]
Finally, the bracket is
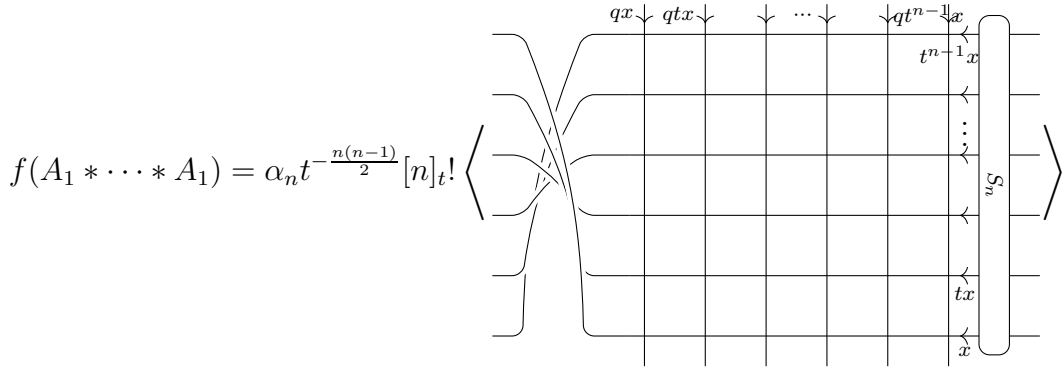
\[
f(A_1*\cdots*A_1)
=\alpha_n t^{-\frac{n(n-1)}{2}}[n]_t!
\left<
\begin{tikzpicture}[baseline=(current  bounding  box.center),rounded corners,scale=0.8]
\foreach\j/\lab in {6/{qt^{n-1}x\!\!\!\!\!\!\!},5/,4/\cdots,3/,2/qtx,1/qx} {
\draw[invarrow=0.97] (\j,0.5) -- node[pos=0.97,left] {$\ss \lab$} (\j,6.5);
}
\foreach\i/\k/\lab in {1/6/{t^{n-1}x\quad},2/5/\vdots,3/4/,4/3/,5/2/tx,6/1/x} {
\draw[draw=white,double=black,ultra thick, double distance=0.4pt] (0.75,7-\i) -- (0,7-\i) [bend right=10] to (-1,7-\k) -- (-1.5,7-\k);
\draw[invarrow=0.97] (0.75,7-\i) -- node[pos=0.96,below] {$\ss \lab$} (6.5,7-\i);
\draw (7,7-\i) -- (7.5,7-\i);
}
\node[draw=black,rectangle,minimum height=0.4cm,minimum width=4.5cm,inner sep=0pt,rotate=-90] at (6.75,3.5) {$\ss S_n$};
\end{tikzpicture}
\right>
\]
This is essentially the same situation that was considered at the end of \S \ref{sec:hecke}. First we expand each $R$-matrix in the upper-left triangle as 
$\check R_i(u)=u(1-t)+t(1-u) T_i^{-1}$ and note that the only term that has a nonzero identity coefficient is
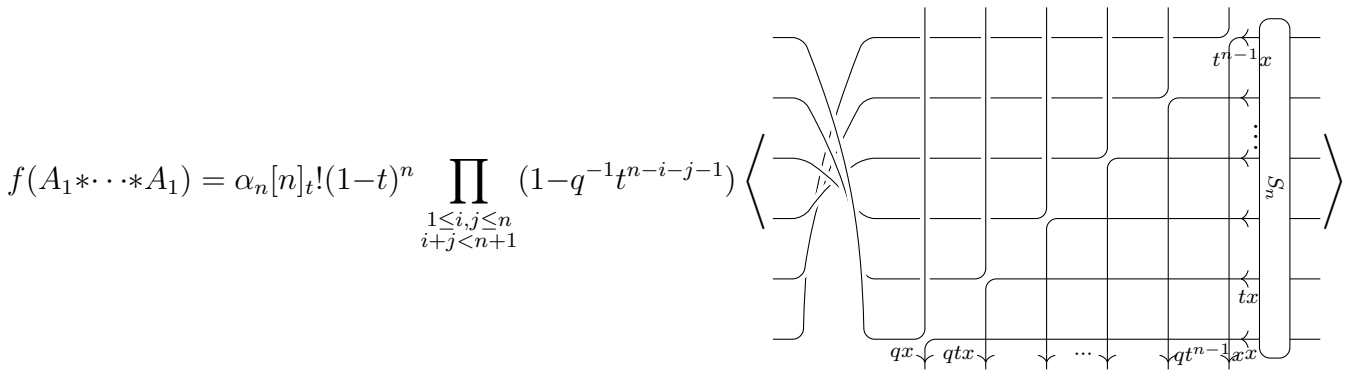
\[
f(A_1*\cdots*A_1)
=\alpha_n [n]_t!
(1-t)^n 
\prod_{\substack{1\le i,j\le n\\i+j<n+1}}
(1-q^{-1}t^{n-i-j-1})
\left<
\begin{tikzpicture}[baseline=(current  bounding  box.center),rounded corners,scale=0.8]
\foreach\i/\lab/\labb in {1/{t^{n-1}x\quad}/{qt^{n-1}x\!\!\!\!\!\!\!},2/\vdots/,3//\cdots,4/,5/tx/qtx,6/{\ x}/qx} {
\draw[draw=white,double=black,ultra thick, double distance=0.4pt] (7-\i,6.5) -- (7-\i,7-\i) -- (0,7-\i) [bend right=10] to (-1,\i) -- (-1.5,\i);
\draw[invarrow=0.97] (7-\i,0.75) -- (7-\i,7-\i) -- node[pos=0.96,below] {$\ss \lab$} (6.5,7-\i);
\draw (7,7-\i) -- (7.5,7-\i);
\draw[invarrow=1] (7-\i,0.5) -- (7-\i,0.75) node[left] {$\ss\labb$};
}
\node[draw=black,rectangle,minimum height=0.4cm,minimum width=4.5cm,inner sep=0pt,rotate=-90] at (6.75,3.5) {$\ss S_n$};
\end{tikzpicture}
\right>
\]

Secondly, we absorb all the crossings of the lower-right triangle into $S_n$,
resulting in a weight of $\prod_{\substack{1\le i,j\le n\\i+j>n+1}}
(1-q^{-1}t^{n-i-j+2})$. Recombining and simplifying, we obtain \eqref{eq:specA1}.
\end{proof}

We are now in a position to prove Theorem~\ref{thm:square}.
First, combining Lemma~\ref{lem:deg}, Proposition~\ref{prop:charsym} and Lemma~\ref{lem:wheel},
we find that $f(c)$, from its definition \eqref{eq:deff2}, is
$\alpha_n$ times an element of $\A_n$. $\alpha_n$ contains a factor of $(1-t)^{-n}$, but it is compensated by the fact that any lattice path entering from the left and exiting from the top must have a bend, resulting in a factor of $(1-t)^n$ according to \eqref{eq:wt} in the expression \eqref{eq:deff} of $f(c)$. Therefore, $f$ is a map from $\Z(\H_n)$ to $\A_n$.

%

Next, we identify $f(A_{\lambda_1}*\cdots*A_{\lambda_m})$.
We use the characterization of Proposition~\ref{prop:P_specs}, so that
$\epsilon_{\lambda_1,\ldots,\lambda_m}$ is entirely determined by Lemma~\ref{lem:epsrec} and \eqref{eq:epsB_0}.
On the other hand, it is easy to check that Proposition \ref{prop:nasty} implies the exact same relations for $V_n^{-1} f(A_{\lambda_1}*\cdots*A_{\lambda_m})$  (where $V_n$ is defined in \eqref{eq:defV}, and implements the isomorphism $V_n\times$ from $\mathcal A_n$ to $\A_n$). Therefore,
\[
f(A_{\lambda_1}*\cdots*A_{\lambda_m})
=V_n\,\epsilon_{\lambda_1,\ldots,\lambda_m}
\]

We finally compare $\Upsilon(f(A_{\lambda_1}*\cdots*A_{\lambda_m}))$ and $\Phi(A_{\lambda_1}*\cdots*A_{\lambda_m})$. 

According to the ring map property of $\Upsilon$ and \eqref{eq:Ueps} applied to $\epsilon_r=\epsilon_r^{(3)}$,
\begin{align*}
\Upsilon(f(A_{\lambda_1}*\cdots*A_{\lambda_m}))&=
\Upsilon(V_n\, \epsilon^{(3)}_{\lambda_1}*\cdots*\epsilon^{(3)}_{\lambda_m})
\\
&=\prod_{i=1}^m \Upsilon(V_{\lambda_i}\,\epsilon^{(3)}_{\lambda_i})
\\
&=
\prod_{i=1}^m \sigma_{q^{-1}}\sigma_t^{-1}(e_{\lambda_i})
\end{align*}

On the other hand, from \eqref{eq:PhiAn} and the ring map property of $\Phi$,
\[
\Phi(A_{\lambda_1}*\cdots*A_{\lambda_m})=\prod_{i=1}^m e_{\lambda_i}
\]
We have therefore checked the statement of Theorem \ref{thm:square} on a basis of $\Z$.
\begin{cor}
We have the following identity
\begin{equation}\label{eq:feqkappa}
f(1_n) = V_n\kappa_n
\end{equation}
\end{cor}
\begin{proof}
According to \eqref{eq:kappa_ev} and Proposition~\ref{prop:epsP}, $\Upsilon(V_n \kappa_n)=\Upsilon'(\kappa_n)=\sigma_q(e_n)$. On the other hand,
according to Proposition~\ref{prop:identsym}, $\Phi(1_n)=\tilde h_n=\sigma_t(h_n)=\sigma_t \sigma_\infty(e_n)$.
Applying Theorem~\ref{thm:square}, we conclude that (\ref{eq:feqkappa}) holds. 
\end{proof}
Using the definition of the shuffle product \eqref{eq:shuffle_prod} we can write
\begin{multline}
    \label{eq:kappa_rec2}
     \left(x_1+\dots + x_n \right) \kappa_n(x_1,\ldots,x_n)\\
    = \sum_{i=1}^n 
x_i  \kappa_{n-1}(x_1,\ldots,\hat{x_i},\ldots,x_n)
\prod_{\substack{j=1 \\ j\neq i}}^n \frac{(x_i-q_1 x_j)(x_i- q_2 x_j)
(x_i- q_3 x_j)}{(x_i-x_j)^3}
\end{multline}
The recurrence relation for $f(1_n)$ stated in  (\ref{eq:frec}) follows from \eqref{eq:feqkappa}, (\ref{eq:kappa_rec2}) and (\ref{eq:defV}) thus proving Theorem \ref{thm:f1n}.

\section{Application to the commuting scheme}\label{sec:commut}
\subsection{Proof of Theorem~\ref{thm:mainK}}\label{sec:proof_mainK}
In this section, all schemes are over $\CC$, as in the introduction.
\begin{prop}\label{prop:sasha}
Assuming Conjecture~\ref{conj:CM},
$K_n$ satisfies the following recurrence relation:
\begin{equation}\label{eq:sasha}
(x_1 + \dots + x_n)  K_n=
\sum_{i=1}^n 
x_i  K_{n-1}[\hat{x_i}]
\prod_{\substack{j=1 \\ j\neq i}}^n \frac{(1-q_1 x_j/x_i)(1- q_2 x_j/x_i)
(1- q_1 q_2 x_i/x_j)}{1-x_j/x_i}
\end{equation}
where $K_{n-1}[\hat{x_i}]=K_{n-1}(x_1,\dots,x_{i-1},x_{i+1},\dots, x_n)$. 
\end{prop}

\begin{proof}
For the purposes of this proof, we need more geometry than in the rest of the paper.
We think of the multi-graded Hilbert series as the pushforward to a point in localized equivariant $K$-theory. Equivalently, the $K$-polynomial of $\C_n$ is viewed as (the pushforward of) the $K_T$-class $[\mathcal O_{\C_n}]$ in the ambient space $\g_n\times\g_n$. Note that $K_T(\g_n\times\g_n)\cong K_T(pt)\cong \ZZ[q_1^{\pm},q_2^{\pm},x_1^\pm,\ldots,x_n^\pm]$ where the variables $q_1,q_2,x_1,\ldots,x_n$ are the {\em inverses}\/ of the equivariant parameters attached to the torus $T$.

Consider
\[
\F_n=\left\{
(v,A,B)\in \PP^{n-1}\times\g_n\times\g_n:
\ [A,B]=0 \text{ and $v$ eigenvector of $A$ and $B$}
\right\}
\]
\junk{is it obvious $\F_n$ is irreducible? yes, it's a fiber bundle over $\PP^{n-1}$ whose fiber we describe below}
There is a natural projective morphism
\begin{align*}
\F_n &\overset{p}{\to} \C_n 
\\
(v,A,B)&\mapsto(A,B)
\end{align*}
which is generically $n$-to-$1$ (choice of one eigenvector among $n$).
There is also the affine morphism $f: \F_n\to \PP^{n-1}$. Since all fibers are isomorphic by $GL_n$-action, the dimension of a fiber is $\dim \F_n - \dim \PP^{n-1} = n^2+1$.

The computation is in two parts:
\begin{itemize}
\item Computing $\pi_*[\mathcal O_{\F_n}(1)]$ by equivariant localization (which relates it to $\pi_*[\mathcal O_{\C_{n-1}}]$ using
$(v,A,B)\mapsto (A|_{\CC^n/v},B|_{\CC^n/v})$).
\item Comparing $p_*[\mathcal O_{\F_n}(1)]$ and $[\mathcal O_{\C_n}]$.
\end{itemize}

\subsubsection{Equivariant localization}\label{sec:equivloc}
Torus fixed points in $\PP^{n-1}$ are basis vectors $e_i$, $i=1,\ldots,n$.
Consider such a fixed point, which up to $\mathcal S_n$ action, we can choose to be $v=e_n$. The fiber is
\[
f^{-1}(v)=\{v\}\times X,\qquad X=\left\{\left(
A=\begin{pmatrix}
A'&0
\\
\alpha&a
\end{pmatrix}
,B=\begin{pmatrix}
B'&0
\\
\beta&b
\end{pmatrix}
\right):\ AB=BA\right\}
\]
More explicitly, the equations are
\begin{align}
A'B'&=B'A'
\\\label{eq:extra}
a\beta+\alpha B' &= b\alpha+\beta A'
\end{align}
In particular note that $(A',B')\in \C_{n-1}$. There are $2n$ extra variables $\alpha,\beta,a,b$ and $n-1$ extra equations, and
$\dim \C_{n-1}+2n-(n-1)=n^2+1=\dim f^{-1}(v)$, so assuming Conjecture~\ref{conj:CM} (Cohen--Macaulay property for $\C_{n-1}$, which implies it for $\C_{n-1}\times \mathbb A^{2n+2}$), the unmixedness theorem asserts that the extra equations \eqref{eq:extra} form a regular sequence. \rem[gray]{I don't know how to prove this without assuming $\C_{n-1}$ to be Cohen--Macaulay.
note that this potentially allows for an inductive proof of CM-ness.}

\rem[gray]{note that one can degenerate -- and then check that we haven't forgotten any extra eqs by using the result at the level of cohomology, which we can prove independently. but even degenerating the extra eqs all the way to $\beta_i A'_{i1}=0$ still requires the nontrivial statement that $A_{i1},\ldots,A_{in}$ form a regular sequence in the commuting scheme.}

We then conclude by standard exact sequence arguments that
\begin{equation}\label{eq:Krec}
[\mathcal O_X]=[\mathcal O_{\C_{n-1}}]\prod_{j=1}^{n-1}(1-q_1 x_j/x_n)(1-q_2 x_j/x_n)(1-q_1q_2 x_n/x_j)
\end{equation}
where $[\mathcal O_X]$ is the $K_T$ class of $\mathcal O_X$ in $\g_n\times \g_n$,
whereas $[\mathcal O_{\C_{n-1}}]$ is the $K_{T}$ class of $\C_{n-1}$ in $\g_{n-1}\times\g_{n-1}$ (where the $T$-action on $\g_{n-1}\times \g_{n-1}$ has a one-dimensional kernel leading to the obvious inclusion
$K_{T/\CC^\times}\subset K_T$, i.e., as a polynomial $[\mathcal O_{\C_{n-1}}]=K_{n-1}$ is independent of $x_n$).
The factors $1-q_1x_j/x_n$ (resp.\ $1-q_2x_j/x_n$) come from the zero column of $A$ (resp.\ $B$) whereas the factors $1-q_1q_2 x_n/x_j$ come from the equations \eqref{eq:extra}.

We have thus managed to compute $[f^{-1}(v)]=[v][\mathcal O_X]$.

The weights at the tangent space of $e_i$ are $x_i/x_j$, $j\ne i$ (recalling that the $x_i$ are the inverses of the equivariant parameters). Pulling back the relation 
\[
1=\sum_{i=1}^n \frac{[e_i]}{\displaystyle\prod_{\substack{j=1\\ j\ne i}}^{n}(1-x_j/x_i)}
\]
in localized equivariant $K$-theory of $\PP^{n-1}$ to $\F_n$ and multiplying by $[\mathcal O_{\F_n}(1)]$, we obtain
\[
[\mathcal O_{\F_n}(1)]
=
\sum_{i=1}^n
x_i[e_i]
[\mathcal O_{\C_{n-1}}](\hat x_i)
\prod_{\substack{j=1\\ j\ne i}}^{n}\frac{(1-q_2 x_j/x_i)(1-q_1 x_j/x_i)(1-q_1q_2 x_i/x_j)}{1-x_j/x_i}
\]
\junk{\[
p_*[\mathcal O_{\F_n}]
=
\sum_{i=1}^n
[\mathcal O_{\C_{n-1}}](\hat x_i)
\prod_{\substack{j=1\\ j\ne i}}^{n}\frac{(1-q_2 x_j/x_i)(1-q_1 x_j/x_i)(1-q_1q_2 x_i/x_j)}{1-x_j/x_i}
\]
this is almost but not quite the formula I want: there's a $x_i$ missing. which suggests we need to push forward the line bundle  $O(1)$ instead:}
Finally, pushing forward to $\g_n\times\g_n$ leads to:
\begin{equation}\label{eq:recone}
p_*[\mathcal O_{\F_n}(1)]
=
\sum_{i=1}^n x_i\,
[\mathcal O_{\C_{n-1}}](\hat x_i)
\prod_{\substack{j=1\\ j\ne i}}^{n}\frac{(1-q_2 x_j/x_i)(1-q_1 x_j/x_i)(1-q_1q_2 x_i/x_j)}{1-x_j/x_i}
\end{equation}

\subsubsection{Pushforward}

Now let us try to compute directly the pushforward of $[\mathcal O_{\F_n}(1)]$ under $p:\F_n\to\C_n$.
Write $\C_n=\spec R$ and $\F_n=\proj S$
where $S_0=R$. 
Since $\C_n$ is affine, the pushforward is simply given in terms of the sheaf cohomology groups of $\mathcal O_{\F_n}(1)$ (viewed as modules over $R$).
We claim the following
\[
H^i(\F_n,\mathcal O_{\F_n}(1))\cong
\begin{cases}
S_1&i=0
\\
0&i>0
\end{cases}
\]
\rem[gray]{Of course $H^0(\F_n,\mathcal O_{\F_n}(r))\equiv\Gamma(\mathcal O_{\F_n}(r))\supset S_r$ -- actually, there is a map $S_r\to H^0()$; why is it injective? but equality is nontrivial. in fact, as example above shows, it is wrong at $r=0$; all we know from general theory is that there is a $N$ such that it's true for all $r>N$.}

\rem[gray]{
General theory only tells us that $H^i(\F_n,\cdot)=0$ for $i>n-1$.
Still, one feels that the proof below is silly and should follow from a general theorem.
in particular $H^{>0}=0$ should follow, just as in the example above, from pushing forward to $\PP^{n-1}$ first (using say ex 8.2 p252 of Hartshorne) and identifying the resulting quasi-coherent sheaf as a direct sum of\dots $\mathcal O(k)$ hopefully?}

\begin{proof}
\rem[gray]{Let's try $n=2$. Similar to the above we have now two equations $ax^2+bxy+cy^2=a'x^2+b'xy+c'y^2=0$
as well as $cb'-bc'=ca'-ac'=ba'-ab'=0$. Defining $z=x/y$, we now have
$
P(z)-z Q(z^{-1})=\sum c_i(z) e_i(z)
$
where the $e_i(z)$ are the equations above written in terms of $z$, and the $c_i(z)$ are Laurent polynomials in $z$.
Assume $\deg P>1$. Since the equations are of degree at most $2$ in $z$, the coefficient of $z^2$ of the r.h.s.\ must be the coefficient of $\sum c_i^{pol}(z) e_i(z)$ where $c_i^{pol}(z)$ is {\em polynomial} in $z$. We can therefore subtract the latter from $P$, resulting in a polynomial of lower degree. Therefore we may assume $P$ of degree at most $1$, and the same for $Q$.
Surjectivity of $(P,Q)\mapsto P(z)-z Q(z^{-1})$ is obvious, showing $H^1=0$.
We can also push foward to $\PP^1$ first; without $[A,B]=0$ it should just be a vector bundle over $\PP^1$ isomorphic to $\mathcal O(-1)^{\oplus 4}$ and then we'd get the Sym of its dual, which tensored with any $\mathcal O(r\ge0)$ obviously has no $H^{>0}$. with $[A,B]=0$, ?
}

Let us first write the equations of $\F_n$ explicitly:
\begin{equation}\label{eq:quadrel}
AB=BA\qquad \text{minors of size $2$ of $(v,Av)$ and $(v,Bv)$}=0
\end{equation}

Now define the patches $U_i=\{v_i\ne0\}$.
Let us consider $U_1\cap U_2$ inside $U_1$, $U_2$. Denote by $z=v_1/v_2$, $r_i=v_i/v_2$, $i>2$ and collectively $r=(r_i)$. A section on $U_1$ is given by a polynomial $P(z,r,A,B)$ modulo the relations \eqref{eq:quadrel} rewritten in terms of $z,r,A,B$. Let us call these equations $q_i(z)$ where dependence on $r,A,B$ is suppressed. Note that the $q_i(z)$ are of degree less or equal to $2$ in $z$.
Similarly a section on $U_2$ is given by a polynomial $Q(z^{-1},rz^{-1},A,B)$ modulo the same equations expressed in terms of $z^{-1},rz^{-1},A,B$. The $\mathcal O(1)$ matching condition on $U_1\cap U_2$ is
\begin{equation}\label{eq:sheafcoho}
P(z)-z Q(z^{-1}) = \sum_i c_i(z) q_i(z)
\end{equation}
 where the $c_i(z)$ are Laurent polynomials in $z$.
 Now assume that $P$ is of degree greater than $1$ in $z$. Then its top coefficient is equal to the top coefficient of $Q:=\sum_i c^{pol}_i(z) q_i(z)$ where $c^{pol}_i(z)$ is the part of $c_i(z)$ with nonnegative powers of $z$. By subtracting $Q$ from $P$, we can reduce the degree of $P$ until it is at most $1$.

This way, one can show that the restriction of a global section to $U_i$ is of degree in $v_i/v_j$ at most $1$ for all $i\ne j$; this immediately implies that the $v_i$ generate the global sections of $\mathcal O(1)$.
Linear independence is obvious (it can be tested at the generic point).

\eqref{eq:sheafcoho} also implies that the map $(P,Q)\mapsto P(z)-z Q(z^{-1})$ is surjective, so that $H^1(\mathcal O_{\F_n}(1))=0$. Vanishing of higher cohomology follows similarly.
\end{proof}

Since $H^0(\mathcal O_{\F_n}(1))\cong S_1$ is a free $R$-module with basis $(v_1,\ldots,v_n)$ and higher sheaf cohomology vanishes, we conclude that
\begin{equation}
\label{eq:rectwo}
p_*[\mathcal O_{\F_n}(1)]
=
[\mathcal O_{\C_n}]\sum_{j=1}^n x_j
\end{equation}
where $x_j$ is the weight of the basis element $v_j$.

Combining \eqref{eq:recone} and \eqref{eq:rectwo} leads to the recurrence relation \eqref{eq:sasha}.
\end{proof}

We are now in a position to prove Theorem~\ref{thm:mainK}
by relating the shuffle element $\kappa_n$ and the $K$-polynomial $K_n$. $K_n$ is uniquely determined by \eqref{eq:sasha} and the initial condition $K_0=1$. Similarly $\kappa_n$ can be computed recursively with (\ref{eq:kappa_rec2}). By comparing (\ref{eq:kappa_rec2}) with \eqref{eq:sasha} (and checking the trivial initial condition $\kappa_0=K_0=1$) we conclude that
\begin{align}
    \label{eq:kappa_K}
        \kappa_n(x_1,\ldots,x_n) =
        \frac{t^{n(n-1)/2}}{V_n(x)}
        K_n(x_1,\dots, x_n)
\end{align}
Combining (\ref{eq:feqkappa}) with \eqref{eq:kappa_K} gives us the statement of Theorem~\ref{thm:mainK}.

\subsection{Proof of Theorem~\ref{thm:mainH}}
Theorem~\ref{thm:mainH} involves computing the multidegree of $\C_n$, which means one needs to work in equivariant cohomology.
The situation is much simpler than in equivariant $K$-theory (which was needed for Theorem~\ref{thm:mainK}), and we only sketch the corresponding geometric argument. By abuse of notation, we use the same symbols
$q_1,q_2,x_1,\ldots,x_n$ for the cohomology equivariant parameters. This is
consistent with our passage from $K$-polynomial to multidegree, which was defined in the introduction as the substitution $q_i\mapsto 1-q_i$ and $x_i\mapsto 1-x_i$ and then expanding at leading order. We refer to this limit in what follows as the GRR (Grothendieck--Riemann--Roch) limit.

View the flag variety $\mathcal B_n$ as the variety of Borel subalgebras of 
$\mathfrak{gl}_n$ and consider
$\tilde{\F}_n = \{(\mathfrak b,A,B)\in \mathcal B_n \times \mathfrak{gl}_n \times \mathfrak{gl}_n:\ [A,B]=0,\ A,B\in \mathfrak b\}$.
There is a projection map $p$ from $\tilde{\F}_n$ to $\C_n$
which is generically $n!$-to-$1$ (generic matrices are diagonalizable);
one then obtains by equivariant localization techniques (similar to \S\ref{sec:equivloc}) the formula
\begin{align}\notag
D_n &=\frac{1}{n!} \prod_{\substack{i,j=1\\i\ne j}}^n(q_1+x_i-x_j)(q_2+x_i-x_j)
\ \sum_{w\in\mathcal{S}_n}
w\left(
\prod_{1\le i<j\le n} \frac{q_1+q_2+x_i-x_j}{(x_j-x_i)(q_1+x_i-x_j)(q_2+x_i-x_j)}
\right)
\\\label{eq:mdegsym}
&=
\text{Sym}
\left(
\prod_{1\le i<j\le n} \frac{(q_1+q_2+x_i-x_j)(q_1-x_i+x_j)(q_2-x_i+x_j)}{x_j-x_i}
\right)
\end{align}
where the symmetric group acts by permuting variables, and $\text{Sym}$ denotes symmetrization w.r.t.\ to it.

One can think of this expression as the GRR limit of the shuffle element
$\frac{1}{n!}f(A_1*\cdots*A_1)$ (studied in \S \ref{sec:proof}). \footnote{It should be noted
that the element $f(A_1*\cdots*A_1)$ itself has a geometric interpretation
\cite{Ginzburg-isospectral} as the Hilbert series of the coordinate ring of the normalization of the isospectral commuting variety $\mathfrak{X}_{\mathrm{norm}}$.
What we are finding here is that to go from $K$-theory to cohomology, instead of using the structure sheaf of $\C_n$ one can use the combinatorially simpler structure sheaf of $\mathfrak X_{\mathrm{norm}}$.}

At $t=1$, from \eqref{eq:defZstar}, one has the identity in the center of the symmetric group algebras
$A_1*\cdots*A_1= n!\, 1_n$, so that in the GRR limit, $\frac{1}{n!}f(A_1*\cdots*A_1)$ and $f(1_n)$ coincide. As was already remarked in \S \ref{sec:intro-commut},
expanding the weights \eqref{eq:wtK}, which correspond to $f(1_n)$, in the GRR limit, result in the weights \eqref{eq:wtH}.
This leads to the formula of Theorem~\ref{thm:mainH}.

\rem[gray]{doesn't work in $K$-theory because there's no reason the structure sheaf is sent to the structure sheaf --
well, it isn't, it's not even 1-to-1 -- though at $n=2$ everything is OK}

A few comments are in order. In \cite{artic32,artic33} another expression for the multidegree of the commuting variety was obtained, using apparently unrelated methods.
There is actually a connection: consider Proposition~\ref{prop:exch}
and take the GRR limit.
One recovers one of the exchange relations of the work above (cf \cite[Eq.~(3.10)]{artic32} or \cite[Prop.~6]{artic33}). The expression for the multidegree of $\C_n$ there is essentially a repeated application of these identities; it is not obvious if it is more or less useful than our (new) formula \eqref{eq:mdegsym}.
(The extension of these ideas to $K$-theory will be investigated
elsewhere \cite{KZJ-progress}.)

The expression of Theorem~\ref{thm:mainH}, in contrast, has many desirable properties; for instance, it is division- and subtraction-free. Regarding positivity, let us only consider the bidegree with respect to scaling of the two matrices $A$ and $B$, i.e., set the $x_i$ to zero. On the one hand, \eqref{eq:mdegsym} specializes to the bidegree being the constant term of the Laurent polynomial
$\frac{1}{n!}\prod_{1\le i<j\le n}\frac{ (u_i-u_j)(q_1-u_i+u_j)(q_2-u_i+u_j)(q_1+q_2+u_i-u_j)}{u_iu_j}$, which clearly involves subtractions.
On the other hand, the expression in \eqref{eq:mainH} is a polynomial in $q_1$ and $q_2$ with manifestly positive coefficients. If we further restrict to the ordinary degree, i.e., set $q_1=q_2=1$, then we have the
simple formula
\begin{equation}\label{eq:degC}
\deg \C_n = \sum_{\substack{\text{lattice paths }P\\\text{on the }n\times n\text{ grid}\\\text{with identity connectivity}}} 2^{\#\{(i,j):\text{ bends}\}-n}
\end{equation}
\eqref{eq:degC} is also extremely effective computationally; here is
the degree of $\C_{12}$ (computed in a few seconds on a laptop):
\[
\deg \C_{12} = 1862632561783036151478238040096092649
\]
(note that this is beyond the currently known entries of
\href{https://oeis.org/A029729}{OEIS A029729}). 

\subsection{The $\mathcal S_3$ symmetry}\label{sec:bi}
As an obvious corollary of \eqref{eq:kappa_x} of Proposition~\ref{prop:kappa}, one has:
\begin{cor}
$\kappa_n$ possesses the $\mathcal S_3$ symmetry of permutations of $q_1,q_2,q_3$.
\end{cor}
In view of \eqref{eq:feqkappa}, the same holds for $f(1_n)$; this symmetry is however not obvious at the level of the partition function.

In a similar vein, assuming Conjecture~\ref{conj:CM}, this also implies
the symmetry property
\[
K_n(q_1,q_2)=K_n(q_2,q_1)=(q_1q_2^2)^{n(n-1)/2} K_n(q_1,(q_1q_2)^{-1})
\]
\begin{ex}Here are the coefficients of $K_3|_{x_i=1}$:
\begin{center}
\begin{tikzpicture}[x={(0.577cm,-1cm)},y={(0.577cm,1cm)},scale=0.87]
\foreach\i/\j/\k in
{0/0/1,1/1/-8,1/2/1,1/3/1,2/1/1,2/2/29,2/3/-16,2/4/1,3/1/1,3/2/-16,3/3/-27,3/4/29,3/5/-8,3/6/1,4/2/1,4/3/29,4/4/-16,4/5/1,5/3/-8,5/4/1,5/5/1,6/3/1}
\node at (\i,\j) {$\k$};
\draw[dotted,->] (0,0) -- (0,1) node[right] {$\ss q_1$};
\draw[dotted,->] (0,0) -- (1,0) node[right] {$\ss q_2$};
\end{tikzpicture}
\end{center}
\end{ex}

It is not a priori obvious why the $K$-polynomial of the commuting scheme $\C_n$ should possess the second symmetry property. Here we propose an ``explanation'':
\footnote{This is to be compared with the following two toroidal algebra representations: the one coming from the Hilbert scheme of points in $\CC^2$, which does not have manifest $\mathcal S_3$ symmetry, and the one coming from the Hilbert scheme of points in $\CC^3$, which does.}
\begin{conj}
Consider the triple commuting scheme
\[
\hat \C_n
:=
\{
(A,B,C)\in \g_n^3:\ [A,B]=[A,C]=[B,C]=0
\}
\]
and its grading w.r.t.\ scaling $(A,B,C)\mapsto (q_1 A,q_2 B,q_3 C)$ and conjugating by diagonal matrices.
Then its $K$-polynomial $\hat K_n$ satisfies
\[
\hat K_n(q_1,q_2,(q_1q_2)^{-1})
= K_n(q_1,q_2) K_n(q_1^{-1},q_2^{-1})
\]
\end{conj}

\rem{more generally, what's the behavior of our constructions/morphisms under permutation of $q_i$s? any special meaning to the invariant sector?}

\gdef\MRshorten#1 #2MRend{#1}%
\gdef\MRfirsttwo#1#2{\if#1M%
MR\else MR#1#2\fi}
\def\MRfix#1{\MRshorten\MRfirsttwo#1 MRend}
\renewcommand\MR[1]{\relax\ifhmode\unskip\spacefactor3000 \space\fi
\MRhref{\MRfix{#1}}{{\scriptsize \MRfix{#1}}}}
\renewcommand{\MRhref}[2]{%
\href{http://www.ams.org/mathscinet-getitem?mr=#1}{#2}}
\bibliographystyle{amsalphahyper}
\bibliography{biblio}

\newcommand{\etalchar}[1]{$^{#1}$}
\def\cprime{$'$} \def\cprime{$'$}
\providecommand{\bysame}{\leavevmode\hbox to3em{\hrulefill}\thinspace}
\providecommand{\MR}{\relax\ifhmode\unskip\space\fi MR }
\providecommand{\MRhref}[2]{%
  \href{http://www.ams.org/mathscinet-getitem?mr=#1}{#2}
}
\providecommand{\href}[2]{#2}
\begin{thebibliography}{FHH{\etalchar{+}}09}

\bibitem[BW18]{BW-coloured}
A.~Borodin and M.~Wheeler, \emph{Coloured stochastic vertex models and their
  spectral theory}, 2018,
  \href{http://arxiv.org/abs/1808.01866}{\path{arXiv:1808.01866}}.

\bibitem[DFZJ06]{artic32}
P.~Di~Francesco and P.~Zinn-Justin, \emph{Inhomogeneous model of crossing loops
  and multidegrees of some algebraic varieties}, Comm. Math. Phys. \textbf{262}
  (2006), no.~2, 459--487,
  \href{http://arxiv.org/abs/math-ph/0412031}{\path{arXiv:math-ph/0412031}}.
  \MR{MR2200268 (2006k:82068)}.

\bibitem[FHH{\etalchar{+}}09]{FHHSY}
B.~Feigin, K.~Hashizume, A.~Hoshino, J.~Shiraishi, and S.~Yanagida, \emph{A
  commutative algebra on degenerate {$\mathbb{CP}^1$} and {M}acdonald
  polynomials}, {J. Math. Phys.} \textbf{50} (2009), no.~9, 095215,
  \href{http://arxiv.org/abs/0904.2291}{\path{arXiv:0904.2291}}.

\bibitem[FO97]{FO}
B.~Feigin and A.~Odesskii, \emph{A family of elliptic algebras}, {Int. Math.
  Res. Not.} \textbf{1997} (1997), no.~11, 531--539.

\bibitem[FT11]{FT-shuffle}
B.~Feigin and A.~Tsymbaliuk, \emph{Equivariant {$K$}-theory of {H}ilbert
  schemes via shuffle algebra}, {Kyoto J. Mat.} \textbf{51} (2011), no.~4,
  831--854, \href{http://arxiv.org/abs/0904.1679}{\path{arXiv:0904.1679}}.

\bibitem[Gin11]{Ginzburg-isospectral}
V.~Ginzburg, \emph{Isospectral commuting variety, the {H}arish-{C}handra
  {$D$}-module, and principal nilpotent pairs}, 2011,
  \href{http://arxiv.org/abs/1108.5367}{\path{arXiv:1108.5367}}.

\bibitem[IO09]{IO-shuffle}
A.~P. Isaev and O.~V. Ogievetsky, \emph{Braids, shuffles and symmetrizers},
  Journal of Physics A: Mathematical and Theoretical \textbf{42} (2009),
  no.~30, 304017,
  \href{http://arxiv.org/abs/0812.3974}{\path{arXiv:0812.3974}},
  {\scriptsize\href{http://dx.doi.org/10.1088/1751-8113/42/30/304017}{\path{doi:10.1088/1751-8113/42/30/304017}}}.

\bibitem[Ize87]{Iz-6v}
A.~Izergin, \emph{Partition function of a six-vertex model in a finite volume},
  Dokl. Akad. Nauk SSSR \textbf{297} (1987), no.~2, 331--333. \MR{MR919260
  (88m:82040)}.

\bibitem[KBI97]{KBI}
V.~Korepin, N.~Bogoliubov, and A.~Izergin, \emph{Quantum inverse scattering
  method and correlation functions}, vol.~3, Cambridge University Press, 1997.

\bibitem[Kor82]{Kor}
V.~Korepin, \emph{Calculation of norms of {B}ethe wave functions}, Comm. Math.
  Phys. \textbf{86} (1982), no.~3, 391--418. \MR{MR677006 (83m:81078)}.

\bibitem[KZJ]{KZJ-progress}
A.~Knutson and P.~Zinn-Justin, work in progress.

\bibitem[KZJ07]{artic33}
\bysame, \emph{A scheme related to the {B}rauer loop model}, Adv. Math.
  \textbf{214} (2007), no.~1, 40--77,
  \href{http://arxiv.org/abs/math.AG/0503224}{\path{arXiv:math.AG/0503224}}.
  \MR{MR2348022 (2008j:14097)}.

\bibitem[Lam18]{JLam}
J.~Lamers, \emph{The functional method for the domain-wall partition function},
  SIGMA Symmetry Integrability Geom. Methods Appl. \textbf{14} (2018), paper
  64, 23, \href{http://arxiv.org/abs/1801.09635}{\path{arXiv:1801.09635}},
  {\scriptsize\href{http://dx.doi.org/10.3842/SIGMA.2018.064}{\path{doi:10.3842/SIGMA.2018.064}}}.

\bibitem[Las06]{Lascoux-Frob}
A.~Lascoux, \emph{The {H}ecke algebra and structure constants of the ring of
  symmetric polynomials}, 2006,
  \href{http://arxiv.org/abs/math/0602379}{\path{arXiv:math/0602379}}.

\bibitem[Mac98]{Macdonald2}
I.~Macdonald, \emph{Symmetric functions and {H}all polynomials}, Oxford
  University Press, 1998.

\bibitem[MT55]{MT-commut}
T.~S. Motzkin and O.~Taussky, \emph{Pairs of matrices with property {L}. {II}},
  Transactions of the American Mathematical Society \textbf{80} (1955), no.~2,
  387--401,
  {\scriptsize\href{http://dx.doi.org/10.2307/1992996}{\path{doi:10.2307/1992996}}}.

\bibitem[Neg14]{Ng}
A.~Negu{\c{t}}, \emph{The shuffle algebra revisited}, {Int. Math. Res. Not.}
  \textbf{2014} (2014), no.~22, 6242--6275,
  \href{http://arxiv.org/abs/1209.3349}{\path{arXiv:1209.3349}}.

\bibitem[Neg16]{Ng-Pieri}
\bysame, \emph{The $m/n$ {P}ieri rule}, Int. Math. Res. Not. \textbf{2016}
  (2016), no.~1, 219--257,
  \href{http://arxiv.org/abs/1407.5303.}{\path{arXiv:1407.5303.}}

\bibitem[Pop08]{Popov-commut}
V.~L. Popov, \emph{Irregular and singular loci of commuting varieties},
  Transformation Groups \textbf{13} (2008), 819--837,
  \href{http://arxiv.org/abs/0801.3074}{\path{arXiv:0801.3074}},
  {\scriptsize\href{http://dx.doi.org/10.1007/s00031-008-9018-9}{\path{doi:10.1007/s00031-008-9018-9}}}.

\bibitem[Ram97]{Ram-seminormal}
A.~Ram, \emph{Seminormal representations of {W}eyl groups and
  {I}wahori--{H}ecke algebras}, Proceedings of the London Mathematical Society
  \textbf{75} (1997), no.~1, 99--133,
  \href{http://arxiv.org/abs/math/9511223}{\path{arXiv:math/9511223}}.

\bibitem[Shi06]{Sh}
J.~Shiraishi, \emph{A family of integral transformations and basic
  hypergeometric series}, {Comm. Math. Phys.} \textbf{263} (2006), no.~2,
  439--460,
  \href{http://arxiv.org/abs/math/0501251}{\path{arXiv:math/0501251}}.

\bibitem[SV13]{SV-Hall}
O.~Schiffmann and E.~Vasserot, \emph{The elliptic {H}all algebra and the
  {$K$}-theory of the {H}ilbert scheme of {$\mathbb{A}^2$}}, {Duke Math. J.}
  \textbf{162} (2013), no.~2, 279--366,
  \href{http://arxiv.org/abs/0905.2555}{\path{arXiv:0905.2555}}.

\bibitem[WW15]{Wan-Frob}
J.~Wan and W.~Wang, \emph{Frobenius map for the centers of {H}ecke algebras},
  Trans. Amer. Math. Soc. \textbf{367} (2015), no.~8, 5507--5520,
  \href{http://arxiv.org/abs/1208.4446}{\path{arXiv:1208.4446}}.

\end{thebibliography}
\end{document}